\documentclass[11pt,reqno]{amsart}

\usepackage[dvipsnames]{xcolor}
\usepackage{amsmath,amssymb,amsthm,amsfonts,enumerate,tikz,bm}
\usepackage{mathrsfs}
\usetikzlibrary{matrix,arrows,positioning,calc}
\usepackage[all]{xy}
\usepackage[bookmarksnumbered,colorlinks]{hyperref}
\usepackage{url}
\numberwithin{equation}{section}
\usepackage{graphicx}
\usepackage[T1]{fontenc}
\usepackage{textcomp}
\usepackage{palatino,helvet}

\theoremstyle{definition}
\newtheorem{definition}{Definition}[section]
\newtheorem{remark}[definition]{Remark}
\newtheorem{example}[definition]{Example}

\theoremstyle{plain}
\newtheorem{theorem}[definition]{Theorem}
\newtheorem{proposition}[definition]{Proposition}
\newtheorem{lemma}[definition]{Lemma}
\newtheorem{corollary}[definition]{Corollary}

\theoremstyle{remark}

\usepackage{paralist}

\newcommand{\Hom}{\mathsf{Hom}}
\newcommand{\Ext}{\mathsf{Ext}}

\newcommand{\Ch}{\mathsf{Ch}}

\newcommand{\pd}{{\rm pd}}
\newcommand{\id}{{\rm id}}
\newcommand{\Gpd}{{\rm Gpd}}

\newcommand{\resdim}{{\rm resdim}}
\newcommand{\Prec}{{\mathrm{Prec}}}
\newcommand{\glGPD}{{\mathrm{gl.GPD}}}
\newcommand{\glGID}{{\mathrm{gl.GID}}}
\newcommand{\glDPD}{{\mathrm{gl.DPD}}}
\newcommand{\glDID}{{\mathrm{gl.DID}}}
\newcommand{\coresdim}{{\rm coresdim}}

\newcommand{\fmod}{\mathsf{mod}}
\newcommand{\Mod}{\mathsf{Mod}}

\newcommand{\Ker}{{\rm Ker}}
\newcommand{\Ima}{{\rm Im}}
\newcommand{\CoKer}{{\rm CoKer}}
\newcommand{\op}{{}^{\rm op}}

\newcommand{\A}{\mathcal{A}}
\newcommand{\B}{\mathcal{B}}
\newcommand{\C}{\mathcal{C}}

\newcommand*\circled[1]{\tikz[baseline=(char.base)]{ \node[shape=circle,draw,inner sep=2pt] (char) {#1};} }

\newcommand\numberthis{\addtocounter{equation}{1}\tag{\theequation}}

\makeatletter
\def\@seccntformat#1{%
  \protect\textup{\protect\@secnumfont
    \ifnum\pdfstrcmp{section}{#1}=0 \scshape\bfseries\fi
    \ifnum\pdfstrcmp{subsection}{#1}=0 \bfseries\fi
    \csname the#1\endcsname
    \protect\@secnumpunct
  }%
}
\makeatother

\begin{document}

\title{$\bm{n}$-Cotorsion pairs}

\author{Mindy Huerta}
\address[M. Huerta]{Instituto de Matem\'aticas. Universidad Nacional Aut\'onoma de M\'exico. Circuito Exterior, Ciudad Universitaria. CP04510. Mexico City, MEXICO}
\email{mindy@matem.unam.mx}

\author{Octavio Mendoza}
\address[O. Mendoza]{Instituto de Matem\'aticas. Universidad Nacional Aut\'onoma de M\'exico. Circuito Exterior, Ciudad Universitaria. CP04510. Mexico City, MEXICO}
\email{omendoza@matem.unam.mx}

\author{Marco A. P\'erez}
\address[M. A. P\'erez]{Instituto de Matem\'atica y Estad\'istica ``Prof. Ing. Rafael Laguardia''. Facultad de Ingenier\'ia. Universidad de la Rep\'ublica. CP11300. Montevideo, URUGUAY}
\email{mperez@fing.edu.uy}

\baselineskip=14pt
\maketitle

\begin{abstract}
Motivated by some properties satisfied by Gorenstein projective and Gorenstein injective modules over an Iwanaga-Gorenstein ring, we present the concept of left and right $n$-cotorsion pairs in an abelian category $\mathcal{C}$. Two classes $\mathcal{A}$ and $\mathcal{B}$ of objects of $\mathcal{C}$ form a left $n$-cotorsion pair $(\mathcal{A,B})$ in $\mathcal{C}$ if the orthogonality relation $\mathsf{Ext}^i_{\mathcal{C}}(\mathcal{A,B}) = 0$ is satisfied for indexes $1 \leq i \leq n$, and if every object of $\mathcal{C}$ has a resolution by objects in $\mathcal{A}$ whose syzygies have $\mathcal{B}$-resolution dimension at most $n-1$. This concept and its dual generalise the notion of complete cotorsion pairs, and has an appealing relation with left and right approximations, especially with those having the so called unique mapping property. 

The main purpose of this paper is to describe several properties of $n$-cotorsion pairs and to establish a relation with complete cotorsion pairs. We also give some applications in relative homological algebra, that will cover the study of approximations associated to Gorenstein projective, Gorenstein injective and Gorenstein flat modules and chain complexes, as well as $m$-cluster tilting subcategories. 
\end{abstract}


\pagestyle{myheadings}
\markboth{\rightline {\scriptsize M. Huerta, O. Mendoza and M. A. P\'{e}rez}}
         {\leftline{\scriptsize $n$-Cotorsion pairs}}


\section*{\textbf{Introduction}}

Given an $n$-Iwanaga-Gorenstein ring $R$, we know that if $\mathcal{GP}(R)$ denotes the class of Gorenstein projective left $R$-modules, and $\mathcal{P}(R)$ the class of projective left $R$-modules, using Auslan-der-Buchweitz approximation theory (see \cite{ABtheory,BMPS}, for instance), we can assert that every left $R$-module $M$ can be covered by an epimorphism $\varphi \colon P \twoheadrightarrow M$ with $P \in \mathcal{GP}(R)$ and whose kernel has projective dimension at most $n-1$. Moreover, the orthogonality relation $\mathsf{Ext}^i_R(\mathcal{GP}(R),\mathcal{P}(R)) = 0$ is satisfied for every $i \geq 1$. We are interested in considering the latter condition in more general contexts and only for indexes $1 \leq i \leq n$.  

In the present paper, we comprise the previous properties in the concept of \emph{left $n$-cotorsion pairs}. In the general setting provided by an abelian category $\mathcal{C}$, these will be defined by two classes $\mathcal{A}$ and $\mathcal{B}$ of objects of $\mathcal{C}$ such that: (1) $\mathcal{A}$ is closed under direct summands, (2) $\mathsf{Ext}^i_{\mathcal{C}}(\mathcal{A,B}) = 0$ for every $1 \leq i \leq n$, and (3) for every object $C \in \mathcal{C}$ there exists an exact sequence
\[
0 \to B_{n-1} \to B_{n-2} \to \cdots \to B_1 \to B_0 \to A \to C \to 0
\]
with $A \in \mathcal{A}$ and $B_k \in \mathcal{B}$ for every $0 \leq k \leq n-1$. 

This concept and its dual, that we shall call \emph{right $n$-cotorsion pair}, will represent an approach to what, roughly speaking, we may call \emph{higher cotorsion}: that is, the study of the possible outcomes of considering two classes of objects of $\mathcal{C}$ which are complete with respect to the orthogonality relation defined by the vanishing of the bifunctor $\Ext^i_{\mathcal{C}}(-,-)$ for ``higher'' indexes $i > 1$. The case $i = 1$, on the other hand, is already covered by the theory of complete cotorsion pairs, widely considered in fields such as relative homological algebra or representation theory.  

The present paper is organised as follows. In the first section we give some preliminaries on homological dimensions, orthogonality and approximations. The next section is devoted to present the concept of left and right $n$-cotorsion pairs and its relation with complete cotorsion pairs. In Proposition~\ref{prop:cotorsion_vs_ncotorsion} and Theorem~\ref{theo:left-n-cotorsion}, we give necessary and sufficient conditions for an $n$-cotorsion pair $(\mathcal{A,B})$ in $\mathcal{C}$ to form a complete cotorsion pair $(\mathcal{A},\mathcal{B}^\wedge_{n-1})$, where $\mathcal{B}^\wedge_{n-1}$ will denote the class of objects of $\mathcal{C}$ with $\mathcal{B}$-resolution dimension $\leq n-1$. 

In the third section we study how to construct covers and envelopes from $n$-cotorsion pairs $(\mathcal{A,B})$ in $\mathcal{C}$. We also define new type of approximations, that we call \emph{special $(\mathcal{A},k,\mathcal{B})$-precovers and preenvelopes}. We shall consider the class of objects in $\mathcal{C}$ having such approximations, and analyse some conditions under which this class is closed under extensions (see Corollary~\ref{Coro-k-prec}). Moreover, given an $n$-cotorsion pair $(\mathcal{A,B})$ in $\mathcal{C}$, we give in Corollaries~\ref{Aperp y Bvee} and \ref{corUMP2} some necessary and sufficient conditions to obtain precovers and envelopes constructed from $\mathcal{A}$ and $\mathcal{B}$ that satisfy the unique mapping property. At this point, we shall make some comparisons with other approaches to higher cotorsion, like for instance the remarkable study \cite{CriveiTorrecillas} by S. Crivei and B. Torrecillas, where the authors establish several equivalent conditions for a class $\mathcal{A} \subseteq \mathcal{C}$ under which every object in $\mathcal{C}$ has an epic $\mathcal{A}$-envelope and a monic $\mathcal{A}$-cover. These conditions have to do with the concept of $(m,n)$-cotorsion pairs, $n$-special precovers and $m$-special preenvelopes. 

Section 4 is devoted to explain what does it mean for a left or right $n$-cotorsion pair to be \emph{hereditary}. We shall see in Proposition~\ref{equiv hered y n-cot} that a (left and right) hereditary $n$-cotorsion pair coincides with the usual concept of hereditary complete cotorsion pair. Thus, we propose a notion for being hereditary that is not trivial for either left or right higher cotorsion pairs. 

In Section 5 we present applications and examples of the theory of $n$-cotorsion pairs, developed in Sections \ref{sec:approximations} and \ref{sec:hereditary}, in the context of relative Gorenstein homological algebra and cluster-tilting subcategories. We shall see in Example~\ref{ex:GProj_ncot} and Proposition~\ref{prop:ncot_spli_silp} that the classes $\mathcal{GP}(R)$ and $\mathcal{P}(R)$ of Gorenstein projective and projective $R$-modules will form a left $n$-cotorsion pair provided that $R$ is an $n$-Iwanaga-Gorenstein ring or a Gorenstein ring (in the sense of \cite{BR}). Moreover, we give characterisations of Gorenstein rings in terms of the pair $(\mathcal{GP}(R),\mathcal{P}(R))$ and its dual $(\mathcal{I}(R),\mathcal{GI}(R))$, formed by the classes of injective and Gorenstein injective $R$-modules. As an application in this setting, we shall prove for example that every module over $2$-Iwanaga-Gorenstein ring has a Gorenstein injective cover with the unique mapping property, and that the existence of Gorenstein projective envelopes implies the existence of such envelopes with the unique mapping property (see Corollaries~\ref{coro:GI_unique_mapping} and \ref{coro:GP_unique_mapping}). An analogous study is done for the notions of Ding projective and Ding injective modules over a ring, but in addition we find some finiteness conditions for the global Ding projective and Ding injective dimensions of a ring. Later, we study some consequences of having the classes $\mathcal{F}(R)$ and $\mathcal{GF}(R)$ of flat and Gorenstein flat $R$-modules as halves of left and right $n$-cotorsion pairs. This will lead for instance to some characterisations of left perfect rings with null global Gorenstein flat dimension (Proposition~\ref{prop:perfect_global_flat}), and of left perfect rings that are also quasi-Frobenius (Proposition~\ref{prop:perfect_QF}). Another interesting fact about pairs of the form $(\mathcal{GF}(R),\mathcal{F}(R))$ is its relation with the pair $(\mathcal{I}(R),\mathcal{GI}(R))$, mentioned before, in terms of the Pontryagin duality functor $M \mapsto M^+ := \mathsf{Hom}_{\mathbb{Z}}(M,\mathbb{Q / Z})$ (see Theorems \ref{theo:GFGI_Pontryagin} and \ref{theo:GFGI_Pontryagin}). Besides its applications in Gorenstein homological algebra, we also study the interplay between the $n$-cotorsion pairs and cluster tilting subcategories in the sense of Iyama \cite{IyamaCluster}. For an abelian category $\mathcal{C}$ with enough projective and injective objects, we shall give a one-to-one correspondence between $n$-cotorsion pairs in $\mathcal{C}$ of the form $(\mathcal{D,D})$ and $(n+1)$-cluster tilting subcategories of $\mathcal{C}$ (see Proposition~\ref{nclustA=B y ff} and Theorem~\ref{ncot y ct}). 

In the last section we show how to induce certain left and right $n$-cotorsion pairs of chain complexes from a given $n$-cotorsion pair $(\mathcal{A,B})$ in an abelian category $\mathcal{C}$. These induced pairs will involve the classes $\widetilde{\mathcal{A}}$ of $\mathcal{A}$-complexes, $\widetilde{\mathcal{B}}$ of $\mathcal{B}$-complexes, and ${\rm dg}\widetilde{\mathcal{A}}$ and ${\rm dg}\widetilde{\mathcal{B}}$ of differential graded complexes of objects in $\mathcal{A}$ and $\mathcal{B}$. The results presented in this section are motivated by the works of J. Gillespie \cite{GillespieFlat}, and X. Yang and N. Ding \cite{YangDingQuestion}, where they show that every complete and hereditary cotorsion pair $(\mathcal{A,B})$ gives rise to two complete cotorsion pairs of complexes of the form $(\widetilde{\mathcal{A}},{\rm dg}\widetilde{\mathcal{B}})$ and $({\rm dg}\widetilde{\mathcal{A}},\widetilde{\mathcal{B}})$. We also prove that if any of these pairs is a left or a right $n$-cotorsion pair of complexes, then so is $(\mathcal{A,B})$ in $\mathcal{C}$, provided that $\mathsf{Ext}^i_{\mathcal{C}}(\mathcal{A,B}) = 0$ for every $1 \leq i \leq n+1$, extending thus an important result in \cite{YangDingQuestion}.


\section{\textbf{Preliminaries}}\label{sec:preliminaries}

Let us recall some categorical and homological preliminaries that will be used in the sequel. 

Throughout this paper, $\mathcal{C}$ will denote an abelian category (not necessarily with enough projective or injective objects). The main example of such category considered here will be the category $\Mod(R)$ of left $R$-modules and $R$-homomorphisms, where $R$ is an associative ring with identity. By a module $M$ we shall mean a left $R$-module unless otherwise specified. Right $R$-modules will be regarded as left modules over the opposite ring $R^{\rm op}$. We shall also consider the category $\Ch(R)$ of complexes of (left) $R$-modules, and the category $\mathsf{mod}(\Lambda)$ of finitely generated modules over an Artin algebra $\Lambda$. 

Every subcategory of $\mathcal{C}$ is assumed to be full, and so any class $\mathcal{A} \subseteq \mathcal{C}$ of objects of $\mathcal{C}$ may be regarded as a (full) subcategory of $\mathcal{C}$. If two objects $C$ and $D$ in $\mathcal{C}$ are isomorphic, we write $C \simeq D$. The notation $F \cong G$ will be reserved to denote the existence of a natural isomorphism between two functors $F$ and $G$. Monomorphisms and epimorphisms in $\mathcal{C}$ may sometimes be denoted using arrows $\rightarrowtail$ and $\twoheadrightarrow$, respectively.

The results presented in this paper have their corresponding dual version, which sometimes will be omitted for simplicity.


\subsection*{Resolution and coresolution dimension} 

Let $\mathcal{B} \subseteq \mathcal{C}$ be a class of objects of $\mathcal{C}$. Given an object $C \in \mathcal{C}$ and a nonnegative integer $m \geq 0$, a \emph{$\mathcal{B}$-resolution of $C$ of length $m$} is an exact sequence
\[
0 \to B_m \to B_{m-1} \to \cdots \to B_1 \to B_0 \to C \to 0
\]
in $\mathcal{C},$ where $B_k \in \mathcal{B}$ for every integer $0 \leq k \leq m$. The \emph{resolution dimension of $C$ with respect to $\mathcal{B}$} (or the \emph{$\mathcal{B}$-resolution dimension} of $C$), denoted $\resdim_{\mathcal{B}}(C)$, is defined as the smallest nonnegative integer $m \geq 0$ such that $C$ has a $\mathcal{B}$-resolution of length $m$. If such $m$ does not exist, we set $\resdim_{\mathcal{B}}(C) := \infty$. Dually, we have the concepts of \emph{$\mathcal{B}$-coresolutions of $C$ of length $m$} and of  \emph{coresolution dimension of $C$ with respect to $\mathcal{B}$}, denoted by $\coresdim_{\mathcal{B}}(C)$.    

With respect to these two homological dimensions, we shall frequently consider the following two classes of objects in $\mathcal{C}$:
\begin{align*}
\mathcal{B}^\wedge_m & := \{ C \in \mathcal{C} \mbox{ : } \resdim_{\mathcal{B}}(C) \leq m \}, \\ 
\mathcal{B}^\vee_m & := \{ C \in \mathcal{C} \mbox{ : } \coresdim_{\mathcal{B}}(C) \leq m \}.
\end{align*}


\subsection*{Orthogonality with respect to extension functors}

In any abelian category $\mathcal{C}$, we can define extension bifunctors $\Ext^i_{\mathcal{C}}(-,-)$ in the sense of Yoneda. See for instance \cite{Sieg} for a detailed treatise on this matter. Recall that $\Ext^1_{\mathcal{C}}(X,Y)$ is defined as the abelian group formed by classes of short exact sequences $0 \to Y \to Z \to X \to 0$ under certain equivalence relation. In case we work in the category $\mathsf{Ch}(\mathcal{C})$ of complexes in $\mathcal{C}$, we shall write $\Ext^i_{\mathsf{Ch}(\mathcal{C})}(-,-)$ as $\mathsf{Ext}^i_{\mathsf{Ch}}(-,-)$ for simplicity. 

Given two classes of objects $\mathcal{A,B} \subseteq \mathcal{C}$ and an integer $i \geq 1$, the notation $\Ext^i_{\mathcal{C}}(\mathcal{A,B}) = 0$ will mean that $\Ext^i_{\mathcal{C}}(A,B) = 0$ for every $A \in \mathcal{A}$ and $B \in \mathcal{B}$. In the case where $\mathcal{A} = \{ M \}$ or $\mathcal{B} = \{ N \}$, we shall write $\Ext^i_{\mathcal{C}}(M,\mathcal{B}) = 0$ and $\Ext^i_{\mathcal{C}}(\mathcal{A},N) = 0$, respectively. 

Recall that the \emph{right $i$-th orthogonal complement} of $\mathcal{A}$ is defined by
\[
\mathcal{A}^{\perp_i} := \{ N \in \mathcal{C} \mbox{ : } \Ext^i_{\mathcal{C}}(\mathcal{A},N) = 0 \},
\]
and the \emph{total right orthogonal complement} of $\mathcal{A}$ by 
\[
\mathcal{A}^\perp := \bigcap_{i \geq 1} \mathcal{A}^{\perp_i}.
\]
Dually, we have the \emph{$i$-th and total left orthogonal complements} ${}^{\perp_i}\mathcal{B}$ and ${}^{\perp}\mathcal{B},$ respectively.


\subsection*{Approximations}

Let $\mathcal{A}$ be a class of objects of $\mathcal{C}$. A morphism $f \colon A \to C$ is said to be an \emph{$\mathcal{A}$-precover} (or a \emph{right $\mathcal{A}$-approximation}) \emph{of $C$} if $A \in \mathcal{A}$ and if for every morphism $f' \colon A' \to C$ with $A' \in \mathcal{A},$ there exists a morphism $h \colon A' \to A$ such that $f' = f \circ h$. If in addition, in the case $A' = A$ and $f' = f$, the previous equality can only be completed by automorphisms $h$ of $A$, then $f$ is called an \emph{$\mathcal{A}$-cover} (or a \emph{minimal right $\mathcal{A}$-approximation}). Furthermore, an $\mathcal{A}$-precover $f \colon A \to C$ of $C$ is \emph{special} if $\CoKer(f) = 0$ and $\Ker(f) \in \mathcal{A}^{\perp_1}$. The class $\mathcal{A}$ is said to be \emph{precovering} if every object of $\mathcal{C}$ has an $\mathcal{A}$-precover. Similarly, we have the concepts of \emph{covering} and \emph{special precovering} classes in $\mathcal{C}$.

Dually, we have the notions of \emph{$\mathcal{A}$-preenvelopes} (\emph{left $\mathcal{A}$-approximations}), \emph{$\mathcal{A}$-envelopes} (\emph{minimal left $\mathcal{A}$-approximations}) and \emph{special $\mathcal{A}$-preenvelopes} in $\mathcal{C}$, along with the corresponding notions of \emph{preenveloping}, \emph{enveloping} and \emph{special preenveloping} classes. 

With these preliminaries in hand, we are ready to begin our approach to higher cotorsion in abelian categories.


\section{\textbf{{\textit n}-Cotorsion pairs}}\label{sec:ncotorsion}

The notion of cotorsion pair was first introduced by L. Salce in \cite{Salce}. It is the analog of a torsion pair where the bifunctor $\Hom_{\mathcal{C}}(-,-)$ is replaced by $\Ext^1_{\mathcal{C}}(-,-)$. Roughly speaking, two classes $\mathcal{A}$ and $\mathcal{B}$ of objects in an abelian category $\mathcal{C}$ form a cotorsion pair $(\mathcal{A,B})$ if they are complete with respect to the orthogonality relation defined by the vanishing of the functor $\Ext^1_{\mathcal{C}}(-,-)$. Specifically, and for the purpose of this paper, it comes handy to recall this concept as follows.

\begin{definition}\label{def:cotorsion_pair}
Let $\mathcal{A}$ and $\mathcal{B}$ be two classes of objects in $\mathcal{C}$. We say that $\mathcal{A}$ and $\mathcal{B}$ form a \textbf{complete left cotorsion pair} $(\mathcal{A,B})$ in $\mathcal{C}$ if $\mathcal{A} = {}^{\perp_1}\mathcal{B}$ and if every object of $\mathcal{C}$ has an epic $\mathcal{A}$-precover with kernel in $\mathcal{B}$. Dually, we have the concept of \textbf{complete right cotorsion pair} in $\mathcal{C}$. 
\end{definition}

Note that $(\mathcal{A,B})$ is a complete cotorsion pair in $\mathcal{C}$ if, and only if, it is both a complete left and right cotorsion pair in $\mathcal{C}$. 

Motivated by the properties of Gorenstein projective and Gorenstein injective modules over Iwanaga-Gorenstein rings mentioned in the introduction, below we present a ``higher'' version of cotorsion pairs, which will cover complete left and right cotorsion pairs in the sense of Definition~\ref{def:cotorsion_pair}, as particular cases. By ``higher'' we mean that orthogonality with respect to $\Ext^i_{\mathcal{C}}(-,-)$ will be considered for indices $i \geq 1$. We shall also see how some well known properties of cotorsion pairs are transferred to the higher context resulting from Definition~\ref{def:ncotorsion} below.  

Throughout, $n > 0$ will be a positive integer.

\begin{definition}\label{def:ncotorsion} 
Let $\mathcal{A}$ and $\mathcal{B}$ be two classes of objects in $\mathcal{C}$. We say that $(\mathcal{A,B})$ is a \textbf{left $\bm{n}$-cotorsion pair} in $\mathcal{C}$ if the following conditions are satisfied:
\begin{enumerate}
\item $\mathcal{A}$ is closed under direct summands.

\item $\Ext^i_{\mathcal{C}}(\mathcal{A,B}) = 0$ for every $1 \leq i \leq n$.

\item For every object $C \in \mathcal{C}$, there exists a short exact sequence
\[
0 \to K \to A \to C \to 0
\]
where $A \in \mathcal{A}$ and $K \in \mathcal{B}^{\wedge}_{n-1}$.
\end{enumerate}

Dually, we say that $(\mathcal{A,B})$ is a \textbf{right $\bm{n}$-cotorsion pair} in $\mathcal{C}$ if condition (2) above is satisfied, with $\mathcal{B}$ closed under direct summands, and if every object of $C$ can be embedded into an object of $\mathcal{B}$ with cokernel in $\mathcal{A}^\vee_{n-1}$. Finally, $\mathcal{A}$ and $\mathcal{B}$ form a \textbf{$\bm{n}$-cotorsion pair} $(\mathcal{A,B})$ in $\mathcal{C}$ if $(\mathcal{A,B})$ is both a left and right $n$-cotorsion pair in $\mathcal{C}$.
\end{definition}

\begin{example}\label{ex:trivial}
In what follows, let us denote by $\mathcal{P}(\mathcal{C})$ and $\mathcal{I}(\mathcal{C})$ the classes of projective and injective objects of $\mathcal{C}$, respectively. It is clear that $\mathcal{C}$ has enough projectives (resp., enough injectives) if, and only if, $(\mathcal{P}(\mathcal{C}),\mathcal{C})$ (resp., $(\mathcal{C},\mathcal{I}(\mathcal{C}))$) is an $n$-cotorsion pair in $\mathcal{C}$ for every $n \geq 1$. 

In what follows, we shall call $(\mathcal{P}(\mathcal{C}),\mathcal{C})$ and $(\mathcal{C},\mathcal{I}(\mathcal{C}))$ the \textbf{trivial $\bm{n}$-cotorsion pairs} in the case where $\mathcal{C}$ has enough projectives and injectives. Some nontrivial examples will be presented later on in Section~\ref{sec:applications}. 
\end{example}


\subsection*{\textbf{Relations between cotorsion and higher cotorsion}}

It is clear that left (resp., right) $1$-cotorsion pairs coincide with complete left (resp., right) cotorsion pairs in $\mathcal{C}$. However, we can say more on how (left and right) $n$-cotorsion pairs interact with the concept of complete cotorsion pairs. Specifically, we shall study under which conditions a complete left cotorsion pair induces a left $n$-cotorsion pair. Conversely, we shall prove that every left $n$-cotorsion pair induces a complete left cotorsion pair. 

Let us begin establishing certain conditions under which two classes of objects $\mathcal{A}$ and $\mathcal{B}$ form a left $n$-cotorsion pair in $\mathcal{C}$. The following lemma can be deduced from a standard dimension shifting argument.

\begin{lemma} \label{lema1}
For any class $\mathcal{B}$ of objects of $\mathcal{C}$, the following containment holds:
\[
\bigcap^n_{i = 1} {}^{\perp_i}\mathcal{B} \subseteq {}^{\perp_1}\mathcal{B}^\wedge_{n-1}.
\]
\end{lemma}

\begin{proposition}\label{prop:cotorsion_vs_ncotorsion}
Let $\mathcal{C}$ be an abelian category with enough injectives, and let $\mathcal{A}$ and $\mathcal{B}$ be two classes of objects of $\mathcal{C}$ such that $\mathcal{I}(\mathcal{C}) \subseteq \mathcal{B}$. Then, $\Ext^1_{\mathcal{C}}(\mathcal{A}, \mathcal{B}^\wedge_{n-1}) = 0$ if, and only if, $\Ext^i_{\mathcal{C}}(\mathcal{A,B}) = 0$ for every $1 \leq i \leq n$. In particular, $(\mathcal{A}, \mathcal{B}^\wedge_{n-1})$ is a complete left cotorsion pair in $\mathcal{C}$ if, and only if, $(\mathcal{A,B})$ is a left $n$-cotorsion pair in $\mathcal{C}$. 
\end{proposition}
\begin{proof} 
The ``if'' part follows from Lemma~\ref{lema1}. In order to show the ``only if'' statement, note that since $\mathcal{C}$ has enough injectives and $\mathcal{I}(\mathcal{C}) \subseteq \mathcal{B},$  for every injective $(i-1)$-cosyzygy $K$ of $B \in \mathcal{B}$ we have that $\resdim_{\mathcal{B}}(K) \leq i - 1 \leq n - 1$ with $1 \leq i \leq n$. Then, $\Ext^1_{\mathcal{C}}(\mathcal{A}, K) = 0$ since $\Ext^1_{\mathcal{C}}(\mathcal{A}, \mathcal{B}^\wedge_{n-1}) = 0$. Therefore, we have that $\Ext^i_{\mathcal{C}}(\mathcal{A},B) \cong \Ext^1_{\mathcal{C}}(\mathcal{A},K) = 0$ for every $B \in \mathcal{B}$ and $1 \leq i \leq n$. 
\end{proof}

In the ``if'' part of the previous proposition, we actually do not need that $\mathcal{C}$ has enough injectives or $\mathcal{I}(\mathcal{C}) \subseteq \mathcal{B}$ either. As a matter of fact, we only require a complete left cotorsion pair of the form $(\mathcal{A},\mathcal{B}^\wedge_{n-1})$. Before showing this in Theorem~\ref{theo:left-n-cotorsion}, let us state and prove the following properties derived from the orthogonality relations $\Ext^i_{\mathcal{C}}(\mathcal{A,B}) = 0$ with $1\leq i\leq n.$

\begin{proposition}\label{prop8}
Let $\mathcal{A}$ and $\mathcal{B}$ be two classes of objects of $\mathcal{C}$ satisfying $\Ext^i_{\mathcal{C}}(\mathcal{A,B}) = 0$ for every $1 \leq i \leq n$. If $Y \in \mathcal{B}^\wedge_k$ with $0 \leq k \leq n-1$, then $\Ext^i_{\mathcal{C}}(\mathcal{A},Y) = 0$ for every $1 \leq i \leq n - k$. In particular, $\Ext^1_{\mathcal{C}}(\mathcal{A},\mathcal{B}^\wedge_{n-1}) = 0$.
\end{proposition}  
\begin{proof}
Note that the case $n = 1$ is clear. Thus, we may assume that $n \geq 2$. We use induction on $k$. The case $k = 0$ is also clear, so we may take $1 \leq k \leq n-1$ for $n \geq 2$. 

Let $A \in \mathcal{A}$ and $Y \in \mathcal{B}^\wedge_k$. First, for the case $k = 1$, we have that $\resdim_{\mathcal{B}}(Y) \leq 1$, and thus there is an exact sequence
\[
0 \to B_1 \to B_0 \to Y \to 0
\]
with $B_0, B_1 \in \mathcal{B}$. Then, we obtain an exact sequence
\[
\Ext^i_{\mathcal{C}}(A,B_0) \to \Ext^i_{\mathcal{C}}(A,Y) \to \Ext^{i+1}_{\mathcal{C}}(A,B_1)
\]
of abelian groups with $\Ext^i_{\mathcal{C}}(A,B_0) = 0$ and $\Ext^{i+1}_{\mathcal{C}}(A,B_1) = 0$ if $1 \leq i \leq n-1$. Hence, $\Ext^i_{\mathcal{C}}(A,Y) = 0$ for every $A \in \mathcal{A}$, $Y \in \mathcal{B}^\wedge_1$ and $1 \leq i \leq n-1$.

Now for the successor case, suppose that for every object $Y' \in \mathcal{B}^\wedge_k$ with $1 \leq k < n-1$ (the case $k = n - 1$ follows from Lemma~\ref{lema1}), we have that $\Ext^i_{\mathcal{C}}(A,Y') = 0$ for every $1 \leq i \leq n - k$. Now let $Y \in \mathcal{C}$ be an object with $\resdim_{\mathcal{B}}(Y) \leq k + 1$, so that there is an exact sequence 
\[
0 \to Y' \to B \to Y \to 0
\] 
with $B \in \mathcal{B}$ and $\resdim_{\mathcal{B}}(Y') \leq k$. Consider an integer $1 \leq i \leq n - (k + 1)$. Then, we have an exact sequence 
\[
\Ext^i_{\mathcal{C}}(A,B) \to \Ext^i_{\mathcal{C}}(A,Y) \to \Ext^{i+1}_{\mathcal{C}}(A,Y')
\] 
of abelian groups where $\Ext^i_{\mathcal{C}}(A, B) = 0.$  Since $1 \leq i \leq n - (k + 1)$ and  $\Ext^{i+1}_{\mathcal{C}}(A, Y') = 0$ by the induction hypothesis, we get that $\Ext^i_{\mathcal{C}}(A,Y) = 0$ for every $1 \leq i \leq n - (k + 1)$. 
\end{proof}

\begin{theorem}\label{theo:left-n-cotorsion} 
Let $\mathcal{A}$ and $\mathcal{B}$ be two classes of objects in $\mathcal{C}.$ Then, the following two conditions are equivalent:
\begin{itemize}
\item[(a)] $(\mathcal{A,B})$ is a left $n$-cotorsion pair in $\mathcal{C}$.

\item[(b)] $\mathcal{A} = \bigcap_{i=1}^n {}^{\perp_i}\mathcal{B}$ and for any $C \in \mathcal{C}$ there is a short exact sequence 
\[
0 \to K \to A \to C \to 0,
\] 
with $A \in \mathcal{A}$ and $K \in \mathcal{B}^\wedge_{n-1}$.
\end{itemize}
Moreover, if one of the above conditions holds true, then  $(\mathcal{A},\mathcal{B}^\wedge_{n-1})$ is a complete left cotorsion pair in $\mathcal{C}$.
\end{theorem}

\begin{proof} Note that the implication (b) $\Rightarrow$ (a) is trivial. We prove that (a) implies (b). So let us assume that $(\mathcal{A,B})$ is a left $n$-cotorsion pair in $\mathcal{C}$. Then, by Lemma~\ref{lema1}, we get the containments 
\[
\mathcal{A} \subseteq \bigcap_{i = 1}^n\,{}^{\perp_i}\mathcal{B} \subseteq {}^{\perp_1}(\mathcal{B}^\wedge_{n-1}).
\] 
Thus, we only need to prove the remaining containment $\mathcal{A} \supseteq {}^{\perp_1}(\mathcal{B}^\wedge_{n-1})$. It suffices to note that for every $X \in {}^{\perp_1}(\mathcal{B}^\wedge_{n-1}),$ there exists a split epimorphism $A \twoheadrightarrow X$ with kernel in $\mathcal{B}^\wedge_{n-1}$.
\end{proof}

Normally, if we are given a cotorsion pair $(\mathcal{A,B})$ in $\mathcal{C}$, a natural question is whether this pair is complete or hereditary, in order to construct special $\mathcal{A}$-precovers and special $\mathcal{B}$-preenvelopes. Now that we have explored the interplay between cotorsion and higher cotorsion, we shall study in the next section the relation between left and right $n$-cotorsion pairs, and left and right approximations by $\mathcal{A}$ and $\mathcal{B}$. In Section~\ref{sec:hereditary}, on the other hand, we shall deal with the hereditary aspect of $n$-cotorsion pairs $(\mathcal{A,B})$ for which $\mathcal{A}$ is resolving or $\mathcal{B}$ is coresolving.


\section{\textbf{Covers and envelopes from {\textit n}-cotorsion pairs}}\label{sec:approximations}

Approximations has been considered before in the study of higher cotorsion. For instance, in \cite{CriveiTorrecillas} Crivei and Torrecillas presented the concept of $n$-special $\mathcal{A}$-precovers and $m$-special $\mathcal{B}$-preenvelopes (epic $\mathcal{A}$-precovers and monic $\mathcal{B}$-preenvelopes with kernel in $\mathcal{A}^{\perp_n}$ and cokernel in ${}^{\perp_m}\mathcal{B}$, respectively), and established several conditions under which it is possible to obtain such approximations from an $(m,n)$-cotorsion pair $(\mathcal{A,B})$ (that is, $\mathcal{A} = {}^{\perp_m}\mathcal{B}$ and $\mathcal{B} = \mathcal{A}^{\perp_n}$). See \cite[Proposition 3.14 and Theorem 3.15]{CriveiTorrecillas}. 

In this section, we study precovers and preenvelopes coming from an $n$-cotorsion pair $(\mathcal{A,B})$ in an abelian category $\mathcal{C}$. We also define a new family of approximations which we call $(\mathcal{A},k,\mathcal{B})$-precovers and $(\mathcal{A},k,\mathcal{B})$-preenvelopes, as analogs of the special precovers and special preenvelopes coming from a complete cotorsion pair. Among the properties of these new concepts, we prove that the class of objects having a $(\mathcal{A},k,\mathcal{B})$-precover is closed under extensions if we put certain orthogonality condition on $\mathcal{B}$. Later, we shall study some other conditions under which left and right $n$-cotorsion pairs are sources of precovers and preenvelopes with the unique mapping property.


\subsection*{\textbf{Special $(\mathcal{A},k,\mathcal{B})$-precovers and $(\mathcal{A},k,\mathcal{B})$-preenvelopes}}

One consequence of Theorem~\ref{theo:left-n-cotorsion} is that left and right $n$-cotorsion pairs are always sources of precovers and preenvelopes, as stated in the following result.

\begin{proposition}\label{A-precub,B-preenv}
If $(\mathcal{A,B})$ is a left $n$-cotorsion pair in $\mathcal{C}$, then $\mathcal{A}$ is a special precovering class.
\end{proposition}

\begin{proof} 
Let $(\mathcal{A,B})$ be a left $n$-cotorsion pair in $\mathcal{C}.$ In particular, for any $C\in\mathcal{C},$ there is an exact sequence 
\[
0 \to K \to A \to C \to 0,
\] 
where $K \in \mathcal{B}^\wedge_{n-1}$ and $A \in \mathcal{A}$. Moreover, by Theorem~\ref{theo:left-n-cotorsion}, we get that $K \in \mathcal{B}^\wedge_{n-1} \subseteq \mathcal{A}^{\perp_1}$ and thus $A \to C$ is a special $\mathcal{A}$-precover of $C$.
\end{proof}

However, special precovers and preenvelopes are not the only type of approximations coming from left and right $n$-cotorsion pairs. Under certain conditions, we can find more information about the approximations resulting from Proposition~\ref{A-precub,B-preenv}. Following the spirit of \cite{CriveiTorrecillas} concerning the relation between $(m,n)$-cotorsion pairs, $n$-special precovers and $m$-special preenvelopes, we propose the following family of approximations and study their relation with left and right $n$-cotorsion pairs.

\begin{definition}
Let $\mathcal{A}$ and $\mathcal{B}$ be two classes of objects of $\mathcal{C}$, and $k $ be a positive integer. Given an object $C \in \mathcal{C}$, we say that an $\mathcal{A}$-precover $f \colon A \to C$ of $C$ is a \textbf{special $\bm{(\mathcal{A},k,\mathcal{B})}$-precover} if $f$ is epic and $\Ker(f) \in \mathcal{B}_{k-1}^{\wedge}$. The notion of \textbf{special $\bm{(\mathcal{A},k,\mathcal{B})}$-preenvelopes} is defined dually.
\end{definition}

\begin{remark} \
\begin{enumerate}
\item Let $\mathcal{A}$ be a class of objects of $\mathcal{C}$ and $C \in \mathcal{C}$. Note that a morphism $f \colon A \to C$ with $A \in \mathcal{A}$ is a special $\mathcal{A}$-precover if, and only if, it is a special $(\mathcal{A},1,\mathcal{A}^{\perp_1})$-precover. 

\item Any $C\in\mathcal{C}$ admits an $(\mathcal{A},n,\mathcal{B})$-precover if $(\mathcal{A,B})$ is a left $n$-cotorsion pair in $\mathcal{C}$. 
\end{enumerate}
\end{remark}

The following is clear by \cite[Theorem 2.10]{Holm}.

\begin{example}\label{ex:special_AkB-precover}
Given a class of modules $\mathcal{Y} \subseteq \Mod(R)$, a chain complex $X = (X_m)_{m \in \mathbb{Z}}$ is called \emph{$\Hom_R(-,\mathcal{Y})$-acyclic} if $\Hom_R(X,Y) := (\Hom_R(X_m,Y))_{m \in \mathbb{Z}}$ is an exact complex of abelian groups for every $Y \in \mathcal{Y}$. 

Recall that $\mathcal{GP}(R)$ denotes the class of \emph{Gorenstein projective} $R$-modules, that is, modules $M \in \Mod(R)$ such that $M \simeq Z_0(P)$ for some exact and $\Hom_R(-,\mathcal{P}(R))$-acyclic complex $P$ of projective modules. \emph{Gorenstein injective} modules are defined dually, that is, as cycles of exact and $\Hom_R(\mathcal{I}(R),-)$-acyclic complexes of injective modules.  For the class of Gorenstein injective $R$-modules, we shall write $\mathcal{GI}(R)$. 

Let us recall also that the \emph{Gorenstein projective dimension} of an $R$-module $M$, which we denote by ${\rm Gpd}(M)$, is defined as the $\mathcal{GP}(R)$-resolution dimension of $M$, that is,
\[
{\rm Gpd}(M) := \resdim_{\mathcal{GP}(R)}(M).
\] 
The \emph{Gorenstein injective dimension of $M$}, denoted ${\rm Gid}(M)$, is defined similarly. It is known from \cite[Theorem 2.10]{Holm} that every $R$-module $M$ with finite Gorenstein projective dimension, say $\Gpd(M) = m < \infty$, has a Gorenstein projective special precover whose kernel has projective dimension at most $m-1$, that is, $M$ has a special $(\mathcal{GP}(R),m,\mathcal{P}(R))$-precover.
\end{example}

In \cite[Theorem 3.1]{Akinci}, Akinci and Alizade proved that for every hereditary cotorsion pair $(\mathcal{A,B})$ in $\Mod(R)$, the class of objects having a special $\mathcal{A}$-precover is closed under extensions. In what follows, we generalise this result for special $(\mathcal{A},k,\mathcal{B})$-precovers. Let us denote by $\Prec^k(\mathcal{A},\mathcal{B})$ the class of all $C\in\mathcal{C}$ admitting a special $(\mathcal{A},k,\mathcal{B})$-precover.

\begin{theorem}\label{Teo-k-prec} 
Let $n$ be a positive integer and $1\leq k\leq \max(1,n-1)$, and let $\mathcal{A}$ and $\mathcal{B}$ be two classes of objects of $\mathcal{C}$ such that $\Ext^i_\C(\mathcal{A},\B) = 0$ for every $1 \leq i \leq n$, or $\Ext^2_{\mathcal{C}}(\mathcal{A,B}) = 0$ if $n = 1$. If $\mathcal{A}$ and $\B^\wedge_{k-1}$ are closed under extensions, then so is $\Prec^k(\mathcal{A},\B)$. 
\end{theorem}

\begin{proof} 
We prove this result by adapting the arguments given in \cite[Theorem 3.1]{Akinci} to our approach of higher cotorsion. Let 
\[
0 \to X \xrightarrow{f} Y \xrightarrow{g} Z \to 0
\] 
be a short exact sequence in $\C$ such that $X, Z \in \Prec^k(\A,\B)$. First, consider a special $(\mathcal{A},k,\mathcal{B})$-precover of $Z$, say a short exact sequence
\[
0 \to K^Z \xrightarrow{\beta^Z} A^Z \xrightarrow{\alpha^Z} Z \to 0,
\]
with $K^Z \in \mathcal{B}^\wedge_{k-1}$ and $A^Z \in \mathcal{A}$. Taking the pullback of $Y \to Z \leftarrow A^Z$ yields the following commutative diagram with exact rows and columns:
\begin{equation}\label{fig3} 
\parbox{1.75in}{
\begin{tikzpicture}[description/.style={fill=white,inner sep=2pt}] 
\matrix (m) [ampersand replacement=\&, matrix of math nodes, row sep=2.5em, column sep=2.5em,text height=1.25ex, text depth=0.25ex] 
{ 
{} \& K^Z \& K^Z \\
X \& E \& A^Z \\
X \& Y \& Z \\
}; 
\path[>->]
(m-1-2) edge node[left] {\footnotesize$\tilde{\beta}^Z$} (m-2-2) (m-1-3) edge node[right] {\footnotesize$\beta^Z$} (m-2-3)
(m-2-1) edge node[above] {\footnotesize$\overline{f}$} (m-2-2) (m-3-1) edge node[below] {\footnotesize$f$} (m-3-2)
;
\path[->>]
(m-2-2) edge node[left] {\footnotesize$\tilde{\alpha}^Z$} (m-3-2) (m-2-3) edge node[right] {\footnotesize$\alpha^Z$} (m-3-3)
(m-2-2) edge node[above] {\footnotesize$\overline{g}$} (m-2-3) (m-3-2) edge node[below] {\footnotesize$g$} (m-3-3)
;
\path[->] 
(m-2-2)-- node[pos=0.5] {\footnotesize$\mbox{\bf pb}$} (m-3-3)
; 
\path[-,font=\scriptsize]
(m-1-2) edge [double, thick, double distance=2pt] (m-1-3)
(m-2-1) edge [double, thick, double distance=2pt] (m-3-1)
;
\end{tikzpicture} 
}
\end{equation}

Now let us consider a special $(\mathcal{A},k,\mathcal{B})$-precover of $X$, say
\[
0 \to K^X \xrightarrow{\beta^X} A^X \xrightarrow{\alpha^X} X \to 0,
\]
with $A^X \in \mathcal{A}$ and $K^X \in \mathcal{B}^\wedge_{k-1}$. We obtain the following exact sequence of abelian groups:
\[
\Ext^1_{\mathcal{C}}(A^Z,A^X) \xrightarrow{\Ext^1_{\mathcal{C}}(A^Z,\alpha^X)} \Ext^1_{\mathcal{C}}(A^Z,X) \to \Ext^2_{\mathcal{C}}(A^Z,K^X). 
\]
The morphism $\Ext^1_{\mathcal{C}}(A^Z,\alpha^X)$ is epic since $\Ext^2_{\mathcal{C}}(A^Z,K^X) = 0$. The latter can be shown as follows: for the case $n = 1$, we use that $\Ext^2_\C(\A,\B) = 0$. If $n \geq 2$, on the other hand, then $k \leq n - 1$ and so $2 \leq n - k + 1$. Thus, by Proposition~\ref{prop8}, we get that $\Ext^i_\C(\A,\B^\wedge_{k-1}) = 0$ for $i = 1, 2$. 

Knowing that $\Ext^1_{\mathcal{C}}(A^Z,\alpha^X)$ is surjective, we can assert that for the central row 
\[
\eta \colon 0 \to X \xrightarrow{\overline{f}} E \xrightarrow{\overline{g}} A^Z \to 0,
\] 
there exists a short exact sequence
\[
\eta' \colon 0 \to A^X \xrightarrow{\hat{f}} A^Y \xrightarrow{\hat{g}} A^Z \to 0
\] 
such that $\eta$ can be obtained as the pushout of $\eta'$ along $\alpha^X \colon A^X \to X$:
\begin{equation}\label{fig4} 
\parbox{1.75in}{
\begin{tikzpicture}[description/.style={fill=white,inner sep=2pt}] 
\matrix (m) [ampersand replacement=\&, matrix of math nodes, row sep=2.5em, column sep=2.5em, text height=1.25ex, text depth=0.25ex] 
{ 
K^X \& K^X \& {} \\
A^X \& A^Y \& A^Z \\
X \& E \& A^Z \\
}; 
\path[->] 
(m-2-1)-- node[pos=0.5] {\footnotesize$\mbox{\bf po}$} (m-3-2) 
; 
\path[>->]
(m-1-1) edge node[left] {\footnotesize$\beta^X$} (m-2-1) (m-1-2) edge node[right] {\footnotesize$\overline{\beta}^X$} (m-2-2)
(m-2-1) edge node[above] {\footnotesize$\hat{f}$} (m-2-2)
(m-3-1) edge node[below] {\footnotesize$\overline{f}$} (m-3-2)
;
\path[->>]
(m-2-1) edge node[left] {\footnotesize$\alpha^X$} (m-3-1) (m-2-2) edge node[right] {\footnotesize$\overline{\alpha}^X$} (m-3-2)
(m-2-2) edge node[above] {\footnotesize$\hat{g}$} (m-2-3) (m-3-2) edge node[below] {\footnotesize$\overline{g}$} (m-3-3)
;
\path[-,font=\scriptsize]
(m-1-1) edge [double, thick, double distance=2pt] (m-1-2)
(m-2-3) edge [double, thick, double distance=2pt] (m-3-3)
;
\end{tikzpicture} 
}
\end{equation} 
Since $A^X, A^Z \in \mathcal{A}$ and $\mathcal{A}$ is closed under extensions, we have that $A^Y \in \mathcal{A}$. From central columns in diagrams \eqref{fig3} and \eqref{fig4}, we obtain the following commutative diagram with exact rows and columns after taking the pullback of $K^Z \to E \leftarrow A^Y$:
\begin{equation}\label{fig5} 
\parbox{1.75in}{
\begin{tikzpicture}[description/.style={fill=white,inner sep=2pt}] 
\matrix (m) [ampersand replacement=\&, matrix of math nodes, row sep=2.5em, column sep=2.5em, text height=1.25ex, text depth=0.25ex] 
{ 
K^X \& K^X \& {} \\
K^Y \& A^Y \& Y \\
K^Z \& E \& Y \\
}; 
\path[->] 
(m-2-1)-- node[pos=0.5] {\footnotesize$\mbox{\bf pb}$} (m-3-2)
; 
\path[>->]
(m-1-1) edge node[left] {\footnotesize$\tilde{f}$} (m-2-1) (m-1-2) edge node[right] {\footnotesize$\overline{\beta}^X$} (m-2-2)
(m-2-1) edge node[above] {\footnotesize$\beta^Y$} (m-2-2) (m-3-1) edge node[below] {\footnotesize$\tilde{\beta}^Z$} (m-3-2)
;
\path[->>]
(m-2-1) edge node[left] {\footnotesize$\tilde{g}$} (m-3-1) (m-2-2) edge node[right] {\footnotesize$\overline{\alpha}^X$} (m-3-2) 
(m-2-2) edge node[above] {\footnotesize$\alpha^Y$} (m-2-3) (m-3-2) edge node[below] {\footnotesize$\tilde{\alpha}^Z$} (m-3-3)
;
\path[-,font=\scriptsize]
(m-1-1) edge [double, thick, double distance=2pt] (m-1-2)
(m-2-3) edge [double, thick, double distance=2pt] (m-3-3)
;
\end{tikzpicture} 
}
\end{equation} 
Note that $K^Y \in \mathcal{B}_{k-1}^{\wedge}$ since $K^X, K^Z \in \mathcal{B}_{k-1}^{\wedge}$ and $\mathcal{B}_{k-1}^{\wedge}$ is closed under extensions. Therefore, by using that $\Ext^1_\C(\A,\B^\wedge_{k-1}) = 0$, it follows that the central row of \eqref{fig5} is a special $(\mathcal{A},k,\mathcal{B})$-precover of $Y$. 
\end{proof}

\begin{remark}\label{rem:closure_Prec-n}
Note that the closure under extensions for the class $\Prec^n(\mathcal{A},\mathcal{B})$ is not covered in the previous result. This will be studied in Section \ref{sec:hereditary}.
\end{remark}

Note that in Theorem \ref{Teo-k-prec} we require the assumption that $\mathcal{B}^\wedge_{k-1}$ is closed under extensions. In the following result, we provide a sufficient condition that guarantees this closure property.

\begin{lemma}\label{Bk-1ext}
Let $k $ be a positive integer and $\mathcal{B}$ be a class of objects of $\mathcal{C}$ that is closed under extensions. If 
$\Ext_{\mathcal{C}}^{1}(\mathcal{B},\mathcal{B}^\wedge_{k-1}) = 0$, then $\mathcal{B}_{k-1}^{\wedge}$ is closed under extensions.
\end{lemma}

\begin{proof} 
Consider a short exact sequence in $\mathcal{C}$
\[
\eta \colon 0 \to X \to Y \to Z \to 0,
\] 
with $X, Z \in \mathcal{B}_{k-1}^{\wedge}$. We show that $\resdim_{\mathcal{B}}(Y) \leq \resdim_{\mathcal{B}}(Z).$ In order to do that,  we proceed by induction on $m := \resdim_{\mathcal{B}}(X).$
\begin{itemize}
\item For the initial case, suppose $m = 0$. If $\resdim_{\mathcal{B}}(Z) = 0$, it follows that $Y \in \mathcal{B}$ and so $\resdim_{\mathcal{B}}(Y) = 0 = \resdim_{\mathcal{B}}(Z)$, since $\mathcal{B}$ is closed under extensions. We can thus assume that $\resdim_{\mathcal{B}}(Z) \geq 1$. Then, there is an exact sequence 
\[
\delta \colon 0 \to Z' \to B_0 \to Z \to 0,
\] 
where $\resdim_{\mathcal{B}}(Z') +1= \resdim_{\mathcal{B}}(Z)$ and $B_0 \in \mathcal{B}$. Taking the pullback of $Y \to Z \leftarrow B_0$ produces the following commutative diagram with exact rows and columns:
\begin{equation}\label{fig1} 
\parbox{1.5in}{
\begin{tikzpicture}[description/.style={fill=white,inner sep=2pt}] 
\matrix (m) [ampersand replacement=\&, matrix of math nodes, row sep=2.5em, column sep=2.5em, text height=1.25ex, text depth=0.25ex] 
{ 
{} \& Z' \& Z' \\
X \& L \& B_0 \\
X \& Y \& Z \\
}; 
\path[>->]
(m-2-1) edge (m-2-2) (m-3-1) edge (m-3-2)
(m-1-2) edge (m-2-2) (m-1-3) edge (m-2-3)
;
\path[->>]
(m-2-2) edge (m-2-3) (m-3-2) edge (m-3-3)
(m-2-2) edge (m-3-2) (m-2-3) edge (m-3-3)
;
\path[->] 
(m-2-2)-- node[pos=0.5] {\footnotesize$\mbox{\bf pb}$} (m-3-3)
; 
\path[-,font=\scriptsize]
(m-1-2) edge [double, thick, double distance=2pt] (m-1-3)
(m-2-1) edge [double, thick, double distance=2pt] (m-3-1)
;
\end{tikzpicture} 
}
\end{equation} 
Note that $L\in \mathcal{B},$ since $\mathcal{B}$ is closed under extensions. Hence, from the central column in \eqref{fig1}, we get the inequality $\resdim_{\mathcal{B}}(Y) \leq 1+\resdim_{\mathcal{B}}(Z') = \resdim_{\mathcal{B}}(Z).$ Therefore, we have $Y \in \mathcal{B}_{k-1}^{\wedge}$.

\item For the successor case, let $1 \leq m \leq k-1$ and suppose that $\resdim_{\mathcal{B}}(Y') \leq \resdim_{\mathcal{B}}(Z')$ in any short exact sequence 
\[
0 \to X' \to Y' \to Z' \to 0,
\] 
with $X', Z' \in \mathcal{B}^\wedge_{k-1}$ and $\resdim_{\mathcal{B}}(X') < m$. 

Now for the object $Z$ appearing in the sequence $\eta$, consider a short exact sequence as $\delta$ above. Take the pullback of $Y \to Z \leftarrow B_0$ to construct a diagram as \eqref{fig1}, and consider the resulting central row and central column:
\begin{align*}
\varepsilon \colon & 0 \to X \to L \to B_0 \to 0, \\
\tau \colon & 0 \to Z' \to L \to Y \to 0.
\end{align*}
Since $\Ext^1_{\mathcal{C}}(\mathcal{B},\mathcal{B}^\wedge_{k-1}) = 0$, the sequence $\varepsilon$ splits, and so $L = X \oplus B_0$. On the other hand, consider a short exact sequence 
\[
\varepsilon' \colon 0 \to X' \to B_1 \to X \to 0
\] 
with $B_1 \in \mathcal{B}$ and $\resdim_{\mathcal{B}}(X') +1= \resdim_{\mathcal{B}}(X)$, and form the exact sequence 
\[
\varepsilon'' \colon 0 \to X' \to B_1 \oplus B_0 \to X \oplus B_0 \to 0
\] 
by adding to $\varepsilon'$ the identity on $B_0$. Now, take the pullback of $Z' \to X \oplus B_0 \leftarrow B_1 \oplus B_0$ in order to obtain the following commutative diagram with exact rows and columns: 
\begin{equation}\label{fig2} 
\parbox{1.75in}{
\begin{tikzpicture}[description/.style={fill=white,inner sep=2pt}] 
\matrix (m) [ampersand replacement=\&, matrix of math nodes, row sep=2.5em, column sep=2.5em, text height=1.25ex, text depth=0.25ex] 
{ 
X' \& X' \& {} \\
Y' \& B_1 \oplus B_0 \& Y \\
Z' \& X \oplus B_0 \& Y \\
}; 
\path[>->]
(m-1-1) edge (m-2-1) (m-1-2) edge (m-2-2)
(m-2-1) edge (m-2-2) (m-3-1) edge (m-3-2)
;
\path[->>]
(m-2-2) edge (m-2-3) (m-3-2) edge (m-3-3)
(m-2-1) edge (m-3-1) (m-2-2) edge (m-3-2)
;
\path[->] 
(m-2-1)-- node[pos=0.5] {\footnotesize$\mbox{\bf pb}$} (m-3-2)
; 
\path[-,font=\scriptsize]
(m-1-1) edge [double, thick, double distance=2pt] (m-1-2)
(m-2-3) edge [double, thick, double distance=2pt] (m-3-3)
;
\end{tikzpicture} 
}
\end{equation} 
Note that $X', Z' \in \mathcal{B}^\wedge_{k-1}$ and $\resdim_{\mathcal{B}}(X') < m$ in the left column of \eqref{fig2}, so we can apply the induction hypothesis to conclude that $\resdim_{\mathcal{B}}(Y') \leq \resdim_{\mathcal{B}}(Z').$  On the other hand, $B_1 \oplus B_0 \in \mathcal{B}$ and so we have that 
\[
\resdim_{\mathcal{B}}(Y) \leq \resdim_{\mathcal{B}}(Y') + 1 \leq \resdim_{\mathcal{B}}(Z') +1=\resdim_{\mathcal{B}}(Z).
\] 
Therefore, $Y \in \mathcal{B}_{k-1}^{\wedge}$.
\end{itemize} 
\end{proof}

The condition $\Ext_{\mathcal{C}}^{1}(\mathcal{B},\mathcal{B}^\wedge_{k-1}) = 0$ in Lemma~\ref{Bk-1ext} seems to be more or less difficult to satisfy for a class $\mathcal{B} \subseteq \mathcal{C}$ closed under extensions. However, we can find classes satisfying this condition, such as the $m$-rigid subcategories. Following \cite[Definition 1.1]{IyamaCluster}, for an integer $m \geq 1$ we say that a subcategory $\mathcal{D} \subseteq \mathcal{C}$ is \emph{$m$-rigid} if $\Ext^i_{\mathcal{C}}(\mathcal{D,D}) = 0$ for any $0 < i < m$. 

Let us show the following characterisation of $m$-rigid subcategories which involves the hypotheses of Lemma~\ref{Bk-1ext}.

\begin{proposition}\label{prop3.5}
If $\mathcal{D}$ is $m$-rigid subcategory of $\mathcal{C}$ for some $m \geq 2$, then the equality $\Ext^1_{\mathcal{C}}(\mathcal{D},\mathcal{D}^\wedge_{k-1}) = 0$ holds for every $1 \leq k \leq m-1$. Moreover, if $\mathcal{C}$ has enough injectives and $\mathcal{I}(\mathcal{C}) \subseteq \mathcal{D}$, then the converse statement also holds. That is, if there exists $m \geq 2$ such that $\Ext^1_{\mathcal{C}}(\mathcal{D},\mathcal{D}^\wedge_{k-1}) = 0$ for every $1 \leq k \leq m - 1$, then $\mathcal{D}$ is $m$-rigid. 
\end{proposition}

\begin{proof}
The ``only if'' part follows by Proposition~\ref{prop8}. Now for the ``if'' part in the case where $\mathcal{C}$ has enough injectives and $\mathcal{I}(\mathcal{C}) \subseteq \mathcal{D}$, suppose that $\Ext^1_{\mathcal{C}}(\mathcal{D},\mathcal{D}^\wedge_{k-1}) = 0$ for every $0 \leq k-1 \leq m-2$. For $0 < i < m$, we have that $\Ext^i_{\mathcal{C}}(D,D') \cong \Ext^1_{\mathcal{C}}(D,K)$ where $D, D' \in \mathcal{D}$ and $K$ is an injective $(i-1)$-cosyzygy of $D'$. Note that $K \in \mathcal{D}^\wedge_{i-1}$, and so from the assumption it follows that $\Ext^1_{\mathcal{C}}(D,K) = 0$, that is, $\Ext^i_{\mathcal{C}}(D,D') = 0$. Therefore, $\mathcal{D}$ is $m$-rigid. 
\end{proof}

\begin{corollary}\label{c3.6} Let $\mathcal{D}$ be a  $m$-rigid class closed under finite coproducts and with $m \geq 2$. Then, the class $\mathcal{D}^\wedge_k$ is closed under extensions, for every $0 \leq k \leq m-2.$
\end{corollary}

\begin{proof} 
Notice that $\Ext^1_{\mathcal{C}}(\mathcal{D,D}) = 0$ and $\mathcal{D}$ being closed under finite coproducts imply that $\mathcal{D}$ is closed under extensions. Then, the result follows from Lemma~\ref{Bk-1ext} and Proposition~\ref{prop3.5}. 
\end{proof}

\begin{remark}\label{rem:special_n-1}
Under the hypothesis of Theorem \ref{Teo-k-prec}, given a short exact sequence 
\[
0 \to X \xrightarrow{f} Y \xrightarrow{g} Z \to 0
\]
with $X$ and $Z$ having a special $(\mathcal{A},n,\mathcal{B})$-precover, say
\begin{align*}
\rho_X \colon & 0 \to B^X_{n-1} \xrightarrow{\beta^X_{n-1}} B^X_{n-2} \to \cdots \to B^X_1 \xrightarrow{\beta^X_1} B^X_0 \xrightarrow{\beta^X_0} A^X \xrightarrow{\alpha^X} X \to 0, \\
\rho_Z \colon & 0 \to B^Z_{n-1} \xrightarrow{\beta^Z_{n-1}} B^Z_{n-2} \to \cdots \to B^Z_1 \xrightarrow{\beta^Z_1} B^Z_0 \xrightarrow{\beta^Z_0} A^Z \xrightarrow{\alpha^Z} Z \to 0, 
\end{align*} 
it is possible in some cases to construct a special $(\mathcal{A},n,\mathcal{B})$-precover of $Y$ \emph{compatible} with $\rho_X$ and $\rho_Z$, that is, an exact sequence
\[
\rho_Y \colon 0 \to B^Y_{n-1} \xrightarrow{\beta^Y_{n-1}} B^Y_{n-2} \to \cdots \to B^Y_1 \xrightarrow{\beta^Y_1} B^Y_0 \xrightarrow{\beta^Y_0} A^Y \xrightarrow{\alpha^Y} Y \to 0
\]
along with a commutative diagram with exact rows and columns:
\begin{equation}\label{fig_compatible} 
\parbox{4.75in}{
\begin{tikzpicture}[description/.style={fill=white,inner sep=2pt}] 
\matrix (m) [ampersand replacement=\&, matrix of math nodes, row sep=2.5em, column sep=2.5em, text height=1.25ex, text depth=0.25ex] 
{ 
B^X_{n-1} \& B^X_{n-2} \& \cdots \& B^X_1 \& B^X_0 \& A^X \& X \\
B^Y_{n-1} \& B^Y_{n-2} \& \cdots \& B^Y_1 \& B^Y_0 \& A^Y \& Y \\
B^Z_{n-1} \& B^Z_{n-2} \& \cdots \& B^Z_1 \& B^Z_0 \& A^Z \& Z \\
}; 
\path[->]
(m-1-2) edge (m-1-3) (m-1-3) edge (m-1-4) (m-1-4) edge node[above] {\footnotesize$\beta^X_1$} (m-1-5) (m-1-5) edge node[above] {\footnotesize$\beta^X_0$} (m-1-6)
(m-2-2) edge (m-2-3) (m-2-3) edge (m-2-4) (m-2-4) edge node[above] {\footnotesize$\beta^Y_1$} (m-2-5) (m-2-5) edge node[above] {\footnotesize$\beta^Y_0$} (m-2-6)
(m-3-2) edge (m-3-3) (m-3-3) edge (m-3-4) (m-3-4) edge node[above] {\footnotesize$\beta^Z_1$} (m-3-5) (m-3-5) edge node[above] {\footnotesize$\beta^Z_0$} (m-3-6)
;
\path[>->]
(m-1-1) edge node[above] {\footnotesize$\beta^X_0$} (m-1-2) 
(m-2-1) edge node[above] {\footnotesize$\beta^Y_0$} (m-2-2)
(m-3-1) edge node[above] {\footnotesize$\beta^Z_0$} (m-3-2)
(m-1-1) edge node[right] {\footnotesize$f_{n-1}$} (m-2-1) (m-1-2) edge node[right] {\footnotesize$f_{n-2}$} (m-2-2) (m-1-4) edge node[right] {\footnotesize$f_1$} (m-2-4) (m-1-5) edge node[right] {\footnotesize$f_0$} (m-2-5) (m-1-6) edge node[right] {\footnotesize$\hat{f}$} (m-2-6) (m-1-7) edge node[right] {\footnotesize$f$} (m-2-7)
;
\path[->>]
(m-1-6) edge node[above] {\footnotesize$\alpha^X$} (m-1-7)
(m-2-6) edge node[above] {\footnotesize$\alpha^Y$} (m-2-7)
(m-3-6) edge node[above] {\footnotesize$\alpha^Z$} (m-3-7)
(m-2-1) edge node[right] {\footnotesize$g_{n-1}$} (m-3-1) (m-2-2) edge node[right] {\footnotesize$g_{n-2}$} (m-3-2)  (m-2-4) edge node[right] {\footnotesize$g_1$} (m-3-4) (m-2-5) edge node[right] {\footnotesize$g_0$} (m-3-5) (m-2-6) edge node[right] {\footnotesize$\hat{g}$} (m-3-6) (m-2-7) edge node[right] {\footnotesize$g$} (m-3-7)
;
\end{tikzpicture} 
}
\end{equation} 
To prove this assertion, we shall need the analog of hereditary cotorsion pairs for left $n$-cotorsion pairs, presented later in Section~\ref{sec:hereditary}. 

For now, we can show the case $n = 1$. That is, we are given two classes of objects $\mathcal{A}$ and $\mathcal{B}$ in $\mathcal{C}$, closed under extensions, such that $\Ext^2_{\mathcal{C}}(\mathcal{A,B}) = 0$. Following the proof of Theorem~\ref{Teo-k-prec}, we have a short exact sequence 
\[
0 \to X \xrightarrow{f} Y \xrightarrow{g} Z \to 0
\] 
where $X$ and $Z$ have special $\mathcal{A}$-precovers with kernel in $\mathcal{B}$, say:
\begin{align*}
\rho_X \colon & 0 \to B^X \xrightarrow{\beta^X} A^X \xrightarrow{\alpha^X} X \to 0, \\
\rho_Z \colon & 0 \to B^Z \xrightarrow{\beta^Z} A^Z \xrightarrow{\alpha^Z} Z \to 0.
\end{align*}
From the diagrams \eqref{fig3}, \eqref{fig4} and \eqref{fig5}, we construct the following diagram with exact rows and columns:
\begin{equation}\label{fig_caso_n1} 
\parbox{1.75in}{
\begin{tikzpicture}[description/.style={fill=white,inner sep=2pt}] 
\matrix (m) [ampersand replacement=\&, matrix of math nodes, row sep=2.5em, column sep=2.5em, text height=1.25ex, text depth=0.25ex] 
{ 
B^X \& A^X \& X \\
B^Y \& A^Y \& Y \\
B^Z \& A^Z \& Z \\
}; 
\path[>->]
(m-1-1) edge node[above] {\footnotesize$\beta^X$} (m-1-2) (m-2-1) edge node[above] {\footnotesize$\beta^Y$} (m-2-2) (m-3-1) edge node[above] {\footnotesize$\beta^Z$} (m-3-2)
(m-1-1) edge node[left] {\footnotesize$\tilde{f}$} (m-2-1) (m-1-2) edge node[left] {\footnotesize$\hat{f}$} (m-2-2) (m-1-3) edge node[right] {\footnotesize$f$} (m-2-3)
;
\path[->>]
(m-1-2) edge node[above] {\footnotesize$\alpha^X$} (m-1-3) (m-2-2) edge node[above] {\footnotesize$\alpha^Y$} (m-2-3) (m-3-2) edge node[above] {\footnotesize$\alpha^Z$} (m-3-3)
(m-2-1) edge node[left] {\footnotesize$\tilde{g}$} (m-3-1) (m-2-2) edge node[left] {\footnotesize$\hat{g}$} (m-3-2) (m-2-3) edge node[right] {\footnotesize$g$} (m-3-3)
;
\end{tikzpicture} 
}
\end{equation} 
We check that \eqref{fig_caso_n1} commutes:
\begin{align*}
\beta^Y \circ \tilde{f} & = \overline{\beta}^X = \hat{f} \circ \beta^X & \mbox{(by \eqref{fig5} and \eqref{fig4})}, \\
\alpha^Y \circ \hat{f} & = \tilde{\alpha}^Z \circ \overline{\alpha}^X \circ \hat{f} = \tilde{\alpha}^Z \circ \overline{f} \circ \alpha^X = f \circ \alpha^X & \mbox{(by \eqref{fig5}, \eqref{fig4} and \eqref{fig3})}, \\
\beta^Z \circ \tilde{g} & = \overline{g} \circ \tilde{\beta}^Z \circ \tilde{g} = \overline{g} \circ \overline{\alpha}^X \circ \beta^Y = \hat{g} \circ \beta^Y & \mbox{(by \eqref{fig3}, \eqref{fig5} and \eqref{fig4})}, \\
\alpha^Z \circ \hat{g} & = \alpha^Z \circ \overline{g} \circ \overline{\alpha}^X = g \circ \tilde{\alpha}^Z \circ \overline{\alpha}^X =  g \circ \alpha^Y & \mbox{(by \eqref{fig4}, \eqref{fig3} and \eqref{fig5})}.
\end{align*}
\end{remark}

\begin{corollary}\label{Coro-k-prec} 
Let $(\mathcal{A,B})$ be a left $n$-cotorsion pair in $\C$ with $n \geq 2$, such that:
\begin{itemize}
\item[(i)] $\B$ is closed under finite coproducts in $\C$, and 

\item[(ii)] $\Ext^1_{\mathcal{C}}(\mathcal{B},\mathcal{B}^\wedge_{k-1}) = 0$ for any $1 \leq k \leq \max(1,n-1)$.
\end{itemize} 
Then, the class $\Prec^k(\A,\B)$ is closed under extensions for any $1\leq k\leq \max(1,n-1)$. 
\end{corollary}

\begin{proof} 
First, note that since $\B$ is closed under finite coproducts in $\C$ and $\Ext^1_{\mathcal{C}}(\mathcal{B},\mathcal{B}^\wedge_{0}) = 0$, it follows that $\B$ is closed under extensions. Furthermore, from Lemma~\ref{Bk-1ext} we obtain that $\mathcal{B}^\wedge_{k-1}$ is closed under extensions. On the other hand, Theorem~\ref{theo:left-n-cotorsion} allows us to conclude that $\A$ is also closed under extensions. Thus, Theorem~\ref{Teo-k-prec} gives us the result.
\end{proof}


\subsection*{\textbf{\textit{n}-Cotorsion and approximations having the unique mapping property}}

Now let us study the relation between higher cotorsion and approximations with the unique mapping property. This point has been tackled in other contexts of higher cotorsion. For instance, Crivei and Torrecillas considered $(m,n)$-cotorsion pairs $(\mathcal{A,B})$ in Grothendieck categories $\mathcal{G}$ with enough projectives, and studied some conditions under which it is possible to obtain special precovers with the unique mapping property. Namely, they proved in \cite[Theorem 3.15]{CriveiTorrecillas} that every object in $\mathcal{G}$ has an $\mathcal{A}^{\perp_{n+1}}$-preenvelope with the unique mapping property if, and only if, $\mathcal{G} = \mathcal{A}^{\perp_{n+1}}$. 

Recall that an $\mathcal{A}$-precover $f \colon A \to C$ of $C \in \mathcal{C}$ is said to have the \emph{unique mapping property} if for every morphism $f' \colon A' \to C$ with $A' \in \mathcal{A}$ there exists a unique $h \colon A' \to A$ such that $f' = f \circ h$. The notion of an $\mathcal{A}$-preenvelope having the unique mapping property is defined dually. 

The importance of the unique mapping property lies in its applications, which go from the description of certain categories of modules, to characterisations of rings that involve its global or weak dimension. One of these applications has to do with the existence of flat envelopes. Specifically, Asensio Mayor and Mart\'inez Hern\'andez proved in \cite[Proposition 2.1]{AsensioMartinez} that for any ring $R$, every module has a flat envelope with the unique mapping property if, and only if, $R$ is right coherent and with weak dimension ${\rm wd}(R) \leq 2$. The dual version of this result was proved by Mao and Ding in \cite[Corollary 2.4]{MaoDing_wgd}, that is, the latter condition is also equivalent to saying that every right $R$-module has an absolutely pure cover with the unique mapping property. One interesting question about the class $\mathcal{AP}(R)$ of absolutely pure $R$-modules (also known as $\text{FP}$-injective $R$-modules) is whether it is closed under direct limits. Mao and Ding also proved in \cite[Proposition 6.7]{MaoDing_finite} that if every module has an absolutely pure cover with the unique mapping property, then $\mathcal{AP}(R)$ is closed under direct limits. There are also other characterisations of the weak dimension of coherent rings involving the unique mapping property with respect to flat and projective envelopes (the reader can see \cite[Corollaries 3.4 and 3.9]{Ding96}, also by Ding).      

In the following lines, we provide other conditions, within the context of $n$-cotorsion pairs, under which one can obtain approximations with the unique mapping property. The results obtained in this direction will be applied in Section~\ref{sec:hereditary} to comment more on Mao a Ding's \cite[Corollary 2.4]{MaoDing_wgd}, and in Section~\ref{sec:applications} in the field of Gorenstein homological algebra, where we extend some results concerning Gorenstein projective envelopes and Gorenstein injective covers with the unique mapping property. 

Recall that the classes $\mathcal{B}^\wedge_{k}$ play an important role in the concept of left $n$-cotorsion pairs $(\mathcal{A,B})$ in abelian categories. In the study of approximations having the unique mapping property, we shall need to consider the classes $\mathcal{A}^\wedge_{k}$ instead. Let us begin showing the following two properties.

\begin{proposition}\label{Bvee sub Aperp}
Let $\mathcal{A}$ and $\mathcal{B}$ be two classes of objects of $\mathcal{C}$ such that $\Ext^i_{\mathcal{C}}(\mathcal{A,B}) = 0$ for every $1 \leq i \leq n$. Then, the containment $\mathcal{A}^\wedge_{k} \subseteq {}^{\perp_{k+1}}\mathcal{B}$ holds for every $0 \leq k \leq n - 1$.
\end{proposition}

\begin{proof}
It follows by induction on $k$.
\end{proof}

\begin{proposition}\label{ncotder y cotder}
The following conditions are equivalent for any left $n$-cotorsion pair $(\mathcal{A,B})$ in $\mathcal{C}$:
\begin{itemize}
\item[(a)] $\mathcal{A} = {}^{\perp_1}\mathcal{B}$.

\item[(b)] The equality ${}^{\perp_{k+1}}\mathcal{B} = \mathcal{A}^\wedge_k$ holds for every $0 \leq k \leq n - 1$.
\end{itemize}
\end{proposition}

\begin{proof} 
The implication (b) $\Rightarrow$ (a) follows by setting $k = 0$, since $\A^\wedge_0 = \A$. 

Now if we assume that condition (a) holds true, note that the case $k = 0$ is clear. Thus, we may suppose that $k \geq 1$. The containment $\mathcal{A}^\wedge_k \subseteq {}^{\perp_{k+1}}\mathcal{B}$ follows by Proposition~\ref{Bvee sub Aperp}. For the remaining containment $\mathcal{A}^\wedge_k \supseteq {}^{\perp_{k+1}}\mathcal{B}$, consider an object $M \in {}^{\perp_{k+1}}\mathcal{B}$. Since $\mathcal{A}$ is precovering, by Proposition~\ref{A-precub,B-preenv} we can construct an exact sequence of the form
\[
0 \to K \to A_{k-1} \to \cdots \to A_1 \to A_0 \to M \to 0,
\]
where $A_j \in \mathcal{A}$ for every $0 \leq j \leq k-1$. By using the relation $\Ext^i_{\mathcal{C}}(\mathcal{A,B}) = 0$ for $1 \leq i \leq n,$ along with dimension shifting, we have that $\Ext^1_\C(K,B)\simeq\Ext^{k+1}_\C(M,B)=0$ and thus
$K \in {}^{\perp_1}\mathcal{B} = \mathcal{A}$. Hence, $M \in \mathcal{A}^\wedge_k$. 
\end{proof}

\begin{remark}\label{RkAortB} 
Given a left $n$-cotorsion pair $(\mathcal{A,B})$ in $\mathcal{C}$, one may ask for possible cases where condition (a) in the previous proposition holds. 
 
The most obvious situation is when $n = 1$, as we get from Theorem~\ref{theo:left-n-cotorsion} that $\A = {}^{\perp_1}\B$, that is, we have that $(\mathcal{A,B})$ is a complete left cotorsion pair in $\mathcal{C}$.  

The equality $\mathcal{A} = {}^{\perp_1}\mathcal{B}$ is not necessarily true if $n \geq 2$, like for instance in the left $n$-cotorsion pair $(\mathcal{GP}(R),\mathcal{P}(R))$ in $\Mod(R)$, with $R$ an $n$-Iwanaga-Gorenstein ring, considered at the beginning of Section~\ref{sec:applications}. However, in the case where $\B$ is closed under taking cokernels of monomorphisms between objects in $\B$, one can note that $\A = {}^{\perp_1}\B$. 
\end{remark}

In \cite[dual of Proposition 3.3]{CriveiTorrecillas}, Crivei and Torrecillas proved that for every class $\mathcal{B}$ of objects of $\mathcal{C}$,  an abelian category with enough projectives, if every object of $\mathcal{C}$ has a ${}^{\perp_{k+1}}\mathcal{B}$-precover with the unique mapping property, then $\mathcal{C} = {}^{\perp_{k+1}}\mathcal{B}$. This fact is extended in the following result using Proposition~\ref{ncotder y cotder}. Moreover, we shall note that the consequences of having $\mathcal{A} = {}^{\perp_1}\mathcal{B}$ for a left $n$-cotorsion pair $(\mathcal{A,B})$ in $\mathcal{C}$ have an impact on how $\mathcal{C}$ can be described in terms of relative homological dimensions.

\begin{corollary}\label{Aperp y Bvee} 
Let $\mathcal{C}$ be an abelian category with enough projectives, and $(\mathcal{A,B})$ be a left $n$-cotorsion pair in $\mathcal{C}$. If the equality $\mathcal{A} = {}^{\perp_1}\mathcal{B}$ holds, then the following conditions are equivalent for any integer $0 \leq k \leq n - 1$:
\begin{enumerate}
\item[(a)] Every object in $\mathcal{C}$ has a ${}^{\perp_{k+1}}\mathcal{B}$-precover with the unique mapping property.

\item[(b)] $\resdim_{\mathcal{A}}(\C) \leq k.$

\item[(c)] $\mathcal{C} = {}^{\perp_{k+1}}\mathcal{B}$.  
\end{enumerate}
\end{corollary}

\begin{proof} 
It follows from \cite[dual of Proposition 3.3]{CriveiTorrecillas} and Proposition~\ref{ncotder y cotder}.
\end{proof}

\begin{remark}
The equivalence (a) $\Leftrightarrow$ (c) of Corollary~\ref{Aperp y Bvee} is also proven in \cite[dual of Theorem 3.7]{CriveiTorrecillas} under different conditions. Namely, the authors work in the case where $\mathcal{C}$ is a Grothendieck category with enough projectives. Also, the class $\mathcal{B}$ need not be part of an $n$-cotorsion pair in this reference. 
\end{remark}

Corollary~\ref{Aperp y Bvee} establishes some conditions under which it is possible to construct, from a left $n$-cotorsion pair of the form $({}^{\perp_1}\mathcal{B},\mathcal{B})$, right approximations with the unique mapping property. With respect to left approximations, we have the following.

\begin{corollary}\label{corUMP2}
Let $\mathcal{C}$ be an abelian category with enough projectives and $(\mathcal{A,B})$ be a left $n$-cotorsion pair with $n \geq 3$.  If the equality $\mathcal{A} = {}^{\perp_1}\mathcal{B}$ holds, then the following conditions are equivalent:
\begin{itemize}
\item[(a)] Every object in $\mathcal{C}$ has an $\mathcal{A}$-envelope with the unique mapping property. 

\item[(b)] Every object in $\mathcal{C}$ has an $\mathcal{A}$-envelope and $\resdim_{\mathcal{A}}(\C) \leq 2.$ 
\end{itemize}
\end{corollary}

\begin{proof} 
First, we get from Proposition~\ref{A-precub,B-preenv} that $\mathcal{A}$ is special precovering. Moreover, by Theorem~\ref{theo:left-n-cotorsion},  $\mathcal{A} = {}^{\perp_1}\mathcal{B} \subseteq {}^{\perp_i}\mathcal{B}$ for $i = 2, 3$ since $n \geq 3.$ Thus, by Proposition~\ref{ncotder y cotder} we have ${}^{\perp_3}\mathcal{B} = \mathcal{A}_{2}^{\wedge}$. Hence, the result follows by \cite[dual of Theorem 3.16]{CriveiTorrecillas}. 
\end{proof}


\section{\textbf{Higher cotorsion from hereditary cotorsion pairs}}\label{sec:hereditary}

In this section we analyse the situation where we are given a (left or right) $n$-cotorsion pair $(\mathcal{A,B})$ in $\mathcal{C}$ where $\mathcal{A}$ is resolving or $\mathcal{B}$ is coresolving. This will be presented in three approaches. First, we study the relation between (left and right) $n$-cotorsion pairs and hereditary cotorsion pairs. We shall see that the only $n$-cotorsion pairs $(\mathcal{A,B})$ with $\mathcal{A}$ resolving are precisely the hereditary complete cotorsion pairs. Then, we shall comment on hereditary cotorsion pairs $(\mathcal{A,B})$ satisfying the property $\mathcal{A} \subseteq \mathcal{B}$, and provide several characterisations for them. These will allow us to note that the only $n$-cotorsion pair $(\mathcal{A,B})$ with $\mathcal{A}$ resolving and $\mathcal{A} \subseteq \mathcal{B}$ is the trivial $n$-cotorsion pair $(\mathcal{P}(\mathcal{C}),\mathcal{C})$ from Example~\ref{ex:trivial}. The previous suggest that being hereditary for higher cotorsion should be a one-sided notion. Thus, we shall propose in the last part of this section the concept of hereditary left $n$-cotorsion pair, and show that for any such pair the class $\Prec^{n}(\mathcal{A,B})$ is closed under extensions (see Remark \ref{rem:closure_Prec-n}). The latter is an important result that will help us to construct certain $n$-cotorsion pairs of chain complexes from an $n$-cotorsion pair in the ground category $\mathcal{C}$.

Recall that a cotorsion pair $(\mathcal{A,B})$ in an abelian category $\mathcal{C}$ is \emph{hereditary} if $\mathcal{A}$ is resolving and $\mathcal{B}$ is coresolving. The term \emph{resolving} means that $\mathcal{A}$ is closed under extensions and under taking kernels of epimorphisms between its objects, and that the class $\mathcal{P}(\mathcal{C})$ of projective objects of $\mathcal{C}$ is contained in $\mathcal{A}$. Coresolving classes are defined dually. 

It is well known that in any abelian category $\mathcal{C}$ with enough injectives (like for instance any Grothendieck category), for any cotorsion pair $(\mathcal{A,B})$ in $\mathcal{C}$ one has that $\mathcal{B}$ is coresolving if, and only if, $\Ext^2_{\mathcal{C}}(\mathcal{A,B}) = 0$. Dually, the latter is also equivalent to $\mathcal{A}$ being resolving provided that $\mathcal{C}$ has enough projectives. For $n$-cotorsion pairs, with $n \geq 2$, we can obtain a similar equivalence without having either enough projectives or injectives, as we show in the following result.

\begin{proposition}\label{equiv hered y n-cot}
Let $\mathcal{A}$ and $\mathcal{B}$ be two classes of objects of $\mathcal{C}$, and $n \geq 2$ be a positive integer. Consider the following conditions.
\begin{enumerate}
\item[(a)] $(\mathcal{A,B})$ is an $n$-cotorsion pair in $\mathcal{C}$ and $\mathcal{B}$ is a coresolving class.

\item[(b)] $(\mathcal{A,B})$ is an $n$-cotorsion pair in $\mathcal{C}$ and $\mathcal{A}$ is a resolving class.

\item[(c)] $(\mathcal{A,B})$ is a hereditary complete cotorsion pair in $\mathcal{C}$.
\end{enumerate}
Then, conditions (a) and (b) are equivalent, and any of them implies (c). If in addition, $\mathcal{C}$ has enough projectives or injectives, then (c) also implies (a) and (b).   
\end{proposition}

\begin{proof}
We first show that (a) and (b) are equivalent. Suppose that condition (a) holds, and so $(\mathcal{A,B})$ is an $n$-cotorsion pair with $\mathcal{B}$ coresolving. By Theorem~\ref{theo:left-n-cotorsion}, $(\mathcal{A},\mathcal{B}^\wedge_{n-1})$ is a complete left cotorsion pair, where $\mathcal{B}^\wedge_{n-1} = \mathcal{B}$ since $\mathcal{B}$ is coresolving. It follows that $\mathcal{A} = {}^{\perp_1}\mathcal{B}$, and thus it is clear that $\mathcal{P}(\mathcal{C}) \subseteq \mathcal{A}$ and that $\mathcal{A}$ is closed under extensions. To show that $\mathcal{A}$ is also closed under taking kernels of epimorphisms between objects of $\mathcal{A}$, suppose we are given a short exact sequence
\[
0 \to A_1 \to A_2 \to A_3 \to 0
\]
with $A_2, A_3 \in \mathcal{A}$. For every $B \in \mathcal{B}$, we have an exact sequence
\[
\Ext^1_{\mathcal{C}}(A_2,B) \to \Ext^1_{\mathcal{C}}(A_1,B) \to \Ext^2_{\mathcal{C}}(A_3,B),
\]
where $\Ext^1_{\mathcal{C}}(A_2,B) = 0$ and $\Ext^2_{\mathcal{C}}(A_3,B) = 0$ since $(\mathcal{A,B})$ is an $n$-cotorsion pair with $n \geq 2$. It follows that $A_1 \in {}^{\perp_1}\mathcal{B} = \mathcal{A}$. Hence, $\mathcal{A}$ is resolving. The implication (b) $\Rightarrow$ (a) follows similarly. 

Note that if we assume (a) or (b), we obtain a complete left and a complete right cotorsion pair $(\mathcal{A,B})$, that is, a complete cotorsion pair in $\mathcal{C}$, which is also hereditary. 

Now if we suppose that (c) holds and that $\mathcal{C}$ has enough projectives or injectives. One can see that this implies that $\Ext^i_{\mathcal{C}}(\mathcal{A,B}) = 0$ holds for every $i \geq 1$, and showing hence conditions (a) and (b). 
\end{proof}

Note that the previous proposition basically says that there is no much hope in finding an $n$-cotorsion pair $(\mathcal{A,B})$ in a Grothendieck category with $\mathcal{A}$ resolving (and $\mathcal{B}$ coresolving) that is not actually a hereditary complete cotorsion pair. In Section~\ref{sec:applications}, we shall find some examples of $n$-cotorsion which do not come from such cotorsion pairs.

A first application of Proposition~\ref{equiv hered y n-cot} has to do with extending the result \cite[Proposition 6.7]{MaoDing_finite} about the existence of absolutely pure covers with the unique mapping property. For a precise statement, we need to recall some concepts.

Let $R$ be a ring and $M \in \Mod(R)$ be an $R$-module. Recall that $M$ is \emph{absolutely pure} (or \emph{FP-injective}) if $\Ext^1_R(F,M) = 0$ for every finitely presented module $F \in \Mod(R)$. In what follows, we shall denote by $\mathcal{AP}(R)$ the class of absolutely pure $R$-modules. The \emph{absolutely pure dimension} of $M$, denoted ${\rm apd}(M)$, is defined as the smallest nonnegative integer $m \geq 0$ such that $\Ext^{m+1}_R(F,M) = 0$ for every finitely presented module $F \in \Mod(R)$, or equivalently, as 
\[
{\rm apd}(M) = {\rm coresdim}_{\mathcal{AP}(R)}(M).
\] 
The (\emph{left}) \emph{global absolutely pure dimension} of $R$, denoted by ${\rm gl.APD}(R)$, is defined as the supremum 
\[
{\rm gl.APD}(R) := \sup \{ {\rm apd}(M) \mbox{ : } M \in \Mod(R) \}. 
\]
The dual concepts are the \emph{flat dimension} of $M$ and the \emph{weak global dimension} of $R$, denoted by ${\rm fd}(M)$ and ${\rm wd}(R)$, respectively.

\begin{corollary}
For any ring $R$ with ${\rm gl.APD}(R^{\rm op}) \leq 2$, the following statements are equivalent:
\begin{itemize}
\item[(a)] $R$ is a right coherent ring.

\item[(b)] Every right $R$-module has an absolutely pure cover with the unique mapping property.
\end{itemize}
Moreover, if any of above conditions holds true, then ${\rm wd}(R) \leq 2$ and the class $\mathcal{AP}(R\op)$ is closed under direct limits.
\end{corollary}

\begin{proof}
On the one hand, if $R$ is a right coherent ring, then from \cite[Proposition 3.6]{MaoDing_finite} and \cite[Corollary 2.7]{Pinzon}, we have that $({}^{\perp_{1}}\mathcal{AP}(R\op),\mathcal{AP}(R\op))$ is a complete hereditary cotorsion pair in $\Mod(R^{\rm op})$ and that every right $R$-module has an absolutely pure cover. Thus, by Proposition~\ref{equiv hered y n-cot} and the dual of Corollary~\ref{corUMP2} we get the implication (a) $\Rightarrow$ (b). The converse, on the other hand, holds true by \cite[Remark 6.8]{MaoDing_finite}. Finally, the second part follows by \cite[Theorem 6.6 and Proposition 6.7]{MaoDing_finite}.
\end{proof}

\begin{corollary}
For any ring $R,$ the following conditions are equivalent:
\begin{itemize}
\item[(a)] Every $R$-module has a flat envelope with the unique mapping property.

\item[(b)] Every $R$-module has a flat envelope and ${\rm wd}(R)\leq 2$. 

\item[(c)] $R$ is a right coherent ring with ${\rm wd}(R)\leq 2$.
\end{itemize}
\end{corollary}

\begin{proof}
From \cite[Theorem 3.4]{MaoDing_finite}, Proposition~\ref{equiv hered y n-cot} and Corollary~\ref{corUMP2}, we have the equivalence (a) $\Leftrightarrow$ (b), while (a) $\Leftrightarrow$ (c) follows by \cite[Proposition 2.1]{AsensioMartinez}.
\end{proof}

Another application of Proposition~\ref{equiv hered y n-cot}, along with Proposition \ref{ncotder y cotder} and its dual, has to do with descriptions for the classes $\mathcal{A}^\wedge_n$ and $\mathcal{B}^\vee_m$ coming from a hereditary complete cotorsion pair $(\mathcal{A,B})$.

\begin{corollary}
Let $(\mathcal{A,B})$ be a hereditary complete cotorsion pair in an abelian category $\mathcal{C}$ with enough projectives and injectives. Then, for any pair of integers $m, n \geq 0$, the equalities $\mathcal{A}^{\perp_{m + 1}} = \mathcal{B}^\vee_m$ and ${}^{\perp_{n+1}}\mathcal{B} = \mathcal{A}^\wedge_n$ hold true.
\end{corollary}

Now let us focus on $n$-cotorsion pairs satisfying the condition $\mathcal{A} \subseteq \mathcal{B}$.

\begin{proposition}\label{Asubseteq B en ncot}
Let $(\mathcal{A,B})$ be an $n$-cotorsion pair in $\mathcal{C}$. Then, the following assertions are equivalent.
\begin{itemize}
\item[(a)] $\mathcal{A} \subseteq \mathcal{B}$.

\item[(b)] $\mathcal{C} = \mathcal{B}_n^{\wedge}$.

\item[(c)] $\Ext_{\mathcal{C}}^1(\mathcal{A}_{n-1}^{\vee},\mathcal{A}) = 0$.
\end{itemize}
\end{proposition}
\begin{proof}
We first show (a) $\Leftrightarrow$ (b). The implication (a) $\Rightarrow$ (b) follows by Definition~\ref{def:ncotorsion} (3). Now suppose that condition (b) $\mathcal{C} = \mathcal{B}_{n}^{\wedge}$ holds true, and let $A \in \mathcal{A}$. Then, $A$ is isomorphic to the image of an epimorphism from $\mathcal{B}$ with kernel in $\mathcal{B}_{n-1}^{\wedge}$. By Proposition~\ref{prop8}, this epimorphism splits, and therefore $A \in \mathcal{B}$ since $\mathcal{B}$ is closed under direct summands. Finally, the equivalence (a) $\Leftrightarrow$ (c) follows by the dual of Theorem~\ref{theo:left-n-cotorsion}.
\end{proof}

We can use the previous proposition to show the following generalisation of Akinci and Alizade's \cite[Remark 2.4]{Akinci}.

\begin{corollary}\label{coro:cot-hered-completo-trivial}  
For every hereditary complete cotorsion pair $(\mathcal{A,B})$ in an abelian category  $\mathcal{C}$, the following assertions are equivalent.
\begin{itemize}
\item[(a)] $\mathcal{A} \subseteq \mathcal{B}.$

\item[(b)] $\mathcal{C} = \mathcal{B}$.

\item[(c)] $\Ext_{\mathcal{C}}^{1}(\mathcal{A,A}) = 0$.

\item[(d)] $\mathcal{A} = \mathcal{P}(\mathcal{C})$.
\end{itemize}
\end{corollary}

\begin{proof} 
Note that $(\A,\B)$ is $1$-cotorsion pair in $\C$ with $\B^\wedge_1 = \B$ and $\A^\vee_1 = \A$. Thus, by Proposition~\ref{Asubseteq B en ncot} we get the equivalences between (a), (b) and (c). The equivalence between (b) and (d), on the other hand, is straightforward. 
\end{proof}

\begin{remark}\label{Rk-n-cot-trivial} \
\begin{enumerate}
\item  Let $(\mathcal{A,B})$ be an $n$-cotorsion pair in an abelian category $\mathcal{C},$ with $n\geq 2$ and $\mathcal{A}$ resolving (or $\mathcal{B}$ coresolving). Then, by Proposition~\ref{equiv hered y n-cot} we know that $(\A,\B)$ is an  hereditary complete cotorsion pair in $\C$. Thus, by Corollary~\ref{coro:cot-hered-completo-trivial}, we get the equivalences
\[
\mathcal{A} \subseteq \mathcal{B}\;\Leftrightarrow\;\mathcal{C} = \mathcal{B}\;\Leftrightarrow\;\Ext_{\mathcal{C}}^{1}(\mathcal{A,A}) = 0\;\Leftrightarrow\;\mathcal{A} = \mathcal{P}(\mathcal{C}).
\]

\item The previous implies that, in an abelian category $\C$ with enough projectives and injectives, the only $n$-cotorsion pair $(\mathcal{A,B})$ with $\mathcal{A}$ resolving and satisfying the continment $\mathcal{A} \subseteq \mathcal{B}$ is the trivial $n$-cotorsion pair $(\mathcal{P}(\mathcal{C}),\mathcal{C})$. 
\end{enumerate}
\end{remark}

Now we are ready to propose the following definition of hereditary unilateral $n$-cotorsion.

\begin{definition}\label{def:hereditary_n-cotorsion}
Let $(\mathcal{A,B})$ be a left $n$-cotorsion pair in $\mathcal{C}$. We say that $(\mathcal{A,B})$ is \textbf{hereditary} if $\Ext^{n+1}_{\mathcal{C}}(\mathcal{A,B}) = 0$. Hereditary right $n$-cotorsion pairs are defined in the same way. 
\end{definition}

One can note the following property of left $n$-cotorsion pairs.

\begin{proposition}\label{prop:hereditary_n-cotorsion}
If $(\mathcal{A,B})$ is a hereditary left $n$-cotorsion pair in $\mathcal{C}$, then the class $\mathcal{A}$ is resolving. 
\end{proposition}

\begin{proof}
Notice by Theorem~\ref{theo:left-n-cotorsion} that $\mathcal{A} = \bigcap^n_{i = 1} {}^{\perp_i}\mathcal{B}$. Then, it is clear that $\mathcal{A}$ is closed under extensions and that $\mathcal{P}(\mathcal{C}) \subseteq \mathcal{A}$. Finally, if we are given a short exact sequence
\[
0 \to A_1 \to A_2 \to A_3 \to 0
\]
in $\mathcal{C}$ with $A_2, A_3 \in \mathcal{A}$, then we have an exact sequence
\[
\Ext^i_{\mathcal{C}}(A_2, B) \to \Ext^i_{\mathcal{C}}(A_1, B) \to \Ext^{i+1}_{\mathcal{C}}(A_3,B)
\]
of abelian groups where $\Ext^i_{\mathcal{C}}(A_2,B) = 0$ for every $1 \leq i \leq n$, and $\Ext^{i+1}_{\mathcal{C}}(A_3, B) = 0$ for every $1 \leq i \leq n$ since $(\mathcal{A,B})$ is hereditary. It follows that $\Ext^i_{\mathcal{C}}(A_1, B) = 0$ for every $1 \leq i \leq n$ and $B \in \mathcal{B}$, that is, $A_1 \in \cap^n_{i = 1} {}^{\perp_i}\mathcal{B} = \mathcal{A}$, and so $\mathcal{A}$ is closed under taking kernels of epimorphisms between its objects. 
\end{proof}

\begin{remark}
Notice by the previous proposition and its dual, that is $(\mathcal{A,B})$ is a hereditary (left and right) $n$-cotorsion pair, then $\mathcal{A}$ is resolving and $\mathcal{B}$ is coresolving. Then, Theorem~\ref{theo:left-n-cotorsion} and its dual imply that $\mathcal{A} = {}^{\perp_1}(\mathcal{B}^\wedge_{n-1})$ and $\mathcal{B} = (\mathcal{A}^\vee_{n-1})^{\perp_1}$, where $\mathcal{B}^\wedge_{n-1} = \mathcal{B}$ and $\mathcal{A}^\vee_{n-1} = \mathcal{A}$. Hence, $(\mathcal{A,B})$ is a hereditary complete cotorsion pair in $\mathcal{C}$. The converse of this fact is clearly true in the case where $\mathcal{C}$ has enough projective and injectives. 
\end{remark}

\begin{theorem}\label{theo:compatible}
Let $n \geq 2$ be a positive integer and $(\mathcal{A,B})$ be a hereditary left $n$-cotorsion pair in $\mathcal{C}$ with $\mathcal{B}$ closed under extensions and such that $\Ext^1_{\mathcal{C}}(\mathcal{B},\mathcal{B}^\wedge_{n-1}) = 0$. Let 
\[
\eta \colon 0 \to X \xrightarrow{f} Y \xrightarrow{g} Z \to 0
\]
be a short exact sequence in $\mathcal{C}$ where $X$ and $Z$ have special $(\mathcal{A},n,\mathcal{B})$-precovers $\rho_X$ and $\rho_Z$ as in Remark~\ref{rem:special_n-1}. Then, $Y$ has a special $(\mathcal{A},n,\mathcal{B})$-precover $\rho_Y$ compatible with $\rho_X$ and $\rho_Z$ in the sense that it fits into a commutative diagram as \eqref{fig_compatible}. In particular, the class ${\rm Prec}^n(\mathcal{A,B})$ is closed under extensions. 
\end{theorem}

\begin{proof}
For a better understanding of the arguments below, we only show the case $n = 2$. The more general cases $n > 2$ follow inductively. 

We are given two short exact sequences 
\begin{align*}
0 & \to K^X \xrightarrow{i^X} A^X \xrightarrow{\alpha^X} X \to 0, \\
0 & \to K^Z \xrightarrow{i^Z} A^Z \xrightarrow{\alpha^Z} Z \to 0,
\end{align*}
where $A^X, A^Z \in \mathcal{A}$ and $K^X, K^Z \in \mathcal{B}^\wedge_1$, that is, we also have two short exact sequences
\begin{align*}
0 & \to B^X_1 \xrightarrow{\beta^X_1} B^X_0 \xrightarrow{\gamma^X} K^X \to 0, \numberthis \label{eqn:BX1BX0KX} \\
0 & \to B^Z_1 \xrightarrow{\beta^Z_1} B^Z_0 \xrightarrow{\gamma^Z} K^Z \to 0, 
\end{align*}
with $B^X_0, B^Z_0, B^X_1, B^Z_1 \in \mathcal{B}$. Following the same arguments as in the proof of Theorem \ref{Teo-k-prec}, we can find the following commutative diagrams with exact rows and columns:
\begin{equation}\label{fig7} 
\parbox{2.5in}{
\begin{tikzpicture}[description/.style={fill=white,inner sep=2pt}] 
\matrix (m) [ampersand replacement=\&, matrix of math nodes, row sep=2.5em, column sep=2.5em, text height=1.25ex, text depth=0.25ex] 
{ 
{} \& {} \& X \& X \\
B^Z_1 \& B^Z_0 \& Y' \& Y \\
B^Z_1 \& B^Z_0 \& A^Z \& Z \\
}; 
\path[->]
(m-2-3)-- node[pos=0.5] {\footnotesize$\mbox{\bf pb}$} (m-3-4)
(m-2-2) edge node[above] {\footnotesize$\hat{\beta}^Z_0$} (m-2-3)
(m-3-2) edge node[below] {\footnotesize$\beta^Z_0 $} (m-3-3)
;
\path[>->]
(m-1-3) edge node[left] {\footnotesize$f'$} (m-2-3) (m-1-4) edge node[left] {\footnotesize$f$} (m-2-4)
(m-2-1) edge node[above] {\footnotesize$\hat{\beta}^Z_1$} (m-2-2) 
(m-3-1) edge node[below] {\footnotesize$\beta^Z_1$} (m-3-2)
;
\path[->>]
(m-2-3) edge node[left] {\footnotesize$g'$} (m-3-3) (m-2-4) edge node[left] {\footnotesize$g$} (m-3-4)
(m-2-3) edge node[above] {\footnotesize$\tilde{\alpha}^Z$} (m-2-4)
(m-3-3) edge node[below] {\footnotesize$\alpha^Z$} (m-3-4)
;
\path[-,font=\scriptsize]
(m-1-3) edge [double, thick, double distance=2pt] (m-1-4)
(m-2-1) edge [double, thick, double distance=2pt] (m-3-1)
(m-2-2) edge [double, thick, double distance=2pt] (m-3-2)
;
\end{tikzpicture} 
}
\end{equation} 
\hfill(where $\beta^Z_0 := i^Z \circ \gamma^Z$)

\begin{equation}\label{fig8} 
\parbox{1.75in}{
\begin{tikzpicture}[description/.style={fill=white,inner sep=2pt}] 
\matrix (m) [ampersand replacement=\&, matrix of math nodes, row sep=2.5em, column sep=2.5em, text height=1.25ex, text depth=0.25ex] 
{ 
K^X \& K^X \& {} \\
A^X \& A^Y \& A^Z \\
X \& Y' \& A^Z \\
}; 
\path[->] 
(m-2-1)-- node[pos=0.5] {\footnotesize$\mbox{\bf po}$} (m-3-2)
; 
\path[>->]
(m-1-1) edge node[left] {\footnotesize$i^X$} (m-2-1) 
(m-1-2) edge node[right] {\footnotesize$\overline{i}^X$} (m-2-2)
(m-2-1) edge node[above] {\footnotesize$\hat{f}$} (m-2-2) 
(m-3-1) edge node[below] {\footnotesize$f'$} (m-3-2)
;
\path[->>]
(m-2-1) edge node[left] {\footnotesize$\alpha^X$} (m-3-1) 
(m-2-2) edge node[right] {\footnotesize$\overline{\alpha}^X$} (m-3-2) 
(m-2-2) edge node[above] {\footnotesize$\hat{g}$} (m-2-3) 
(m-3-2) edge node[below] {\footnotesize$g'$} (m-3-3)
;
\path[-,font=\scriptsize]
(m-1-1) edge [double, thick, double distance=2pt] (m-1-2)
(m-2-3) edge [double, thick, double distance=2pt] (m-3-3)
;
\end{tikzpicture} 
}
\end{equation} 
\hfill (where $A^Y \in \mathcal{A}$ since $\mathcal{A}$ is closed under extensions)

\begin{equation}\label{fig9} 
\parbox{2.5in}{
\begin{tikzpicture}[description/.style={fill=white,inner sep=2pt}] 
\matrix (m) [ampersand replacement=\&, matrix of math nodes, row sep=2.5em, column sep=2.5em, text height=1.25ex, text depth=0.25ex] 
{ 
{} \& K^X \& K^X \& {} \\
B^Z_1 \& Q \& A^Y \& Y \\
B^Z_1 \& B^Z_0 \& Y' \& Y \\
}; 
\path[->] 
(m-2-2)-- node[pos=0.5] {\footnotesize$\mbox{\bf pb}$} (m-3-3)
(m-2-2) edge (m-2-3)
(m-3-2) edge node[below] {\footnotesize$\hat{\beta}^Z_0$} (m-3-3)
; 
\path[>->]
(m-1-2) edge (m-2-2)
(m-1-3) edge node[right] {\footnotesize$\overline{i}^X$} (m-2-3)
(m-2-1) edge (m-2-2)
(m-3-1) edge node[below] {\footnotesize$\hat{\beta}^Z_1$} (m-3-2)
;
\path[->>]
(m-2-3) edge node[above] {\footnotesize$\alpha^Y$} (m-2-4)
(m-3-3) edge node[below] {\footnotesize$\tilde{\alpha}^Z$} (m-3-4)
(m-2-2) edge (m-3-2) 
(m-2-3) edge node[right] {\footnotesize$\overline{\alpha}^X$} (m-3-3)
;
\path[-,font=\scriptsize]
(m-1-2) edge [double, thick, double distance=2pt] (m-1-3)
(m-2-1) edge [double, thick, double distance=2pt] (m-3-1)
(m-2-4) edge [double, thick, double distance=2pt] (m-3-4)
;
\end{tikzpicture} 
}
\end{equation} 
Since $B^Z_0 \in \mathcal{B}$, $K^X \in \mathcal{B}^\wedge_1$ and $\Ext^1_{\mathcal{C}}(\mathcal{B},\mathcal{B}^\wedge_1) = 0$, the column $K^X \rightarrowtail Q \twoheadrightarrow B^Z_0$ is split exact and so $Q \simeq B^Z_0 \oplus K^X$. Thus, the diagram \eqref{fig9} can be rewritten as:
\begin{equation}\label{fig9mod} 
\parbox{3.5in}{
\begin{tikzpicture}[description/.style={fill=white,inner sep=2pt}] 
\matrix (m) [ampersand replacement=\&, matrix of math nodes, row sep=3.5em, column sep=4em, text height=1.25ex, text depth=0.25ex] 
{ 
{} \& K^X \& K^X \& {} \\
B^Z_1 \& B^Z_0 \oplus K^X \& A^Y \& Y \\
B^Z_1 \& B^Z_0 \& Y' \& Y \\
}; 
\path[->] 
(m-2-2) edge node[above] {\scriptsize$\left( \begin{array}{cc} a & \overline{i}^X \end{array} \right)$} (m-2-3)
(m-3-2) edge node[below] {\footnotesize$\hat{\beta}^Z_0$} (m-3-3)
; 
\path[>->]
(m-1-2) edge node[left] {\scriptsize$\left( \begin{array}{c} 0 \\ {\rm id}_{K^X} \end{array} \right)$} (m-2-2)
(m-1-3) edge node[right] {\footnotesize$\overline{i}^X$} (m-2-3)
(m-2-1) edge node[below] {\scriptsize$\left( \begin{array}{c} \beta^Z_1 \\ 0 \end{array} \right)$} (m-2-2)
(m-3-1) edge node[below] {\scriptsize$\hat{\beta}^Z_1$} (m-3-2)
;
\path[->>]
(m-2-3) edge node[above] {\footnotesize$\alpha^Y$} (m-2-4)
(m-3-3) edge node[below] {\footnotesize$\tilde{\alpha}^Z$} (m-3-4)
(m-2-2) edge node[right] {\scriptsize$\left( \begin{array}{cc} {\rm id}_{B^Z_0} & 0 \end{array} \right)$} (m-3-2) 
(m-2-3) edge node[right] {\scriptsize$\overline{\alpha}^X$} (m-3-3)
;
\path[-,font=\scriptsize]
(m-1-2) edge [double, thick, double distance=2pt] (m-1-3)
(m-2-1) edge [double, thick, double distance=2pt] (m-3-1)
(m-2-4) edge [double, thick, double distance=2pt] (m-3-4)
;
\end{tikzpicture} 
}
\end{equation} 
The existence of the arrow $a \colon B^Z_0 \to A^Y$ is a consequence of the pullback construction, and satisfies the relations 
\begin{align}\label{eqn:arrow_a}
\overline{\alpha}^X \circ a & = \hat{\beta}^Z_0.
\end{align}
Moreover, following the arguments that show the commutativity of the diagram \eqref{fig_caso_n1} in Remark \ref{rem:special_n-1}, we have the following commutative diagram with exact rows and columns: 
\begin{equation}\label{fig_kernel} 
\parbox{1.75in}{
\begin{tikzpicture}[description/.style={fill=white,inner sep=2pt}] 
\matrix (m) [ampersand replacement=\&, matrix of math nodes, row sep=2.5em, column sep=2.5em, text height=1.25ex, text depth=0.25ex] 
{ 
K^X \& K^Y \& K^Z \\
A^X \& A^Y \& A^Z \\
X \& Y \& Z \\
}; 
\path[>->]
(m-1-1) edge node[left] {\footnotesize$i^X$} (m-2-1) 
(m-1-2) edge node[left] {\footnotesize$i^Y$} (m-2-2)
(m-1-3) edge node[left] {\footnotesize$i^Z$} (m-2-3)
(m-1-1) edge node[above] {\footnotesize$\tilde{f}$} (m-1-2)
(m-2-1) edge node[above] {\footnotesize$\hat{f}$} (m-2-2) 
(m-3-1) edge node[above] {\footnotesize$f$} (m-3-2)
;
\path[->>]
(m-2-1) edge node[left] {\footnotesize$\alpha^X$} (m-3-1) 
(m-2-2) edge node[left] {\footnotesize$\alpha^Y$} (m-3-2) 
(m-2-3) edge node[left] {\footnotesize$\alpha^Z$} (m-3-3)
(m-1-2) edge node[above] {\footnotesize$\tilde{g}$} (m-1-3)
(m-2-2) edge node[above] {\footnotesize$\hat{g}$} (m-2-3) 
(m-3-2) edge node[above] {\footnotesize$g$} (m-3-3)
;
\end{tikzpicture} 
}
\end{equation}

On the other hand, let us add the identity ${\rm id}_{B^Z_0}$ to the sequence \eqref{eqn:BX1BX0KX}, in order to obtain the exact sequence
\[
0 \to B^X_1 \xrightarrow{{\scriptsize\left( \begin{array}{c} 0 \\ \beta^X_1 \end{array} \right)}} B^Z_0 \oplus B^X_0 \xrightarrow{\scriptsize{\left( \begin{array}{cc} {\rm id}_{B^Z_0} & 0 \\ 0 & \gamma^X \end{array} \right)}} B^Z_0 \oplus K^X \to 0.
\]
Now take the pullback of $B^Z_1 \to B^Z_0 \oplus K^X \leftarrow B^Z_0 \oplus B^X_0$ in order to obtain the following commutative diagram with exact rows and columns:
\begin{equation}\label{fig10} 
\parbox{2.5in}{
\begin{tikzpicture}[description/.style={fill=white,inner sep=2pt}] 
\matrix (m) [ampersand replacement=\&, matrix of math nodes, row sep=4em, column sep=5.5em, text height=1.25ex, text depth=0.25ex] 
{ 
B^X_1 \& B^X_1 \& {} \\
B^Y_1 \& B^Z_0 \oplus B^X_0 \& K^Y \\
B^Z_1 \& B^Z_0 \oplus K^X \& K^Y \\
}; 
\path[->] 
(m-2-1)-- node[pos=0.5] {\footnotesize$\mbox{\bf pb}$} (m-3-2)
; 
\path[>->]
(m-1-1) edge node[left] {\footnotesize$f_1$} (m-2-1) 
(m-1-2) edge node[right] {\scriptsize$\left( \begin{array}{c} 0 \\ \beta^X_1 \end{array} \right)$} (m-2-2)
(m-2-1) edge node[above] {\scriptsize$\left( \begin{array}{c} \beta^Z_1 \circ g_1 \\ b \end{array} \right)$} (m-2-2) 
(m-3-1) edge node[below] {\scriptsize$\left( \begin{array}{c} \beta^Z_1 \\ 0 \end{array} \right)$} (m-3-2)
;
\path[->>]
(m-2-1) edge node[left] {\footnotesize$g_1$} (m-3-1) 
(m-2-2) edge node[right] {\scriptsize$\left( \begin{array}{cc} {\rm id}_{B^Z_0} & 0 \\ 0 & \gamma^X \end{array} \right)$} (m-3-2) 
(m-2-2) edge node[above] {\scriptsize$\left( \begin{array}{cc} c & \tilde{f} \circ \gamma^X \end{array} \right)$} (m-2-3) 
(m-3-2) edge node[below] {\scriptsize$\left( \begin{array}{cc} c & \tilde{f} \end{array} \right)$} (m-3-3)
;
\path[-,font=\scriptsize]
(m-1-1) edge [double, thick, double distance=2pt] (m-1-2)
(m-2-3) edge [double, thick, double distance=2pt] (m-3-3)
;
\end{tikzpicture} 
}
\end{equation} 
where $b \colon B^Y_1 \to B^X_0$ and $c \colon B^Z_0 \to K^Y$ are arrows given by te pullback construction that satisfy the following relations:
\begin{align}
b \circ f_1 & = \beta^X_1, \label{eqn:arrow_b} \\
c \circ \beta^Z_1 & = 0. \label{eqn:arrow_c}
\end{align}
Finally, we form the following diagram with exact rows and columns: 
\begin{equation}\label{fig11} 
\parbox{4.25in}{
\begin{tikzpicture}[description/.style={fill=white,inner sep=2pt}] 
\matrix (m) [ampersand replacement=\&, matrix of math nodes, row sep=4em, column sep=5.5em, text height=1.25ex, text depth=0.25ex] 
{ 
B^X_1 \& B^X_0 \& A^X \& X \\
B^Y_1 \& B^Z_0 \oplus B^X_0 \& A^Y \& Y \\
B^Z_1 \& B^Z_0 \& A^Z \& Z \\
}; 
\path[->] 
(m-1-2) edge node[above] {\footnotesize$\beta^X_0$} (m-1-3)
(m-2-2) edge node[below] {\scriptsize$\left( \begin{array}{cc} a & \hat{f} \circ \beta^X_0 \end{array} \right)$} (m-2-3)
(m-3-2) edge node[below] {\footnotesize$\beta^Z_0$} (m-3-3)
; 
\path[>->]
(m-1-1)-- node[pos=0.25] {\scriptsize$\circled{1}$} (m-2-2)
(m-1-2)-- node[pos=0.75] {\scriptsize$\circled{2}$} (m-2-3)
(m-1-3)-- node[pos=0.5] {\scriptsize$\circled{3}$} (m-2-4)
(m-2-1)-- node[pos=0.15] {\scriptsize$\circled{4}$} (m-3-2)
(m-2-2)-- node[pos=0.75] {\scriptsize$\circled{5}$} (m-3-3)
(m-2-3)-- node[pos=0.5] {\scriptsize$\circled{6}$} (m-3-4)
(m-1-1) edge node[above] {\footnotesize$\beta^X_1$} (m-1-2)
(m-2-1) edge node[above] {\scriptsize$\left( \begin{array}{c} \beta^Z_1 \circ g_1 \\ b \end{array} \right)$} (m-2-2)
(m-3-1) edge node[below] {\footnotesize$\beta^Z_1$} (m-3-2)
(m-1-1) edge node[left] {\footnotesize$f_1$} (m-2-1)
(m-1-2) edge node[right] {\scriptsize$\left( \begin{array}{c} 0 \\ {\rm id}_{B^X_0} \end{array} \right)$} (m-2-2)
(m-1-3) edge node[left] {\footnotesize$\hat{f}$} (m-2-3)
(m-1-4) edge node[left] {\footnotesize$f$} (m-2-4)
;
\path[->>]
(m-1-3) edge node[above] {\footnotesize$\alpha^X$} (m-1-4)
(m-2-3) edge node[above] {\footnotesize$\alpha^Y$} (m-2-4)
(m-3-3) edge node[below] {\footnotesize$\alpha^Z$} (m-3-4)
(m-2-1) edge node[left] {\footnotesize$g_1$} (m-3-1)
(m-2-2) edge node[left] {\footnotesize$\left( \begin{array}{cc} {\rm id}_{B^Z_0} & 0 \end{array} \right)$} (m-3-2)
(m-2-3) edge node[left] {\footnotesize$\hat{g}$} (m-3-3)
(m-2-4) edge node[left] {\footnotesize$g$} (m-3-4)
;
\end{tikzpicture} 
}
\end{equation} 
where $\beta^X_0 := i^X \circ \gamma^X$. The commutativity of squares $\scriptsize\circled{3}$ and $\scriptsize\circled{6}$ was already verified for the diagram \eqref{fig_kernel}, while for $\scriptsize\circled{2}$ and $\scriptsize\circled{4}$ is clear. We check that the remaining squares also commute:
\begin{itemize}
\item[\scriptsize$\circled{1}$] We have by the first equality in \eqref{eqn:arrow_b} that
\[
\left( \begin{array}{c} \beta^Z_1 \circ g_1 \\ b \end{array} \right) \circ f_1 = \left( \begin{array}{c} \beta^Z_1 \circ g_1 \circ f_1 \\ b \circ f_1 \end{array} \right) = \left( \begin{array}{c} 0 \\ \beta^X_1 \end{array} \right) = \left( \begin{array}{c} 0 \\ {\rm id}_{B^X_0} \end{array} \right) \circ \beta^X_1.
\]

\item[\scriptsize$\circled{5}$] Using the diagram \eqref{fig8}, we have that $\hat{g} = g' \circ \overline{\alpha}^X$, and so $\hat{g} \circ a = g' \circ \overline{\alpha}^X \circ a$, where $\overline{\alpha}^X \circ a = \hat{\beta}^Z_0$ by \eqref{eqn:arrow_a} and $g' \circ \hat{\beta}^Z_0 = \beta^Z_0$ by \eqref{fig7}. Thus,
\begin{align*}
\hat{g} \circ \left( \begin{array}{cc} a & \hat{f} \circ \beta^X_0 \end{array} \right) & = \left( \begin{array}{cc} \hat{g} \circ a & \hat{g} \circ \hat{f} \circ \beta^X_0 \end{array} \right) = \left( \begin{array}{cc} g' \circ \hat{\beta}^Z_0 & 0 \end{array} \right) = \left( \begin{array}{cc} \beta^Z_0 & 0 \end{array} \right) = \beta^Z_0 \circ \left( \begin{array}{cc} {\rm id}_{B^Z_0} & 0 \end{array} \right).
\end{align*}
\end{itemize}
Note that $B^Z_0 \oplus B^X_0, B^Y_1 \in \mathcal{B}$ since $\mathcal{B}$ is closed under extensions. Therefore, the central row of \eqref{fig11} defines a special $(\mathcal{A},2,\mathcal{B})$-precover compatible with $\rho_X$ and $\rho_Z$.
\end{proof}

\begin{remark}
One can include the case $n = 1$ in the previous theorem, for which the hypothesis $\Ext^1_{\mathcal{C}}(\mathcal{B,B}) = 0$ is not needed (see Remark~\ref{rem:special_n-1}).
\end{remark}


\section{\textbf{Applications and examples}}\label{sec:applications}

Below we present some examples of (left and right) $n$-cotorsion pairs along with some applications, which are related to the characterisation of certain rings as well as to finding covers and envelopes with the unique mapping property. 

We need to mention a couple of considerations. Let us denote the projective and injective dimensions of an object $C$ in an abelian category $\mathcal{C}$ by $\pd(C)$ and $\id(C)$, respectively. Recall that $\pd(C)$ is defined as the smallest nonnegative integer $m \geq 0$ such that $\Ext^i_{\mathcal{C}}(C,\mathcal{C}) = 0$ for every $i > m$. If such $m$ does not exist, one sets $\pd(C) := \infty$. Note that if $\mathcal{C}$ has enough projectives, then $\pd(C)$ coincides with the $\mathcal{P}(\mathcal{C})$-resolution dimension of $C$. Similarly $\id(C)$, defined dually, coincides with the $\mathcal{I}(\mathcal{C})$-coresolution dimension of $C$ in the case where $\mathcal{C}$ has enough injectives. For simplicity, if $\mathcal{C} = \Mod(R)$, we write the classes $\mathcal{P}(\Mod(R))$ and $\mathcal{I}(\Mod(R))$ as $\mathcal{P}(R)$ and $\mathcal{I}(R)$, respectively.


\subsection*{\textbf{Gorenstein projective modules}}

Recall the classes $\mathcal{GP}(R)$ and $\mathcal{GI}(R)$ of Gorenstein projective and Gorenstein injective $R$-modules from Example~\ref{ex:special_AkB-precover}. Over an arbitrary ring $R$, it is well known by \cite[Theorem 2.5]{Holm} that $\mathcal{GP}(R)$ is closed under direct summands. On the other hand, \cite[Proposition 2.3]{Holm} asserts that $\Ext^i_R(C,P) = 0$ for every $C \in \mathcal{GP}(R)$, $P \in \mathcal{P}(R)$ and $i \geq 1$. So $(\mathcal{GP}(R),\mathcal{P}(R))$ is a left $n$-cotorsion pair in $\mathsf{Mod}(R)$ if, and only if, every module has a Gorenstein projective special precover whose kernel has projective dimension at most $n-1$. The most obvious choice of a ring $R$ over which the latter condition holds, is when $R$ is an $n$-Iwanaga-Gorenstein ring, that is, $R$ is two-sided noetherian with $\id({}_R R) = \id(R_R) = n$. Over such rings $R$, it is known that every module has Gorenstein projective dimension at most $n$. Therefore, we have the following example of a left $n$-cotorsion pair, which is also a consequence of Proposition~\ref{equiv hered y n-cot}, Hovey's \cite[Theorem 8.3]{Hovey} and Holm's \cite[Theorem 2.5]{Holm}.

\begin{example}\label{ex:GProj_ncot}
Let $R$ be an $n$-Iwanaga-Gorenstein ring with $n\geq 1$. Then, $(\mathcal{GP}(R),\mathcal{P}(R))$ is a left $n$-cotorsion pair and $(\mathcal{I}(R),\mathcal{GI}(R))$ is a right $n$-cotorsion pair in $\mathsf{Mod}(R)$. 

For the case $n = 0$, a $0$-Iwanaga-Gorenstein ring is just a \textbf{quasi-Frobenius ring} (or \textbf{QF ring}, for short) by Bland's \cite[Proposition 10.2.14]{Bland}. Moreover, $\mathcal{P}(R) = \mathcal{I}(R)$ and $\mathcal{P}(R^{\rm op}) = \mathcal{I}(R^{\rm op})$ by \cite[Proposition 10.2.15]{Bland}, if $R$ is a QF ring, and so one can note that every module in $\Mod(R)$ is Gorenstein projective. Thus, in the case $n = 0$, $(\mathcal{GP}(R),\mathcal{P}(R))$ and $(\mathcal{I}(R),\mathcal{GI}(R))$ coincide with the trivial cotorsion pairs $(\Mod(R),\mathcal{I}(R))$ and $(\mathcal{P}(R),\Mod(R))$, respectively. 
\end{example}

The previous example is not necessarily an equivalence. Indeed, there are slightly more general conditions for $R$ under which $(\mathcal{GP}(R),\mathcal{P}(R))$ is still a left $n$-cotorsion pair in $\Mod(R)$. These conditions will involve the following two relative homological dimensions: 
\begin{align}
\pd\,(\mathcal{I}(R)):= {\rm sup}\{ \pd(I) \mbox{ : } I \in \mathcal{I}(R) \}, \label{eqn:pdi} \\
 \id\,(\mathcal{P}(R)):={\rm sup}\{ \id(P) \mbox{ : } P \in \mathcal{P}(R) \}. \label{eqn:idp}
\end{align}
Recall from Beligiannis and Reiten's \cite[Definitions 2.1 and 2.5]{BR07} that a ring $R$ is a \emph{left Gorenstein ring} if 
$\mathsf{Mod}(R)$ is a Gorenstein category, that is, if $\pd\,(\mathcal{I}(R)) $ and $ \id\,(\mathcal{P}(R))$ are both finite. Every $n$-Iwanaga-Gorenstein ring is a Gorenstein ring, but the converse is not necessarily true. 

Below we give a characterisation and properties of Gorenstein rings in terms of left and right $n$-cotorsion pairs involving the classes $\mathcal{GP}(R)$, $\mathcal{GI}(R)$, $\mathcal{P}(R)$ and $\mathcal{I}(R)$, the homological dimensions \eqref{eqn:pdi} and \eqref{eqn:idp}, and the global Gorenstein homological dimensions. Recall that the (\emph{left}) \emph{global Gorenstein projective dimension} of a ring $R$ is defined as the supremum 
\[
{\rm gl.GPD}(R) := {\rm sup}\{ {\rm Gpd}(M) \mbox{ : } M \in \Mod(R) \}.
\]
Dually, we have the \emph{global Gorenstein injective dimension} ${\rm gl.GID}(R)$ of $R$.

\begin{proposition}\label{prop:ncot_spli_silp}
The following conditions hold true for any ring $R$:
\begin{enumerate}
\item If $(\mathcal{GP}(R),\mathcal{P}(R))$ is a left $n$-cotorsion pair in $\mathsf{Mod}(R)$, then 
\[
\glGPD(R) = \id\,(\mathcal{P}(R))\leq n.
\] 
Dually, if $(\mathcal{I}(R),\mathcal{GI}(R))$ is a right $m$-cotorsion pair in $\mathsf{Mod}(R)$, then 
\[
\glGID(R) = \pd\,(\mathcal{I}(R))\leq m.
\] 

\item The following assertions are equivalent:
\begin{itemize}
\item[(a)] $R$ is a left Gorenstein ring which is not QF.

\item[(b)] There exist integers $n, m \geq 1$ such that $(\mathcal{GP}(R),\mathcal{P}(R))$ is a left $n$-cotorsion pair and $(\mathcal{I}(R),\mathcal{GI}(R))$ is a right $m$-cotorsion pair in $\Mod(R)$.
\end{itemize}
Moreover, if any of the previous holds true, we can choose 
\[
n = m = \id\,(\mathcal{P}(R)) = \pd\,(\mathcal{I}(R)).
\]
\end{enumerate}
\end{proposition}

\begin{proof}
Let us first show part (1). Suppose $(\mathcal{GP}(R),\mathcal{P}(R))$ is a left $n$-cotorsion pair in $\mathsf{Mod}(R)$. Then, every module has Gorenstein projective dimension at most $n$. It follows by \cite[Corollary 5.19]{BecerrilMendozaSantiago} that $\glGPD(R) = \id\,(\mathcal{P}(R))\leq n$. 

Now for part (2), let us show first the implication (a) $\Rightarrow$ (b). If $R$ is a Gorenstein ring which is not QF, we have that both $\pd\,(\mathcal{I}(R)) $ and $ \id\,(\mathcal{P}(R))$ are finite. By \cite[Proposition VII.1.3 (vi)]{BR07}, we have $\pd\,(\mathcal{I}(R)) = \id\,(\mathcal{P}(R))$. Thus, let $n := \id\,(\mathcal{P}(R))$ and note that $n \geq 1$ since $R$ is not QF. The first two conditions of Definition~\ref{def:ncotorsion} are well known for $\mathcal{GP}(R)$ and $\mathcal{P}(R)$. By \cite[Theorem VII.2.2 ($\gamma$)]{BR07}, we have that every module has Gorenstein projective dimension at most $n$. Thus, the remaining condition (3) in Definition~\ref{def:ncotorsion} follows after setting $\mathcal{X} = \mathcal{GP}(R)$ and $\omega = \mathcal{P}(R)$ in \cite[Theorem 2.8]{BMPS}. In a similar way, we can show that $(\mathcal{I}(R),\mathcal{GI}(R))$ is a right $n$-cotorsion pair in $\mathsf{Mod}(R)$.

Finally, the converse implication (b) $\Rightarrow$ (a) in part (2) follows by part (1). 
\end{proof}

For the following observations, recall that an $R$-module $M \in \Mod(R)$ is \emph{Gorenstein flat} if $M \simeq Z_0(F)$, where $F = (F_m)_{m \in \mathbb{Z}}$ is an exact complex of flat $R$-modules such that for every injective right $R$-module $E \in \mathcal{I}(R^{\rm op})$, the induced complex of abelian groups
\[
E \otimes_R F = \cdots \to E \otimes_R F_1 \to E \otimes_R F_0 \to E \otimes_R F_{-1} \to \cdots
\]
is exact. We shall denote the class of Gorenstein flat $R$-modules by $\mathcal{GF}(R)$.

\begin{remark}\label{rem:consequences_InjGInj}
Let us mention some other consequences of having a right $n$-cotorsion pair of $R$-modules $(\mathcal{I}(R),\mathcal{GI}(R))$ for some $n \geq 1$.
\begin{enumerate}
\item If $R$ is a right coherent ring, then the \textbf{Pontryagin dual} $M^+ := \Hom_{\mathbb{Z}}(M,\mathbb{Q / Z})$ of every Gorenstein injective left $R$-module $M \in \mathsf{Mod}(R)$ is a Gorenstein flat right $R$-module. For this, consider the value 
\[
\mathrm{fd}\,(\mathcal{I}(R)):= {\rm sup}\{ {\rm fd}(I) \mbox{ : } I \in \mathcal{I}(R) \}\leq\pd\,(\mathcal{I}(R)).
\]
Under the assumption that $(\mathcal{I}(R),\mathcal{GI}(R))$ is a right $n$-cotorsion pair, we have by Proposition~\ref{prop:ncot_spli_silp} that $\mathrm{fd}\,(\mathcal{I}(R)) \leq n$. Thus, Iacob's \cite[Theorem 4]{IacobAuslandercondition} implies that $M^+ \in \mathcal{GF}(R\op)$ for every $M \in \mathcal{GI}(R)$.

\item If $R$ is a left Noetherian and right coherent ring, then both $\pd\,(\mathcal{I}(R)) $ and $ \id\,(\mathcal{P}(R))$ are finite. Indeed, we already know $\pd\,(\mathcal{I}(R)) \leq n$ by Proposition \ref{prop:ncot_spli_silp}. Now let $P$ be a projective $R$-module. Then by Fieldhouse's \cite[Theorem 2.2]{Fieldhouse}, we have that $\id(P) = {\rm fd}(P^+)$, where $P^+$ is an injective $R^{\rm op}$-module. By (1) above, it follows that $\id(P) = {\rm fd}(P^+) \leq n$, and hence $\id\,(\mathcal{P}(R)) \leq n$. 

\item If $R$ is a two sided Noetherian ring, then $\mathcal{GI}(R)$ is covering and $\mathcal{GF}(R)$ is preenveloping. This follows by part (1) of Propositon~\ref{prop:ncot_spli_silp} and \cite[Theorem 4]{IacobAuslandercondition}.

\item If $M$ is an $R$-module with finite injective dimension, then $M$ has projective dimension at most $n.$ Indeed, let $M\in\mathcal{I}(R)^\vee.$ Then, by \cite[Lemma 2.6]{BecerrilMendozaSantiago}, we get $\pd\,(M)\leq\pd\,(\mathcal{I}(R)^\vee)=\pd\,(\mathcal{I}(R))\leq n.$ 

\item By (4) and \cite[Proposition VII.1.3(iii)]{BR07}, it follows that the \emph{big finitistic injective dimension} of $R$ is finite. Specifically, 
\[
{\rm FID}(R) := {\rm sup}\{ \id(M) \mbox{ : $M$ has finite injective dimension} \} \leq n.
\] 
\end{enumerate}
\end{remark}

In what remains of this section, we mention some consequences of Section~\ref{sec:approximations}. Most of our comments below have to do with right and left approximations by $\mathcal{GP}(R)$ and $\mathcal{GI}(R)$ with the unique mapping property. We begin with the following application of Corollary~\ref{Aperp y Bvee} in the context of Gorenstein homological algebra.

\begin{corollary}\label{corUMP1}
Let $R$ be an $n$-Iwanaga-Gorenstein ring with $n \geq 1$. Then, the following equalities hold:
\begin{enumerate}
\item $\Mod(R) = \mathcal{GP}(R)_n^\wedge = {}^{\perp_{n}}(\mathcal{P}(R)^\wedge_{n-1})$.

\item $\Mod(R) = \mathcal{GI}(R)_{n}^{\vee} = (\mathcal{I}(R)^\vee_{n-1})^{\perp_{n}}$.
\end{enumerate}
\end{corollary}

\begin{proof} 
We only focus on the Gorenstein projective case (1). Firstly, by Example~\ref{ex:GProj_ncot} we have that $(\mathcal{GP}(R),\mathcal{P}(R))$ is a left $n$-cotorsion pair. Then, by Proposition~\ref{prop:ncot_spli_silp} (1) the  first equality $\mathsf{Mod}(R) = \mathcal{GP}(R)_n^\wedge$ holds true. The equality $\mathsf{Mod}(R) = {}^{\perp_{n}}(\mathcal{P}(R)^\wedge_{n-1})$, on the other hand, will follow by condition (c) in Corollary~\ref{Aperp y Bvee} after showing that $\mathcal{GP}(R) = {}^{\perp_1}(\mathcal{P}(R)^\wedge_{n-1})$ and that $(\mathcal{GP}(R),\mathcal{P}(R)^\wedge_{n-1})$ is a left $(n+1)$-cotorsion pair in $\mathsf{Mod}(R)$. The former follows by the already known fact that $(\mathcal{GP}(R),\mathcal{P}(R))$ is a left $n$-cotorsion pair and by Theorem~\ref{theo:left-n-cotorsion}, while the latter can be noticed from the inclusion $\mathcal{P}(R)^\wedge_{n-1}\subseteq (\mathcal{P}(R)^\wedge_{n-1})^\wedge_n.$
\end{proof}

It is known that every module over an $n$-Iwanaga-Gorenstein ring has a Gorenstein injective cover (see, for instance \cite[Theorem 11.1.3]{EJ}). We can deduce a stronger assertion for the case $n = 2$, due to the dual of Corollary~\ref{corUMP2}.

\begin{corollary}\label{coro:GI_unique_mapping}
Let $R$ be a $2$-Iwanaga-Gorenstein ring. Then, every module has a Gorenstein injective cover with the unique mapping property. 
\end{corollary}

\begin{proof}
From the dual of the proof of Corollary~\ref{corUMP1}, for the case $n = 2$, we can note that the pair $(\mathcal{I}(R)^\vee_1,\mathcal{GI}(R))$ is a right $3$-cotorsion pair such that $\mathcal{GI}(R) = (\mathcal{I}(R)^\vee_1)^{\perp_1}$. Then, the result follows by dual of Corollary~\ref{corUMP2}, since over a $2$-Iwanaga-Gorenstein ring, every module has Gorenstein injective dimension at most $2$. 
\end{proof}

The existence of Gorenstein projective envelopes with the unique mapping property, on the other hand, has been studied by Mao in \cite{MaoPiCoherent}. Mao establishes a series of equivalent conditions under which a finitely generated module over a ring $R$ has a Gorenstein projective envelope with the unique mapping property \cite[Theorem 3.7]{MaoPiCoherent}. For (not necessarily finitely generated) modules over a $2$-Iwanaga-Gorenstein ring, we can say that if a module has a Gorenstein projective envelope, then we can always find for this module a Gorenstein projective envelope with the unique mapping property.

\begin{corollary}\label{coro:GP_unique_mapping}
Let $R$ be a $2$-Iwanaga-Gorenstein ring. Then, the following conditions are equivalent.
\begin{itemize}
\item[(a)] Every module has a Gorenstein projective envelope. 

\item[(b)] Every module has a Gorenstein projective envelope with the unique mapping property. 
\end{itemize}
\end{corollary}

\begin{proof}
It follows by Corollary \ref{corUMP2} after noting that $(\mathcal{GP}(R),\mathcal{P}(R))$ is a left $3$-cotorsion pair in $\Mod(R)$ over any $2$-Iwanaga-Gorenstein ring $R$. 
\end{proof}


\subsection*{\textbf{Ding projective modules}} 

In what follows, let us denote by $\mathcal{F}(R)$ the class of flat $R$-modules. Recall from Gillespie's \cite[Definition~3.7]{GillespieDing} that an $R$-module $M$ is \emph{Ding projective} (also called \emph{strongly Gorenstein flat} in Ding, Li and Mao's \cite{DLM}) if $M = Z_0(P)$ for some exact and $\Hom_R(-,\mathcal{F}(R))$-acyclic complex $P$ of projective $R$-modules. Dually, \emph{Ding injective} $R$-modules are defined as cycles in an exact and $\Hom_R(\mathcal{AP}(R),-)$-acyclic complexes of injective $R$-modules. We denote the classes of Ding projective and Ding injective $R$-modules by $\mathcal{DP}(R)$ and $\mathcal{DI}(R)$, respectively. 

After a careful revision of the results cited from \cite{Holm} in the previous example, we can assert that the same results carry over to the context of Ding projective modules. Specifically, one can show that, over an arbitrary ring $R$, the class $\mathcal{DP}(R)$ is closed under direct summands and that $\Ext^i_R(C,F) = 0$ for every $C \in \mathcal{DP}(R)$, $F \in \mathcal{F}(R)$ and $i \geq 1$. The dual statements hold for the classes $\mathcal{DI}(R)$ and $\mathcal{AP}(R)$. On the other hand, condition (3) in Definition~\ref{def:ncotorsion} and its dual are valid for certain rings introduced by J. Chen and N. Ding \cite{DingChen93,DingChen96}. These rings $R$ are known as \emph{$n$-FC rings} (or \emph{Ding-Chen rings}): $R$ is left and right coherent and ${\rm apd}({}_R R) = {\rm apd}(R_R) = n$. 

Similar to Gorenstein projective and Gorenstein injective dimensions, the \emph{Ding projective} and \emph{Ding injective dimensions} of a module $M \in \Mod(R)$, denoted by ${\rm Dpd}(R)$ and ${\rm Did}(M)$, are defined as the $\mathcal{DP}(R)$-resolution and the $\mathcal{DI}(R)$-coresolution dimensions of $M$, respectively. For these two homological dimensions, it is not true in general that the equality $\Mod(R) = \mathcal{DP}(R)_n^\wedge$ holds for an $n$-FC ring $R$. An example of such ring $R$ for which $\Mod(R) \neq \mathcal{DP}(R)_n^\wedge$ is constructed by Wang in \cite[Example 3.3]{Wang}. It follows that we can not always have the Ding projective analog of Example~\ref{ex:GProj_ncot}. As a matter of fact, the condition $\Mod(R) = \mathcal{DP}(R)_n^\wedge$ is strong enough to guarantee the existence of $(\mathcal{DP}(R),\mathcal{F}(R))$ as a left $n$-cotorsion pair in $\Mod(R)$. 

For the rest of this section, recall that the global Ding projective and Ding injective dimensions of a ring $R$ are defined by:
\begin{align*}
{\rm gl.DPD}(R) & = {\rm sup}\{ {\rm Dpd}(M) \mbox{ : } M \in \Mod(R) \}, \\
{\rm gl.DID}(R) & = {\rm sup}\{ {\rm Did}(M) \mbox{ : } M \in \Mod(R) \}. 
\end{align*}

\begin{example}\label{ex:DP_n-cotorsion}
Let $n \geq 1$ be an integer and $R$ be any ring with $\glDPD(R) \leq n.$ Then, the pair $(\mathcal{DP}(R),\mathcal{F}(R))$ is a left $n$-cotorsion pair in $\Mod(R)$. Indeed, by the previous comments it suffices to show that for every module $M \in \Mod(R)$ there is an epimorphism $P \twoheadrightarrow M$ with $P \in \mathcal{DP}(R)$ and kernel in $\mathcal{F}(R)^\wedge_{n-1}$. This follows by setting $\mathcal{X} = \mathcal{DP}(R)$ and $\omega := \mathcal{P}(R)\subseteq\mathcal{F}(R)$ in \cite[Theorem 2.8]{BMPS}, since ${\rm Dpd}(M) \leq n$.
\end{example}

We can obtain characterisations of Von Neumann regular rings by considering the situation in which $(\mathcal{DP}(R),\mathcal{F}(R))$ is a left and right $n$-cotorsion pair in $\Mod(R)$.

\begin{proposition} 
For any ring $R$, the following conditions are equivalent.
\begin{itemize}
\item[(a)] $(\mathcal{DP}(R),\mathcal{F}(R))$ is an $n$-cotorsion pair in $\mathsf{Mod}(R)$ and  $\mathcal{DP}(R) \subseteq \mathcal{F}(R)$.

\item[(b)] $R$ is a Von Neumann regular ring (that is, $\mathcal{F}(R) = \Mod(R)$).
\end{itemize}
\end{proposition}

\begin{proof}
Indeed, let us suppose that (a) holds true. Then, by using the fact that $\mathcal{DP}(R)$ is resolving, we get (b) from Remark~\ref{Rk-n-cot-trivial} (2). 

Assume now that $\mathcal{F}(R) = \Mod(R)$. In order to prove (a), it is enough to show that $\mathcal{DP}(R) = \mathcal{P}(R)$. Note that in this case, we can choose any $n \geq 1$. Let $M \in \mathcal{DP}(R).$ Then, there is an exact sequence 
\[
\eta \colon 0 \to M \to P \to M' \to 0,
\] 
where $P \in \mathcal{P}(R)$ and $M' \in \mathcal{DP}(R)$. Since $\mathcal{F}(R) = \Mod(R)$, it follows that $\Hom_R(\eta,M)$ is exact and thus $\eta$ splits, proving that $M \in \mathcal{P}(R)$.
\end{proof}

Let us give in the next result some finiteness conditions for the global Ding injective dimension $\glDID(R)$ of any ring $R$.

\begin{lemma}\label{glDIDfinite} 
For any ring $R$, the following statements are equivalent:
\begin{itemize}
\item[(a)] $(\mathcal{I}(R)^\vee,\mathcal{DI}(R))$ is a hereditary complete cotorsion pair in $\Mod(R)$, and $\pd(\mathcal{I}(R)) < \infty$.

\item[(b)] $\mathcal{DI}(R) = \mathcal{I}(R)^\perp$ and $\pd(\mathcal{I}(R)) < \infty$.

\item[(c)] $\glDID(R) < \infty$.
\end{itemize}

Moreover, if one of the above conditions holds true, then
\[
\glDID(R) = {\rm pd}(\mathcal{AP}(R)) = {\rm pd}(\mathcal{I}(R)).
\]
\end{lemma}

\begin{proof} 
We use freely the notation and results from \cite{BecerrilMendozaSantiago}. Note that the pair $(\mathcal{AP}(R),\mathcal{I}(R))$ is GI-admissible and WGI-admissible in the sense of \cite[Definitions 3.6 and 4.5]{BecerrilMendozaSantiago}. Then, by the dual of \cite[Corollaries  5.12 (c2) and 5.17]{BecerrilMendozaSantiago}, the result follows.
\end{proof}

Motivated by Example~\ref{ex:DP_n-cotorsion}, we present the following family of rings.

\begin{definition} 
We say that a ring $R$ is \textbf{left Ding-finite} if 
\begin{align*}
\glDPD(R) & < \infty & & \mbox{and} & \glDID(R) & < \infty.
\end{align*}
\end{definition}

\begin{proposition}\label{Prop:nDPDI} 
If a ring $R$ is left Ding-finite, then the following statements hold true:
\begin{enumerate}
\item The cotorsion pairs $(\mathcal{DP}(R),\mathcal{P}(R)^\wedge)$ and $(\mathcal{I}(R)^\vee,\mathcal{DI}(R))$ are hereditary and complete.

\item $\mathcal{DP}(R) = {}^\perp\mathcal{P}(R)$ and $\mathcal{DI}(R)) = \mathcal{I}(R)^\perp$.

\item Both $\pd(\mathcal{I}(R))$ and $\id(\mathcal{P}(R))$ are finite, and
\[
\glDID(R) = {\rm pd}(\mathcal{AP}(R)) = {\rm pd}(\mathcal{I}(R)) = {\rm id}(\mathcal{P}(R)) = \glDPD(R) = {\rm id}(\mathcal{F}(R)).
\]
\end{enumerate}
\end{proposition}

\begin{proof} 
It follows by Lemma~\ref{glDIDfinite}, \cite[Corollary 5.18]{BecerrilMendozaSantiago} and \cite[Proposition VII.1.3 (vi)]{BR07}.
\end{proof}

\begin{remark} \
\begin{enumerate}
\item By Proposition~\ref{Prop:nDPDI}, note that every left Ding-finite ring is a left Gorenstein ring in the sense of \cite{BR07}.

\item In case (3) holds in Proposition~\ref{Prop:nDPDI}, we have that $\glDPD(R)$ also coincides with the (left) global Gorenstein dimension (see Bennis and Mahdou's \cite[Theorem 1.1]{BMglobal} and Mahdou and Tamekkante's \cite[Theorem 3.2]{MahdouTamekkante}). 

\item Let $R$ be a Ding-Chen ring. Then, by \cite{Yang} $\glDPD(R) = \glDID(R)$. The latter may include the case where $\glDPD(R) = \infty$ and $\glDID(R) = \infty$. Then, not every Ding-Chen ring is left Ding-finite, as shown by Wang in \cite[Example 3.3]{Wang}.
\end{enumerate} 
\end{remark}

The following result is the Ding-Chen analogous of Proposition~\ref{prop:ncot_spli_silp}, and follows similarly.

\begin{proposition}\label{prop:ncot_DPDI}
The following conditions hold true for any ring $R$:
\begin{enumerate}
\item If $(\mathcal{DP}(R),\mathcal{P}(R))$ is a left $n$-cotorsion pair in $\mathsf{Mod}(R)$, then 
\[
\glDPD(R) = {\rm id}(\mathcal{P}(R)) \leq n.
\] 
Dually, if $(\mathcal{I}(R),\mathcal{DI}(R))$ is a right $m$-cotorsion pair in $\mathsf{Mod}(R)$, then 
\[
\glDID(R) = {\rm pd}(\mathcal{I}(R)) \leq m.
\] 

\item The following assertions are equivalent:
\begin{itemize}
\item[(a)] $R$ is left Ding-finite with $\glDPD(R) = \glDID(R) = n \geq 1$. 

\item[(b)] There exist integers $n, m \geq 1$ such that $(\mathcal{DP}(R),\mathcal{P}(R))$ is a left $n$-cotorsion pair and $(\mathcal{I}(R),\mathcal{DI}(R))$ is a right $m$-cotorsion pair in $\Mod(R)$.
\end{itemize}

Moreover, if any of the previous two conditions holds true, we can choose 
\[
n = m = {\rm id}(\mathcal{P}(R)) = {\rm pd}(\mathcal{I}(R)).
\]
\end{enumerate}
\end{proposition}


\subsection*{\textbf{Gorenstein flat modules}}

We have previously mentioned in Section~\ref{sec:approximations} characterisations of certain rings which consider their global dimensions. In this example, given a left perfect ring $R$, we shall find equivalent conditions for which $R$ is quasi-Frobenius or has null global Gorenstein flat dimension. These conditions involve left and right $n$-cotorsion pairs formed by the classes $\mathcal{F}(R)$ and $\mathcal{GF}(R)$ of flat and Gorenstein flat $R$-modules.

The \emph{Gorenstein flat dimension} of an $R$-module $M \in \Mod(R)$, which we denote by ${\rm Gfd}(M)$, is defined as the $\mathcal{GF}(R)$-resolution dimension of $M$, that is,
\[
{\rm Gfd}(M) := {\rm resdim}_{\mathcal{GF}(R)}(M).
\] 
Let us define the (\emph{left}) \emph{global Gorenstein flat dimension} of $R$ as the value
\[
{\rm gl.Gfd}(R) := {\rm sup}\{ {\rm Gfd}(M) \mbox{ : } M \in \Mod(R) \}.
\]

In the next results, we explore the situation where $(\mathcal{F}(R),\mathcal{GF}(R))$ is a left or a right $n$-cotorsion pair in $\Mod(R)$.

\begin{proposition}\label{prop:FlatGFlatLeft}
The following conditions are equivalent for any ring $R$ and any integer $n \geq 1$:
\begin{itemize}
\item[(a)] $(\mathcal{F}(R),\mathcal{GF}(R))$ is a left $n$-cotorsion pair in $\mathsf{Mod}(R)$.

\item[(b)] $\Ext^1_R(\mathcal{F}(R),\mathcal{GF}(R)) = 0$ and ${\rm gl.Gfd}(R) \leq n$. 
\end{itemize}
\end{proposition}

\begin{proof}
The implication (a) $\Rightarrow$ (b) is straightforward. Now suppose that condition (b) holds. It is well known that the class $\mathcal{F}(R)$ is closed under direct summands. Moreover, using the fact that $\mathcal{F}(R)$ is resolving, the condition $\Ext^1_R(\mathcal{F}(R),\mathcal{GF}(R)) = 0$ implies that $\Ext^i_R(\mathcal{F}(R),\mathcal{GF}(R)) = 0$ for every $1 \leq i \leq n$. 

It was recently proved by J. {\v{S}}aroch and J. {\v{S}}\v{t}ov\'{\i}\v{c}ek \cite[Corollary 3.12]{SarochStovicek} that the class $\mathcal{GF}(R)$ is closed under extensions for any ring $R$ (that is, any ring $R$ is GF-closed). In particular, from \cite[Proposition 6.17]{BMPS} it follows that $\omega := \mathcal{F}(R) \cap \mathcal{F}(R)^{\perp_1}$ is a relative cogenerator in $\mathcal{X} := \mathcal{GF}(R)$. Thus, by \cite[Theorem 2.8]{BMPS} for every $M \in \mathsf{Mod}(R)$ we can obtain a short exact sequence
\[
0 \to K \to G \to M \to 0,
\]
where $G \in \mathcal{GF}(R)$ and $K \in (\mathcal{F}(R) \cap \mathcal{F}(R)^{\perp_1})^\wedge_{n-1}$, since ${\rm Gfd}(M) \leq n$. On the other hand, by the definition of $\mathcal{GF}(R)$, we have another short exact sequence
\[
0 \to G' \to F \to G \to 0,
\]
where $F \in \mathcal{F}(R)$ and $G' \in \mathcal{GF}(R)$. Taking the pullback of $K \to G \leftarrow F$, we have the following commutative diagram with exact rows and columns:
\begin{equation}\label{fig6} 
\parbox{1.75in}{
\begin{tikzpicture}[description/.style={fill=white,inner sep=2pt}] 
\matrix (m) [ampersand replacement=\&, matrix of math nodes, row sep=2.5em, column sep=2.5em, text height=1.25ex, text depth=0.25ex] 
{ 
G' \& G' \& {} \\
E \& F \& M \\
K \& G \& M \\
}; 
\path[->] 
(m-2-1)-- node[pos=0.5] {\footnotesize$\mbox{\bf pb}$} (m-3-2)
; 
\path[>->]
(m-1-1) edge (m-2-1) (m-1-2) edge (m-2-2)
(m-2-1) edge (m-2-2) (m-3-1) edge (m-3-2)
;
\path[->>]
(m-2-1) edge (m-3-1) (m-2-2) edge (m-3-2)
(m-2-2) edge (m-2-3) (m-3-2) edge (m-3-3)
;
\path[-,font=\scriptsize]
(m-1-1) edge [double, thick, double distance=2pt] (m-1-2)
(m-2-3) edge [double, thick, double distance=2pt] (m-3-3)
;
\end{tikzpicture} 
}
\end{equation} 
By Bennis' \cite[Theorem 2.11]{Bennis}, we can note that ${\rm Gfd}(E) \leq n-1$. Hence, the central row in \eqref{fig6} completes the proof that $(\mathcal{F}(R),\mathcal{GF}(R))$ is a left $n$-cotorsion pair.
\end{proof}

If we assume that $(\mathcal{F}(R),\mathcal{GF}(R))$ is a right $n$-cotorsion pair instead, we can show that $R$ is a left perfect and a left IF ring. Recall from Colby's \cite{Colby} that a ring $R$ is a \emph{left IF ring} if every injective left $R$-module is flat. Before proving the previous assertion concerning $(\mathcal{F}(R),\mathcal{GF}(R))$, let us show the following characterisation of IF rings in terms of its global Gorenstein flat dimension.

\begin{lemma}\label{Lema:IF-ring}
Let $R$ be a ring. If ${\rm gl.Gfd}(R) = 0$, then $R$ is a left IF ring. If in addition ${\rm gl.Gfd}(R^{op}) = 0$, then the converse also holds. Moreover, if $R$ is commutative, then $R$ is an IF ring if, and only if, ${\rm gl.Gfd}(R) = 0$. 
\end{lemma}

\begin{proof}
Suppose first that ${\rm gl.Gfd}(R) = 0$, and let $E$ be an injective module. Then, $E$ is also Gorenstein flat, and so there exists a short exact sequence 
\[
0 \to E \to F \to N \to 0
\]
with $F \in \mathcal{F}(R)$ and $N \in \mathcal{GF}(R)$, which splits since $E$ is injective. It follows that $E$ is flat, and hence $R$ is a left IF ring. 

Now, suppose that $R$ is a left IF ring with ${\rm gl.Gfd}(R^{op}) = 0$, and so $R$ is also a right IF ring. For every $M \in \mathsf{Mod}(R)$, note that we can find a chain complex 
\[
F_\bullet = \cdots \to P_1 \to P_0 \to E^0 \to E^1 \to \cdots
\]
where $M = {\rm Ker}(E^0 \to E^1)$, $P_i \in \mathcal{P}(R)$ and $E^j \in \mathcal{I}(R)$, for every $i, j \geq 0$. This is a complex of flat modules, since every injective is flat. Moreover, injective right $R$-modules are also flat, and then $E \otimes_R F_\bullet$ is exact, for every injective $E \in \mathcal{I}(R\op)$. 
\end{proof}

\begin{proposition}\label{prop:perfect_global_flat}
The following conditions are equivalent for any ring $R$.
\begin{itemize}
\item[(a)] $(\mathcal{F}(R),\mathcal{GF}(R))$ is a right $n$-cotorsion pair in $\mathsf{Mod}(R)$ for some integer $n \geq 1$.

\item[(b)] $R$ is left perfect and ${\rm gl.Gfd}(R) = 0$. 
\end{itemize}
Moreover, in the case $R$ is commutative, we have that $R$ is a left perfect and an IF ring if, and only if, there exists an integer $n \geq 1$ such that $(\mathcal{F}(R),\mathcal{GF}(R))$ is a right $n$-cotorsion pair in $\mathsf{Mod}(R)$. 
\end{proposition}

\begin{proof}
First, note that the implication (b) $\Rightarrow$ (a) is clear. To show (a) $\Rightarrow$ (b), let us assume that there exists an integer $n \geq 1$ such that $(\mathcal{F}(R),\mathcal{GF}(R))$ is a right $n$-cotorsion pair in $\Mod(R)$. Then, for every $M \in \mathsf{Mod}(R)$ there exists an exact sequence
\[
0 \to M \to G \to F^0 \to F^1 \to \cdots \to F^{n-2} \to F^{n-1} \to 0
\]
with $G \in \mathcal{GF}(R)$ and $F^k \in \mathcal{F}(R)$, for every $0 \leq k \leq n-1$. Since the class $\mathcal{GF}(R)$ is resolving by \cite[Corollary 3.12]{SarochStovicek}, we have that $M$ is Gorenstein flat, and hence ${\rm gl.Gfd}(R) = 0$. Now, let $F$ be a flat $R$-module, and consider an exact sequence
\[
0 \to K \to P \to F \to 0,
\]
with $P$ projective. Note that this sequence splits since $\Ext^1_R(\mathcal{F}(R),\mathsf{Mod}(R)) = 0$, and so $F$ is projective. Therefore, $R$ is a left perfect ring with ${\rm gl.Gfd}(R) = 0$. 
\end{proof}

Two more interesting results occur if we switch the roles for the classes $\mathcal{F}(R)$ and $\mathcal{GF}(R)$, that is, if we analyse the implications of assuming that $(\mathcal{GF}(R),\mathcal{F}(R))$ is a left or a right $n$-cotorsion pair in $\mathsf{Mod}(R)$.

\begin{proposition}\label{prop:perfect_QF}
For any ring $R$ the following conditions are equivalent:
\begin{itemize}
\item[(a)] $(\mathcal{GF}(R),\mathcal{F}(R))$ is a right $n$-cotorsion pair in $\mathsf{Mod}(R)$ for some integer $n \geq 1$.

\item[(b)] $R$ is left perfect and QF. 
\end{itemize}
\end{proposition}

\begin{proof}
Let us show first the implication (a) $\Rightarrow$ (b). Our first step is to show that every module is Gorenstein flat. Indeed, for every $M \in \mathsf{Mod}(R)$ we have an exact sequence
\[
0 \to M \to F \to G^0 \to G^1 \to \cdots \to G^{n-2} \to G^{n-1} \to 0
\]
with $F \in \mathcal{F}(R)$ and $G^k \in \mathcal{GF}(R)$ for every $0 \leq k \leq n-1$, since $(\mathcal{GF}(R),\mathcal{F}(R))$ is a right $n$-cotorsion pair. Using the fact that $\mathcal{GF}(R)$ is resolving, we obtain that $M$ is Gorenstein flat. Now, from the dual of Theorem \ref{theo:left-n-cotorsion}, we get that $\mathcal{F}(R) = (\mathcal{GF}(R)^\vee_{n-1})^{\perp_1} = \mathsf{Mod}(R)^{\perp_1} = \mathcal{I}(R)$. On the other hand, for every flat module $F$ we have a short exact sequence 
\[
0 \to F' \to P \to F \to 0
\]
with $P$ projective and $F'$ flat (and so injective). Thus, this sequence splits, and then $F$ is projective. Therefore, we finally obtain $\mathcal{P}(R) = \mathcal{F}(R) = \mathcal{I}(R)$. This implies that $R$ is a left perfect and a QF ring.

Now, we prove that (b) implies (a). Assume that $R$ is left perfect and QF. Then, we have that the following conditions hold: (i) $\mathcal{P}(R) = \mathcal{F}(R) = \mathcal{I}(R)$, and (ii) $\mathcal{I}(R^{\rm op}) = \mathcal{P}(R^{\rm op}) \subseteq\mathcal{F}(R^{\rm op})$. We assert that $\mathcal{GF}(R) = \Mod(R)$. Indeed, for any $M \in \Mod(R)$ we can construct an exact complex 
\[
\eta \colon \cdots \to P_1 \to P_0 \to I^0 \to I^1 \to \cdots,
\] 
where $P_i \in \mathcal{P}(R)$ and $I^j \in \mathcal{I}(R)$ for any $i, j \geq 0$, and $M = {\rm Ker}(I^0 \to I^1)$.  By condition (i), we get that $\eta$ is an acyclic complex of flat modules, and applying (ii) it follows that the complex $E \otimes_R \eta$ is acyclic for any injective $E \in \mathcal{I}(R^{\rm op})$. Then, $M \in \mathcal{GF}(R)$. Once we have the equality $\mathcal{GF}(R) = \mathsf{Mod}(R)$, it can be shown easily that $(\mathcal{GF}(R),\mathcal{F}(R))$ is a right $1$-cotorsion pair in $\mathsf{Mod}(R)$, since $\mathcal{I}(R) = \mathcal{F}(R)$.
\end{proof}

Now let us consider the remaining scenario where $(\mathcal{GF}(R),\mathcal{F}(R))$ is a left $n$-cotorsion pair in $\Mod(R)$. It will be important to recall that for an arbitrary ring $R$, the Pontryagin duality functor $(-)^+ \colon \Mod(R) \longrightarrow \Mod(R^{\rm op})$ maps every flat $R$-module into an injective $R^{\rm op}$-module (see Enochs and Jenda's \cite[Theorem 3.2.9]{EJ}), and every Gorenstein flat $R$-module into a Gorenstein injective $R^{\rm op}$-module (as proved for example in Holm's \cite[Theorem 3.6]{Holm} or in Meng and Pan's \cite[Proposition 4.4]{MengPan}).

Recall also that for every $N \in \Mod(R\op)$ there is a \textbf{pure exact sequence}
\begin{align}\label{eqn:canonical_pure}
\rho_N \colon & 0 \to N \to N^{++} \to N^{++} / N \to 0,
\end{align}
that is, $\rho_N \otimes_R M$ is exact for every $M \in \Mod(R)$ (see \cite[Proposition 5.3.9]{EJ}).

\begin{proposition}\label{prop:GFF_globlaGID}
Let $R$ be a ring over which $(\mathcal{GF}(R),\mathcal{F}(R))$ is a left $n$-cotorsion pair in $\Mod(R)$ for some integer $n \geq 1$. Then, ${\rm Gid}(N^{++}) \leq n$ for every $N \in \Mod(R^{\rm op})$. If in addition $R$ is a commutative noetherian ring with dualizing complex, then $\glGID(R^{\rm op}) \leq n$. 
\end{proposition}

\begin{proof}
Let $N \in \Mod(R^{\rm op})$ and consider its character module $N^+ \in \Mod(R)$. Since $(\mathcal{GF}(R),\mathcal{F}(R))$ is a left $n$-cotorsion pair in $\Mod(R)$ for some $n \geq 1$, we can find a short exact sequence
\[
0 \to K \to G \to N^+ \to 0
\]
where $G$ is a Gorenstein flat $R$-module and ${\rm fd}(K) \leq n-1$. Then, we have the following exact sequence involving $N^{++}$:
\[
0 \to N^{++} \to G^+ \to K^+ \to 0.
\]
Here, $G^+$ if Gorenstein injective and ${\rm id}(K^+) \leq n-1$ by previous comments. Hence, ${\rm Gid}(N^{++}) \leq n$ for every $N \in \Mod(R^{\rm op})$. 

Now let us assume that $R$ is a commutative noetherian ring with dualizing complex. Under these conditions, it is known that the class $\mathcal{GI}(R^{\rm op})^\vee_n$ of modules with Gorenstein injective dimension $\leq n$ is closed under pure submodules by \cite[Lemma 2.5 (b) and Theorem 3.1]{HJduality}. On the other hand, for every $N \in \Mod(R^{\rm op})$ there is a canonical pure exact sequence
\[
0 \to N \to N^{++} \to N^{++} / N \to 0,
\] 
where ${\rm Gid}(N^{++}) \leq n$ by the previous part. It follows that ${\rm Gid}(N) \leq n$ for every $N \in \Mod(R^{\rm op})$.
\end{proof}

The dual statement of the previous result holds in case $R$ is a two-sided noetherian ring.

\begin{proposition}\label{prop:plusplus}
Let $R$ be any ring and $n \geq 1$ be an integer. If $R$ is two-sided Noetherian and $(\mathcal{I}(R\op), \mathcal{GI}(R\op))$ is a right $n$-cotorsion pair in $\Mod(R\op)$, then ${\rm Gfd}(M^{++}) \leq n$ for every $M \in \Mod(R)$.
\end{proposition}

\begin{proof}
Recall that over a two-sided Noetherian ring $R$, a right $R$-module is injective if, and only if, its Pontryagin dual is flat. Furthermore, by Remark~\ref{rem:consequences_InjGInj} (1) we have that $N^+$ is a Gorenstein flat $R$-module, for any Gorenstein injective $N \in \Mod(R\op)$. Using these two facts, the rest of the proof follows as in Proposition~\ref{prop:GFF_globlaGID}.
\end{proof}

Let us consider again the left global Gorenstein flat dimension of $R$, but this time in the case where $(\mathcal{GF}(R),\mathcal{F}(R))$ is a left $n$-cotorsion pair.

\begin{proposition}\label{(GF,F)}
The following conditions are equivalent for any ring $R$ and any integer $n \geq 1$:
\begin{itemize}
\item[(a)] $(\mathcal{GF}(R),\mathcal{F}(R))$ is a left $n$-cotorsion pair in $\Mod(R)$.

\item[(b)] $\Ext^1_R(\mathcal{GF}(R),\mathcal{F}(R)) = 0$ and ${\rm gl.Gfd}(R) \leq n$. 
\end{itemize}
\end{proposition}

\begin{proof}
The implication (a) $\Rightarrow$ (b) is clear. Now, let us assume that (b) holds true. The condition $\Ext^i_R(\mathcal{GF}(R),\mathcal{F}(R)) = 0$ is clear for every $1 \leq i \leq n$, since $\mathcal{GF}(R)$ is resolving. Moreover, $\mathcal{GF}(R)$ is closed under direct summands by \cite[Corollary 3.12]{SarochStovicek}. The rest of the implication follows by applying \cite[Theorem 2.8]{BMPS} again, as in the proof of Proposition~\ref{prop:FlatGFlatLeft}. 
\end{proof}

Propositions~\ref{prop:GFF_globlaGID} and \ref{(GF,F)} are not the only consequences of having $\mathcal{GF}(R)$ and $\mathcal{F}(R)$ forming a left $n$-cotorsion pair $(\mathcal{GF}(R),\mathcal{F}(R))$ in $\Mod(R)$. For the rest of this section, we shall study other possible results from this assumption, regarding the relation between the classes $\mathcal{F}(R)$, $\mathcal{GF}(R)$, $\mathcal{I}(R^{\rm op})$ and $\mathcal{GI}(R\op)$ via the Pontryagin duality functor $(-)^+$. Namely, we shall focus on the following: 
\begin{enumerate}
\item To look for conditions under which it is possible to find an integer $k \geq 1$ such that $(\mathcal{I}(R\op),\mathcal{GI}(R\op))$ is a right $k$-cotorsion pair in $\Mod(R\op)$, provided that $(\mathcal{GF}(R),\mathcal{F}(R))$ is a left $n$-cotorsion pair in $\Mod(R)$.

\item To see if the converse procedure is possible, that is, if there exists $k \geq 1$ for which $(\mathcal{GF}(R),\mathcal{F}(R))$ is a left $k$-cotorsion pair in $\Mod(R)$, assuming that $(\mathcal{I}(R\op),\mathcal{GI}(R\op))$ is a right $n$-cotorsion pair in $\Mod(R\op)$. 
\end{enumerate}

\begin{theorem}\label{theo:GFGI_Pontryagin}
Let $R$ be any ring such that $(\mathcal{GF}(R),\mathcal{F}(R))$ is a left $n$-cotorsion pair in $\Mod(R)$. Then, the following conditions are equivalent.
\begin{itemize}
\item[(a)] $m := \sup \{{\rm Gid}(N^{++} / N) \mbox{ {\rm :} } N \in \Mod(R\op) \} < \infty$.

\item[(b)] $(\mathcal{I}(R\op),\mathcal{GI}(R\op))$ is a right $(k+1)$-cotorsion pair in $\Mod(R\op)$ for some integer $k \geq 1$. 
\end{itemize}
Moreover, if any of the previous conditions holds true, one can take $k = \max\{n, m\}$.
\end{theorem}

\begin{proof}
For the first implication, let $k := \max\{n, m\}.$ It suffices to show that, for every $N \in \Mod(R\op),$ one can find an embedding into $\mathcal{GI}(R\op)$ whose cokernel has injective dimension at most $k$. Consider the canonical pure exact sequence $\rho_N$ from \eqref{eqn:canonical_pure}. By Proposition~\ref{prop:GFF_globlaGID}, we get that ${\rm Gid}(N^{++}) \leq k$. The latter, along with condition (a) implies that ${\rm Gid}(N) \leq k + 1$ (see \cite[Proposition 2.15]{MengPan}). Hence, (b) follows by \cite[dual of Theorem 2.8]{BMPS}.

Conversely, if $(\mathcal{I}(R\op),\mathcal{GI}(R\op))$ is a right $(k+1)$-cotorsion pair for some $k \in \mathbb{N}$, then ${\rm Gid}(N) \leq k + 1$ and ${\rm Gid}(N^{++}) \leq n$ for every $N \in \Mod(R\op)$. Using \cite[Proposition 2.15]{MengPan} again, we obtain that ${\rm Gid}(N^{++} / N) \leq {\rm max}\{ k, n \}$.
\end{proof}

\begin{theorem}\label{theo:GFGI_Pontryagin}
Let $R$ be a two-sided noetherian ring such that $(\mathcal{I}(R\op), \mathcal{GI}(R\op))$ is a right $n$-cotorsion pair in $\Mod(R\op)$. Then, $\mathrm{gl.Gfd}(R)\leq n$. Moreover, the following conditions are equivalent:
\begin{itemize}
\item[(a)] $\Ext_R^1(\mathcal{GF}(R),\mathcal{F}(R)) = 0$.
 
\item[(b)] $(\mathcal{GF}(R),\mathcal{F}(R))$ is a left Frobenius pair in $\Mod(R)$ (in the sense of \cite[Definition 2.5]{BMPS}).

\item[(c)] $(\mathcal{GF}(R),\mathcal{F}(R))$ is a left $k$-cotorsion pair in $\Mod(R)$ for some integer $k \geq 1$.
\end{itemize}
If any of the previous conditions holds, one can take $k = n$. Furthermore,  
\[
\mathrm{gl.Gfd}(R) = \id(\mathcal{F}(R)).
\]
\end{theorem}

\begin{proof} 
For the first part, let $M\in\Mod(R)$ and consider the canonical pure exact sequence
\[
\rho_M \colon 0 \to M \to M^{++} \to M^{++} / M \to 0,
\] 
where ${\rm Gfd}(M^{++}) \leq n$ by Proposition~\ref{prop:plusplus}. On the other hand, by \cite[Lemma 2.5 (a) and Theorem 3.1]{HJduality}, we get that the class $\mathcal{GF}(R)^\wedge_n$ is closed under pure submodules, and thus from $\rho_M$ we get that ${\rm Gfd}(M) \leq n$.

Now, let us show the equivalence between (a), (b) and (c).  
\begin{itemize}
\item (a) $\Rightarrow$ (b): From \cite[Corollary 3.12]{SarochStovicek}, we know that $\mathcal{GF}(R)$ is the left part of an hereditary cotorsion pair in $\Mod(R).$ Thus, the condition $\Ext_R^{1}(\mathcal{GF}(R),\mathcal{F}(R)) = 0$ implies that  $\Ext_R^{i}(\mathcal{GF}(R),\mathcal{F}(R)) = 0$ for any $i \geq 1$. Therefore, $\mathcal{F}(R)$ is an $\mathcal{GF}(R)$-injective relative cogenerator in $\mathcal{GF}(R)$. Finally, it is clear that $\mathcal{F}(R)$ is closed under direct summands, thus proving (b).

\item (b) $\Rightarrow$ (c): Assume that (b) holds true. Since $\mathcal{GF}(R)^\wedge_n = \Mod(R)$, we get by \cite[Theorem 2.10]{BMPS} that 
\[
\mathrm{gl.Gfd}(R) = \pd_{\mathcal{F}(R)}(\Mod(R)) = \id\,(\mathcal{F}(R)).
\] 
Moreover, for any $M\in\Mod(R)$ we get by \cite[Theorem 2.8]{BMPS} an exact sequence 
\[
0 \to K \to G \to M \to 0,
\] 
where $G\in\mathcal{GF}(R)$ and $K \in \mathcal{F}(R)^\wedge_{n-1}$.
  
\item (c) $\Rightarrow$ (a): Trivial.
\end{itemize}
\end{proof}


\subsection*{\textbf{Cluster tilting subcategories}}

Following Iyama's \cite[Definition 1.1]{IyamaCluster}, for an integer $m \geq 1$, a subcategory $\mathcal{D} \subseteq \mathcal{C}$ is said to be \emph{$m$-cluster tilting} if it is precovering and preenveloping, and the following equalities hold true
\[
\mathcal{D} = \bigcap\limits_{0 < i < m}^{}{}^{\perp_{i}}\mathcal{D} = \bigcap\limits_{0 < i < m}^{}\mathcal{D}^{\perp_{i}}.
\]

\begin{remark}\label{rem:cluster_precover}
Note that if $\mathcal{D}$ is an $m$-cluster tilting subcategory (with $m \geq 2$) of an abelian category $\mathcal{C}$ with enough projectives and injectives, then $\mathcal{D}$-precovers and $\mathcal{D}$-preenvelopes are special, since $\Ext^1_{\mathcal{C}}(\mathcal{D,D}) = 0$. 
\end{remark}

In this example, we prove that a subcategory $\mathcal{D}$ of an abelian category $\mathcal{C}$ is an $(n+1)$-cluster tilting subcategory if, and only if, it forms an $n$-cotorsion pair of the form $(\mathcal{D,D})$. The following result is straightforward.

\begin{lemma}\label{nclustA=B}
Let $\mathcal{A}$ and  $\mathcal{B}$ be classes of objects of $\mathcal{C}$ such that $\Ext^i_{\mathcal{C}}(\mathcal{A,B}) = 0$ for any integer $1 \leq i \leq n$. If the containment
\[
\bigcap\limits_{i = 1}^{m} {}^{\perp_{i}}(\mathcal{A} \cap \mathcal{B}) \subseteq \mathcal{A} \cap \mathcal{B},
\] 
holds for some integer $1 \leq m \leq n$, then $\mathcal{A} \subseteq \mathcal{B}$. 
\end{lemma}

\begin{proposition}\label{nclustA=B y ff}
Let $(\mathcal{A,B})$ be an $n$-cotorsion pair in $\mathcal{C}$. Then, the following conditions hold true:
\begin{enumerate}
\item If there exists an integer $1 \leq m \leq n$ such that the equalities 
\[
\mathcal{A} \cap \mathcal{B} = \bigcap\limits_{i=1}^{m} {}^{\perp_{i}}(\mathcal{A} \cap \mathcal{B}) = \bigcap\limits_{i=1}^{m} (\mathcal{A} \cap \mathcal{B})^{\perp_{i}}
\]
hold true, then $\mathcal{A} = \mathcal{B}$ and the class $\mathcal{A} \cap \mathcal{B} = \mathcal{A}$ is special precovering and special preenveloping. 

\item The class $\mathcal{A} \cap \mathcal{B}$ is an $(n+1)$-cluster tilting subcategory if, and only if, $\mathcal{A} = \mathcal{B}$. 
\end{enumerate}
\end{proposition}

\begin{proof}
Part (1) follows by Lemma~\ref{nclustA=B} and Proposition~\ref{A-precub,B-preenv} and their duals. 

The ``only if'' statement of part (2) is a consequence of part (1). Now for the ``if'' statement, suppose that $\mathcal{A} = \mathcal{B}$. Then, by Proposition~\ref{A-precub,B-preenv} and its dual, we get that $\mathcal{A}$ is a special precovering and a special preenveloping class. Thus, it suffices to prove that 
\[
\mathcal{A} = \bigcap_{i = 1}^n {}^{\perp_i}\mathcal{A} = \bigcap_{i = 1}^n \mathcal{A}^{\perp_i},
\] 
but this follows from Theorem~\ref{theo:left-n-cotorsion} (2) and its dual. 
\end{proof}

We are now ready to show the following characterisation of $(n+1)$-cluster tilting subcategories.

\begin{theorem} \label{ncot y ct}
Let $\mathcal{C}$ be an abelian category with enough projectives and injectives. Then, for any subcategory $\mathcal{D} \subseteq \mathcal{C}$ and any integer $n \geq 1$, the following statements are equivalent:
\begin{itemize}
\item[(a)] $(\mathcal{D,D})$ is an $n$-cotorsion pair in $\mathcal{C}$.

\item[(b)] $\mathcal{D}$ is an $(n+1)$-cluster tilting subcategory of $\mathcal{C}$.
\end{itemize}
Moreover, in case any of the above conditions holds true, we have $\mathcal{C} = \mathcal{D}^\wedge_n = \mathcal{D}^\vee_n$.
\end{theorem}

\begin{proof} 
The implication (a) $\Rightarrow$ (b) follows by Proposition~\ref{nclustA=B y ff}. Now suppose that $\mathcal{D}$ is an $(n+1)$-cluster tilting subcategory of $\mathcal{C}$. Then, we have that $\mathcal{D}$ is closed under direct summands and that $\Ext^{i}_{\mathcal{C}}(\mathcal{D,D}) = 0$ for any integer $1 \leq i \leq n$. The result will follow after showing the equalities $\mathcal{C} = \mathcal{D}^{\wedge}_{n} = \mathcal{D}^{\vee}_{n}$.

By Remark~\ref{rem:cluster_precover}, for any $M \in \mathcal{C}$ we can consider an exact sequence 
\[
\eta \colon 0 \to K_0 \to D_0 \xrightarrow{f_0} M \to 0,
\] 
where $f_0$ is a special $\mathcal{D}$-precover. After applying the functor $\Hom_{\mathcal{C}}(D,-)$ to $\eta$, with $D$ running over $\mathcal{D}$, we get:
\begin{align*}
\Ext^1_{\mathcal{C}}(D,K_0) & = 0 \quad \text{and}\quad\Ext^{i+1}_{\mathcal{C}}(D,K_0) \cong \Ext^{i}_{\mathcal{C}}(D,M)\;\text{ for any}\; 1 \leq i \leq n-1.
\end{align*}
Inductively, we can construct an exact sequence
\begin{align*}
0 & \to K_{n} \to D_{n-1} \xrightarrow{f_{n-1}} D_{n-2} \to \cdots \to D_1 \xrightarrow{f_1} D_0 \xrightarrow{f_0} M \to 0,
\end{align*}
where $D_i \in \mathcal{D}$ and $K_i := \Ima\,(f_i)$ for any $0 \leq i \leq n-1$, and such that the following relations hold:
\begin{align*}
\Ext^1_{\mathcal{C}}(D,K_n) &= 0,\\
\Ext^2_{\mathcal{C}}(D,K_n) & \cong \Ext^1_{\mathcal{C}}(D,K_{n-1}) = 0,\\
\vdots & \qquad\qquad\qquad \vdots\qquad\qquad\qquad \vdots \\
\Ext^n_{\mathcal{C}}(D,K_n) & \cong \Ext^{n-1}_{\mathcal{C}}(D,K_{n-1}) \cong \cdots \cong \Ext^1_{\mathcal{C}}(D,K_1) = 0.
\end{align*}
Therefore, we get that $K_n \in \bigcap_{i = 1}^n\,\mathcal{D}^{\perp_{i}} = \mathcal{D}$ and thus $M \in \mathcal{D}^{\wedge}_n$. Dually, we get $M \in \mathcal{D}^{\vee}_n$.
\end{proof}

One interesting fact to note about $(n+1)$-cluster tilting subcategories $\mathcal{D}$ is that $n$ is the biggest integer for which the condition $\Ext^n_{\mathcal{C}}(\mathcal{D,D}) = 0$ is true, in the sense that letting $\Ext^{n+1}_{\mathcal{C}}(\mathcal{D,D}) = 0$ forces $\mathcal{C}$ to be a Frobenius category. We specify this in the following result.

\begin{proposition}
Let $n \geq 1$ and $\mathcal{D}$ be an $(n+1)$-cluster tilting subcategory of an abelian category $\mathcal{C}$ with enough projectives and injectives. Then, the following conditions are equivalent:
\begin{itemize}
\item[(a)] $\Ext^{n+1}_{\mathcal{C}}(\mathcal{D,D}) = 0$.

\item[(b)] $\mathcal{P}(\mathcal{C}) = \mathcal{D}$.

\item[(c)] $\mathcal{I}(\mathcal{C}) = \mathcal{D}$.

\item[(d)] $\Ext^i_{\mathcal{C}}(\mathcal{D,D}) = 0$ for every $i \geq 1$. 
\end{itemize} 
\end{proposition}

\begin{proof}
It suffices to show that (a) implies (b) and (c). So let us assume that $\Ext^{n+1}_{\mathcal{C}}(\mathcal{D,D}) = 0$. By Proposition \ref{prop8}, we have that $\Ext^1_{\mathcal{C}}(\mathcal{D},\mathcal{D}^\wedge_n) = 0$, and since $\mathcal{C} = \mathcal{D}^\wedge_n$ by the proof of Theorem \ref{ncot y ct}, we obtain the containment $\mathcal{D} \subseteq {}^{\perp_1}\mathcal{C} = \mathcal{P}(\mathcal{C})$. Dually, we can also show that $\mathcal{D} \subseteq \mathcal{C}^{\perp_1} = \mathcal{I}(\mathcal{C})$ holds. On the other hand, we know that $\mathcal{P}(\mathcal{C}) \cup \mathcal{I}(\mathcal{C}) \subseteq \mathcal{D}$, since $\bigcap_{1 \leq i \leq n} {}^{\perp_i}\mathcal{D} = \mathcal{D} = \bigcap_{1 \leq i \leq n} \mathcal{D}^{\perp_i}$. Therefore, $\mathcal{P}(\mathcal{C}) = \mathcal{D} = \mathcal{I}(\mathcal{C})$. 
\end{proof}

\begin{remark}\label{rem:cluster_trivial}
Theorem~\ref{ncot y ct} may constitute a nontrivial example of a two-sided $n$-cotorsion pair. Namely, let $\mathcal{D}$ be an $(n+1)$-cluster tilting subcategory of an abelian category $\mathcal{C},$ with enough projectives and injectives. Then by Theorem~\ref{ncot y ct} and Remark~\ref{Rk-n-cot-trivial} (2), $(\mathcal{D,D})$ is the trivial $n$-cotorsion pair (that is, $\mathcal{D} = \mathcal{P}(\mathcal{C})$) if, and only if, $\mathcal{D}$ is resolving. 
\end{remark}

Using the previous theorem and \cite[Theorem 1.6]{IyamaCluster}, we obtain the following example.

\begin{example}
Let $\Lambda$ be an Artin $R$-algebra. Note that the category  $\fmod(\Lambda),$ of finitely generated left $\Lambda$-modules, is an abelian category with enough projectives and injectives, as it is well known that every finitely generated $\Lambda$-module has a finitely generated projective cover and a finitely generated injective envelope. 
\begin{enumerate}
\item If ${\rm gl.dim}(\Lambda) \leq n+1$ and $\fmod(\Lambda)$ has an $(n+1)$-cluster tilting object $T$, then there exists a unique $n$-cotorsion pair in $\fmod(\Lambda)$ of the form $(\mathcal{D,D})$, where $\mathcal{D} := {\rm add}(T)$ is the class of direct summands of finite direct sums of copies of $T$. In this case, note that $\mathcal{D}$ is resolving if, and only if, $\mathcal{D} = {\rm add}(\Lambda)$. 

\item If $\Lambda$ is not self-injective having an $(n+1)$-cluster tilting object $T$, we necessarily have that $\Ext^{n+1}_{\Lambda}(T,T) \neq 0$ and $\mathcal{P}(\Lambda) \cup \mathcal{I}(\Lambda) \subsetneq {\rm add}(T)$.
\end{enumerate}
\end{example}


\section{Higher cotorsion for chain complexes}\label{sec:complexes}

The last part of the present paper is devoted to study $n$-cotorsion pairs in the setting provided by the category $\Ch(\mathcal{C})$ of chain complexes of objects in $\mathcal{C}$. In the first part of this section, we characterise certain families of $n$-cotorsion pairs of complexes in terms of $n$-cotorsion pairs in the ground category $\mathcal{C}$. In the second part, we shall study how to induce $n$-cotorsion pairs of complexes from an $n$-cotorsion pair $(\mathcal{A,B})$ in $\mathcal{C}$. The complexes involved in these $n$-cotorsion pairs are the $\mathcal{A}$-complexes, $\mathcal{B}$-complexes, and differential graded complexes considered by Gillespie in \cite{GillespieFlat}.

Let us set some notation for the category $\Ch(\mathcal{C})$. Given a chain complex $X \in \mathsf{Ch}(\mathcal{C})$ with differentials $\partial^X_m \colon X_m \to X_{m-1}$, we denote its cycle and boundary objects in $\mathcal{C}$ by $Z_m(X) := {\rm Ker}(\partial^X_m)$ and $B_m(X) := {\rm Im}(\partial^X_{m+1})$, respectively. 

Let us also borrow some notation from \cite[Section 3]{GillespieDegreewise}. Let $(\mathscr{A,B})$ be an $n$-cotorsion pair in $\Ch(\C)$. The symbol $\mathscr{A}'$ will denote the class of all objects $M \in \mathcal{C}$ such that $M = A_m$ for some $A \in \mathscr{A}$ and some $m \in \mathbb{Z}$. The class $\mathscr{B}'$ is defined similarly. 

Motivated by \cite[Definition 3.4]{GillespieDegreewise}, we propose the following.

\begin{definition}
An $n$-cotorsion pair $(\mathscr{A,B})$ in $\Ch(\C)$ is \textbf{degreewise orthogonal} if for every pair if integers $i, j \in \mathbb{Z}$ we have the relations:
\begin{enumerate} 
\item $\Ext_{\C}^1(A_i, Y_j) = 0$ whenever $A \in \mathscr{A}$ and $Y \in \mathscr{B}_{n-1}^{\wedge}$, and

\item $\Ext_{\C}^1(X_i, B_j) = 0$ whenever $X \in \mathscr{A}_{n-1}^{\vee}$ and $B \in \mathscr{B}$.
\end{enumerate} 
\end{definition}

Given an object $M \in \mathcal{C}$ and an integer $m \in \mathbb{Z}$, the \emph{$m$-th disk complex centred at $M$} is the chain complex denoted by $D^m(M)$, such that $M$ appears at degrees $m$ and $m\mbox{-}1$, and $0$ elsewhere. The only nonzero differential is the identity on $M$.  The \emph{$m$-th sphere complex centred at $M$}, on the other hand, is the chain complex $S^m(M) \in \Ch(\mathcal{C})$ with $M$ at the $m$-th component and $0$ elsewhere.

The first relation we note between $n$-cotorsion in $\Ch(\mathcal{C})$ and $n$-cotorsion in $\mathcal{C}$ is described in the following result, which is the $n$-cotorsion version of \cite[Lemma 3.5]{GillespieDegreewise}.

\begin{lemma}\label{JKtoJ'K'}
Let $\C$ be an abelian category with enough injectives. Then, the following statements are equivalent for any $n$-cotorsion pair $(\mathscr{A,B})$ in $\Ch(\mathcal{C})$:
\begin{itemize}
\item[(a)] $(\mathscr{A,B})$ is degreewise orthogonal.

\item[(b)] If $A \in \mathscr{A}$ and $B \in \mathscr{B}$, then $D^{m}(A_i) \in \mathscr{A}$ and $D^n(B_j) \in \mathscr{B}$ for every $m, n, i, j \in \mathbb{Z}$.

\item[(c)] $(\mathscr{A}',\mathscr{B}')$ is an $n$-cotorsion pair in $\C$.
\end{itemize}
\end{lemma}

\begin{proof}
We prove the implications (a) $\Rightarrow$ (b) $\Rightarrow$ (c) $\Rightarrow$ (a). 
\begin{itemize}
\item (a) $\Rightarrow$ (b): Let $A \in \mathscr{A}$ and $Y \in \mathscr{B}_{n-1}^{\wedge}$. By \cite[Lemma 3.1]{GillespieFlat} we have that 
\[
\Ext_{\Ch}^1(D^{m}(A_i), Y) \cong \Ext_{\C}^1(A_i, Y_m) = 0
\] 
where $\Ext_{\C}^1(A_i, Y_m) = 0$ by condition (a). Thus, $D^m(A_i) \in {}^{\perp_1}(\mathscr{B}_{n-1}^{\wedge}) = \mathscr{A}$ by Theorem~\ref{theo:left-n-cotorsion}. In a similar way, we can prove that $D^m(B_j) \in \mathscr{B}$ for any $j, m \in \mathbb{Z}$ whenever $B \in \mathscr{B}$.

\item (b) $\Rightarrow$ (c): We prove that $(\mathscr{A}', \mathscr{B}')$ is a left $n$-cotorsion pair in $\mathcal{C}$ assuming (b). 

We first show that $\mathscr{A}'$ is closed under direct summands. Let $N \in \mathscr{A}'$ and $M$ be a direct summand of $A$. Then, $N = A_m$ for some complex $A \in \mathscr{A}$ and some $m \in \mathbb{Z}$. Note that $D^0(A_m) \in \mathscr{A}$ by condition (b), and that $D^0(M)$ is a direct summand of $D^0(A_m)$. Since $\mathscr{A}$ is closed under direct summands by hypothesis, we have that $D^0(M) \in \mathscr{A}$, that is, $M \in \mathscr{A}'$. Hence, $\mathscr{A}'$ is closed under direct summands. 

Now let us show that $\Ext^i_{\C}(\mathscr{A}',\mathscr{B}') = 0$ for every $1 \leq i \leq n$. By Proposition~\ref{prop:cotorsion_vs_ncotorsion}, since $\mathcal{C}$ has enough injectives, it is equivalent to show that $\Ext^1_{\mathcal{C}}(\mathscr{A}',(\mathscr{B}')^\wedge_{n-1}) = 0$. So let $M \in \mathscr{A}'$ and $N \in (\mathscr{B}')^\wedge_{n-1}$. By condition (b), we can note that $D^{0}(M) \in \mathscr{A}$ and $D^1(N) \in \mathscr{B}_{n-1}^{\wedge}$. From \cite[Lemma 3.1]{GillespieFlat}, we have that
\[
\Ext_{\C}^1(M,N) \cong \Ext_{\Ch}^1(D^0(M),D^1(N)) = 0,
\]
where the last equality follows by Proposition~\ref{prop:cotorsion_vs_ncotorsion}. Hence, $\Ext_{\C}^1(\mathscr{A}', (\mathscr{B}')_{n-1}^\wedge) = 0$. 

Finally, we show that for every object $C \in \mathcal{C}$ there exists a short exact sequence
\[
0 \to N \to M \to C \to 0
\]
where $M \in \mathscr{A}'$ and $N \in (\mathscr{B}')^\wedge_{n-1}$. For, consider the sphere complex $S^0(C) \in \Ch(\mathcal{C})$. Since $(\mathscr{A},\mathscr{B})$ is an $n$-cotorsion pair in $\Ch(\mathcal{C})$, there exists a short exact sequence
\[
0 \to Y \to A \to S^0(C) \to 0 
\]
where $A \in \mathscr{A}$ and $Y \in \mathscr{B}^\wedge_{n-1}$. Thus, at degree $0$ we have the exact sequence 
\[
0 \to N \to M \to C \to 0
\] 
where $M = A_0 \in \mathscr{A}'$ and $N = Y_0 \in (\mathscr{B}')^\wedge_{n-1}$.  

The previous shows that $(\mathscr{A}',\mathscr{B}')$ is a left $n$-cotorsion pair in $\mathcal{C}$. In a similar way, one can show that $(\mathscr{A}',\mathscr{B}')$ is also a right $n$-cotorsion pair in $\mathcal{C}$. Therefore, (c) follows. 

\item (c) $\Rightarrow$ (a): It is clear due to the equalities $\mathscr{A}' = {}^{\perp_1}((\mathscr{B}')_{n-1}^{\wedge})$ and $\mathscr{B}' = ((\mathscr{A}')_{n-1}^{\vee})^{\perp_1}$.
\end{itemize}
\end{proof}

In what follows, we need to consider the subgroup $\mathsf{Ext}^1_{\mathsf{dw}}(X,Y)$ of $\mathsf{Ext}^1_{\mathsf{Ch}}(X,Y)$ of classes of short exact sequences 
\[
0 \to Y \to Z \to X \to 0
\] 
which are degreewise split, that is, 
\[
0 \to Y_m \to Z_m \to X_m \to 0
\] 
is a split exact sequence in $\mathcal{C}$ for every $m \in \mathbb{Z}$. 

Recall also that given a chain complex $X \in \Ch(\mathcal{C})$ and an integer $k \in \mathbb{Z}$, the \emph{$k$-th suspension of $X$} is the complex $X[k] \in \Ch(\mathcal{C})$ with components $(X[k])_m := X_{m-k}$ and differentials $\partial^{X[k]}_m := (-1)^k \partial^X_{m-k}$. 

The following result corresponds to \cite[Proposition 3.7]{GillespieDegreewise} in the context of $n$-cotorsion pairs. We provide a characterisation for the class $\mathscr{A}$ in every degreewise orthogonal $n$-cotorsion pair $(\mathscr{A,B})$ in $\Ch(\mathcal{C})$.

\begin{proposition}
Let $(\mathscr{A,B})$ be a degreewise orthogonal $n$-cotorsion pair in $\Ch(\C)$ (where $\mathcal{C}$ is an abelian category with enough injectives), and let $(\mathscr{A}',\mathscr{B}')$ be the corresponding $n$-cotorsion pair in $\C$ from Lemma~\ref{JKtoJ'K'}. If $\mathscr{B}$ is closed under suspensions, then $\mathscr{A}$ equals the class of all complexes $A \in \Ch(\mathcal{C})$ for which $A_m \in \mathscr{A}'$ for every $m \in \mathbb{Z}$, and such that every chain map $A \to Y$ is null homotopic whenever $Y \in \mathscr{B}^{\wedge}_{n-1}$. 
\end{proposition}

\begin{proof}
Suppose that $(\mathscr{A,B})$ is an $n$-cotorsion pair in $\Ch(\mathcal{C})$ with $\mathscr{B}$ closed under suspensions. Note that the latter implies that $\mathscr{B}^{\wedge}_{n-1}$ is also closed under suspensions.

Let us denote by $\mathscr{X}$ the class of complexes $X \in \Ch(\mathcal{C})$ such that $X_m \in \mathscr{A}'$ and such that every chain map $X \to Y$ is null homotopic whenever $Y \in \mathscr{B}^\wedge_{n-1}$. We show $\mathscr{A} = \mathscr{X}$ using the equality $\mathscr{A} = {}^{\perp_1}\mathscr{B}^\wedge_{n-1}$ from Theorem~\ref{theo:left-n-cotorsion}. 
\begin{itemize}
\item $\mathscr{A} \supseteq \mathscr{X}$: Let $X \in \mathscr{X}$ and $Y \in \mathscr{B}^\wedge_{n-1}$. Since $(\mathscr{A}',\mathscr{B}')$ is an $n$-cotorsion pair by Lemma~\ref{JKtoJ'K'}, we have that $\Ext^1_{\mathcal{C}}(X_m, Y_m) = 0$ for every $m \in \mathbb{Z}$, and so $\Ext_{\Ch}^1(X, Y) = \Ext_{\mathsf{dw}}^1(X, Y)$. On the other hand, $\Ext_{\mathsf{dw}}^1(X,Y) \cong \Hom_{\Ch}(X, Y[1]) / \sim$ by \cite[Lemma 2.1]{GillespieFlat}, where $\sim$ represents the chain homotopy relation. Since $X \in \mathscr{X}$ and $Y[1] \in \mathscr{B}^\wedge_{n-1}$ being $\mathscr{B}^\wedge_{n-1}$  closed under suspensions, we have that $\Hom_{\Ch}(X, Y[1]) / \sim \mbox{} = 0$, and hence $\Ext_{\Ch}^1(X, Y) = 0$. Then, we have that $X \in \mathscr{A}$. 

\item $\mathscr{A} \subseteq \mathscr{X}$: Let $A \in \mathscr{A}$ and $Y \in \mathscr{B}^\wedge_{n-1}$. We have that
\[
0 = \Ext^1_{\Ch}(A,Y[-1]) \supseteq \Ext^1_{\mathsf{dw}}(A,Y[-1]) \cong \Hom_{\Ch}(A,Y) / \sim.
\]
since $Y[-1] \in \mathscr{B}^\wedge_{n-1}$. It follows that every chain map $A \to Y$ is null homotopic whenever $Y \in \mathscr{B}^\wedge_{n-1}$. 

Now let $N \in (\mathscr{B}')^\wedge_{n-1}$ and note that $D^{m+1}(N)\in \mathscr{B}_{n-1}^{\wedge}$ by Lemma~\ref{JKtoJ'K'}. Then, we have
\[
\Ext_{\C}^1(A_m, N) \cong \Ext_{\Ch}^1(A, D^{m+1}(N)) = 0,
\]
that is, $A_m \in {}^{\perp_1}[(\mathscr{B}')^\wedge_{n-1}] = \mathscr{A}'$ (by Lemma~\ref{JKtoJ'K'} and Theorem~\ref{theo:left-n-cotorsion}). Therefore, we have that $A_m \in \mathscr{A}'$ for all $m \in \mathbb{Z}$.
\end{itemize}
\end{proof}

Now let us show how to induce $n$-cotorsion pairs involving certain families of complexes from an $n$-cotorsion pair in $\mathcal{C}$. These families are presented below in Definition \ref{def:special_complexes}, which follows the spirit of Gillespie's \cite[Definition 3.3]{GillespieFlat}.

For the rest of this section, it will be important to recall that $\Ch(\mathcal{C})$ is equipped with an \emph{internal hom} functor $\mathcal{H}{om}(-,-)$ defined as follows: for every $X, Y \in \Ch(\mathcal{C})$, $\mathcal{H}{om}(X,Y)$ is the chain complex of abelian groups defined by 
\[
\mathcal{H}{om}(X,Y)_m := \prod_{k \in \mathbb{Z}} \Hom_{\mathcal{C}}(X_k, Y_{m+k})
\]
for every $m \in \mathbb{Z}$, and with differentials given by $f \mapsto \partial^Y \circ f - (-1)^m f \circ \partial^X$ (see Garc\'ia Rozas' \cite{JRGR}, for instance). It is known that every chain map $X \to Y$ is null homotopic if, and only if, the complex $\mathcal{H}{om}(X,Y)$ is exact.

\begin{definition}\label{def:special_complexes}
Let $\mathcal{X}$ be a class of objects of $\mathcal{C}$. A chain complex $X \in \Ch(\mathcal{C})$ is:
\begin{enumerate}
\item a \textbf{complex (with terms) in $\bm{\mathcal{X}}$} (or a \textbf{degreewise $\bm{\mathcal{X}}$-complex}) if $X_m \in \mathcal{X}$ for every $m \in \mathbb{Z}$;

\item an \textbf{$\bm{\mathcal{X}}$-complex} if $X$ is exact and $Z_m(X) \in \mathcal{X}$ for every $m \in \mathbb{Z}$.
\end{enumerate}
We shall denote by $\Ch(\mathcal{X})$ the class of complexes in $\mathcal{X}$, and by $\widetilde{\mathcal{X}}$ the class of $\mathcal{X}$-complexes. 

Now let $\mathcal{A}$ and $\mathcal{B}$ be two classes of objects in $\mathcal{C}$ such that $\Ext^1_{\mathcal{C}}(\mathcal{A,B}) = 0$. We can also define two new families of complexes from $\Ch(\mathcal{A})$, $\widetilde{\mathcal{A}}$, $\Ch(\mathcal{B})$ and $\widetilde{\mathcal{B}}$. 
\begin{enumerate}
\setcounter{enumi}{2}
\item We shall say that a complex $X \in \Ch(\mathcal{C})$ is \textbf{$\bm{\mathcal{H}{om}(-,\widetilde{\mathcal{B}})}$-acyclic in $\bm{\Ch(\mathcal{A})}$} if $X \in \Ch(\mathcal{A})$ and if $\mathcal{H}{om}(X,B)$ is an exact complex of abelian groups whenever $B \in \widetilde{\mathcal{B}}$. 

\item \textbf{$\bm{\mathcal{H}{om}(\widetilde{\mathcal{A}},-)}$-acyclic complexes in $\bm{\Ch(\mathcal{B})}$} are defined dually, that is, as those complexes $Y \in \Ch(\mathcal{B})$ such that $\mathcal{H}{om}(A,Y)$ is exact for every $A \in \widetilde{\mathcal{A}}$. 
\end{enumerate}
We shall denote by $\Ch_{\rm acy}(\mathcal{A};\widetilde{\mathcal{B}})$ the class of $\mathcal{H}{om}(-,\widetilde{\mathcal{B}})$-acyclic complexes in $\Ch(\mathcal{A})$. Dually, $\Ch_{\rm acy}(\widetilde{\mathcal{A}};\mathcal{B})$ will denote the class of $\mathcal{H}{om}(\widetilde{\mathcal{A}},-)$-acyclic complexes in $\Ch(\mathcal{B})$.
 \end{definition}

\begin{remark} \label{sumB'} \
\begin{enumerate}
\item If a class $\mathcal{X}$ of objects in $\mathcal{C}$ is closed under extensions, then $\widetilde{\mathcal{X}} \subseteq \Ch(\mathcal{X})$. 

\item If $X \in \widetilde{\mathcal{X}}$, then $X[k] \in \widetilde{\mathcal{X}}$ for every $k \in \mathbb{Z}$. 

\item If $0 \in \mathcal{X}$, then $D^m(X) \in \widetilde{\mathcal{X}}$ for every $X \in \mathcal{X}$ and $m \in \mathbb{Z}$. 

\item In the case where $(\mathcal{A,B})$ is a cotorsion pair in $\mathcal{C}$, $\mathcal{H}{om}(-,\widetilde{\mathcal{B}})$-acyclic complexes in $\Ch(\mathcal{A})$ and $\mathcal{H}{om}(\widetilde{\mathcal{A}},-)$-acyclic complexes in $\Ch(\mathcal{B})$ are called in \cite{GillespieFlat} \textit{differential graded $\mathcal{A}$-complexes} and \textit{differential graded $\mathcal{B}$-complexes}, respectively. Since there may be more than two pairs $(\mathcal{A,B})$ of classes objects in $\mathcal{C}$ satisfying the condition $\Ext^1_{\mathcal{C}}(\mathcal{A,B}) = 0$, we have preferred to use the terminology specified in Definition~\ref{def:special_complexes} above in order to avoid confusion. 

In fact, we can find an example for $\Ch_{\rm acy}(\mathcal{A};\widetilde{\mathcal{B}})$ where $\mathcal{A}$ and $\mathcal{B}$ do not form a cotorsion pair $(\mathcal{A,B})$ in $\mathcal{C}$. This is the case for the classes $\mathcal{A} := \mathcal{DP}(R)$ and $\mathcal{B} := \mathcal{F}(R)$ of Ding projective and flat modules. By \cite[Theorem 3.7]{DPcomplexes}, a chain complex over an arbitrary ring $R$ is Ding projective if, and only if, it is $\mathcal{H}{om}(-,\widetilde{\mathcal{F}(R)})$-acyclic in $\Ch(\mathcal{DP}(R))$. Keep in mind that $\widetilde{\mathcal{F}(R)}$ is precisely the class of flat complexes.

\item If $\Ext_{\C}^1(\mathcal{A,B}) = 0$, $\mathcal{B}$ is closed under extensions  and $0 \in \mathcal{A}$, then $S^m(A)\in \Ch_{\rm acy}(\mathcal{A};\widetilde{\mathcal{B}})$ for every $A \in \mathcal{A}$ and $m \in \mathbb{Z}$.
\end{enumerate}
\end{remark}

The following result follows as \cite[Lemma 3.9]{GillespieFlat}.

\begin{lemma}\label{lem:null_homotopic}
Let $\mathcal{A}$ and $\mathcal{B}$ be two classes of objects in $\C$. If $\Ext_{\C}^1(\mathcal{A,B}) = 0$ and $\mathcal{B}$ is closed under extensions, then every chain map from a $\mathcal{A}$-complex to a $\mathcal{B}$-complex is null homotopic.
\end{lemma}

Before inducing higher cotorsion pairs from an $n$-cotorsion pair in $\mathcal{C}$, we prove the following orthogonality relations between the classes (1), (2), (3) and (4) in Definition~\ref{def:special_complexes}. Recall that $\mathcal{C}$ is said to have \textit{enough $\mathcal{X}$-objects}, for some class $\mathcal{X}$ of objects of $\mathcal{C}$, if every object of $\mathcal{C}$ is an epimorphic image of an object in $\mathcal{X}$.

\begin{lemma}\label{B subseteq dgA}
Let $\mathcal{A}$ and $\mathcal{B}$ be two classes of objects in $\mathcal{C}$. Then, the following statements hold true:
\begin{enumerate}
\item If $\Ext_{\C}^1(\mathcal{A,B}) = 0$, then $\Ext_{\Ch}^1(\Ch_{\rm acy}(\mathcal{A};\widetilde{\mathcal{B}}), \widetilde{\mathcal{B}}) = 0$ and $\Ext_{\Ch}^1(\widetilde{\mathcal{A}},\Ch_{\rm acy}(\widetilde{\mathcal{A}};\mathcal{B})) = 0$.

\item If $0 \in \mathcal{A}$, then $Z_m(Y) \in \mathcal{A}^{\perp_1}$ for every $Y \in (\Ch_{\rm acy}(\mathcal{A};\widetilde{\mathcal{A}^{\perp_1}}))^{\perp_1}$ and for each $m \in \mathbb{Z}$. Moreover, if $\mathcal{C}$ has enough $\A$-objects, then $Y$ is a $\mathcal{A}^{\perp_1}$-complex.

\item ${}^{\perp_1}\widetilde{\mathcal{B}} \subseteq \Ch_{\rm acy}({}^{\perp_1}\mathcal{B};\widetilde{\mathcal{B}})$.
\end{enumerate}
\end{lemma}

\begin{proof} \
\begin{enumerate}
\item Suppose the relation $\Ext^1_{\mathcal{C}}(\mathcal{A,B}) = 0$ holds true, and let $A \in \Ch_{\rm acy}(\mathcal{A};\widetilde{\mathcal{B}})$ and $B \in \widetilde{\mathcal{B}}$. We aim to show that $\Ext^1_{\Ch}(A,B) = 0$. Consider the subgroup $\Ext^1_{\mathsf{dw}}(A,B) \subseteq \Ext^1_{\Ch}(A,B)$. We know that $A_m \in \mathcal{A}$ for every $m \in \mathbb{Z}$. On the other hand, we have short exact sequences
\[
0 \to Z_m(B) \to B_m \to Z_{m-1}(B) \to 0
\]
with $Z_{m-1}(B), Z_m(B) \in \mathcal{B}$, and so $\Ext^1_{\mathcal{C}}(A_m,B_m) = 0$ for every $m \in \mathbb{Z}$. This implies that $\Ext^1_{\mathsf{dw}}(A,B) = \Ext^1_{\Ch}(A,B)$. Now in oder to show that $\Ext^1_{\mathsf{dw}}(A,B) = 0$, it suffices to use the isomorphism
\[
\Ext^1_{\mathsf{dw}}(A,B) \cong {\rm H}_0(\mathcal{H}{om}(A,B[1]))
\]
from \cite[Lemma 2.1]{GillespieFlat}. Since the complex $\mathcal{H}{om}(A,B[1])$ is exact, we have that its $0$-th homology is zero, that is, ${\rm H}_0(\mathcal{H}{om}(A,B[1])) = 0$. Hence, $\Ext^1_{\Ch}(A,B) = 0$. The equality $\Ext_{\Ch}^1(\widetilde{\mathcal{A}},\Ch_{\rm acy}(\widetilde{\mathcal{A}};\mathcal{B})) = 0$ follows in the same way. 

\item Let $Y \in (\Ch_{\rm acy}(\mathcal{A};\widetilde{\mathcal{A}^{\perp_1}}))^{\perp_1}$ and consider $Z_m(Y)$ and $A \in \mathcal{A}$. We show $\Ext^1_{\mathcal{C}}(A,Z_m(Y)) = 0$. By Gillespie's \cite[Lemma 4.2]{GillespieDegreewise}, we know that there is a monomorphism 
\[
0 \to \Ext^1_{\mathcal{C}}(A,Z_m(Y)) \to \Ext^1_{\Ch}(S^m(A),Y).
\]
So it suffices to show that $S^m(A) \in \Ch_{\rm acy}(\mathcal{A};\widetilde{\mathcal{A}^{\perp_1}})$. It is clear that $S^m(A) \in \Ch(\mathcal{A})$ since $0, A \in \mathcal{A}$. Now let $B \in \widetilde{\mathcal{A}^{\perp_1}}$. For each $m \in \mathbb{Z}$, we have that
\begin{align*}
{\rm H}_i(\mathcal{H}{om}(S^m(A),B)) & \cong \Ext^1_{\mathsf{dw}}(S^m(A),B[-i-1]).
\end{align*}
Note that $\Ext^1_{\mathcal{C}}((S^m(A))_k,(B[-i-1])_k) = 0$ for every $k \neq m$. Now for $k = m$, we have that $\Ext^1_{\mathcal{C}}((S^m(A))_m,(B[-i-1])_m) = \Ext^1_{\mathcal{C}}(A,B_{m+i+1})$. Since $B$ is an exact complex with cycles in $\mathcal{A}^{\perp_1}$, we can note that $\Ext^1_{\mathcal{C}}(A,B_{m+i+1}) = 0$ as we did in part (1). Hence, it follows that 
\[
\Ext^1_{\mathsf{dw}}(S^m(A),B[-i-1]) = \Ext^1_{\Ch}(S^m(A),B[-i-1]).
\]
Using again \cite[Lemma 4.2]{GillespieDegreewise} and the fact that $B$ is exact yields an isomorphism
\[
\Ext^1_{\Ch}(S^m(A),B[-i-1]) \cong \Ext^1_{\mathcal{C}}(A,Z_m(B[-i-1])),
\]
where $\Ext^1_{\mathcal{C}}(A,Z_m(B[-i-1])) = 0$. Then, we have that ${\rm H}_i(\mathcal{H}{om}(S^m(A),B)) = 0$ for every $i \in \mathbb{Z}$, that is, $\mathcal{H}{om}(S^m(A),B)$ is an exact complex. Thus, $S^m(A) \in \Ch_{\rm acy}(\mathcal{A};\widetilde{\mathcal{A}^{\perp_1}})$, and hence $\Ext^1_{\mathcal{C}}(A,Z_m(Y)) = 0$. 

Now suppose that in addition $\mathcal{C}$ has enough $\mathcal{A}$-objects. Since we already know that $Y$ has cycles in $\mathcal{A}^{\perp_1}$, it suffices to show that $Y$ is exact, that is, that the equality $Z_m(Y) = B_m(Y)$ holds for every $m \in \mathbb{Z}$. The containment $B_m(Y) \subseteq Z_m(Y)$ is clear. For the converse containment, we have an epimorphism $f_m \colon A \to Z_m(Y)$ with $A \in \mathcal{A}$ since $\mathcal{C}$ has enough $\mathcal{A}$-objects. This induces a chain map $\tilde{f} \colon S^m(A) \to Y$ given by $\tilde{f}_m := i_m \circ f_m$ and $0$ elsewhere, where $i_m$ is the inclusion $Z_m(Y) \hookrightarrow Y_m$. On the other hand, 
\begin{align*}
\Hom_{\Ch}(S^m(A),Y) / \sim \mbox{} & \cong {\rm H}_0(\mathcal{H}{om}(S^m(A),Y)) \cong \Ext^1_{\mathsf{dw}}(S^m(A),Y[-1]) \\
& \subseteq \Ext^1_{\Ch}(S^m(A),Y[-1]),
\end{align*}
where $\Ext^1_{\Ch}(S^m(A),Y[-1]) = 0$ since $S^m(A) \in \Ch_{\rm acy}(\mathcal{A};\widetilde{\mathcal{A}^{\perp_1}})$ (by Remark~\ref{sumB'} (5)) and $Y[-1] \in (\Ch_{\rm acy}(\mathcal{A};\widetilde{\mathcal{A}^{\perp_1}}))^{\perp_1}$. It follows that the map $\tilde{f}$ is null homotopic, and so there exists a morphism $D_{m+1} \colon A \to Y_{m+1}$ such that $\partial^Y_{m+1} \circ D_{m+1} = i_m \circ f_m$. The latter implies $Z_m(Y) \subseteq B_m(Y)$, since $f_m$ is epic. 

\item Finally, we show the containment ${}^{\perp_1}(\widetilde{B}) \subseteq \Ch_{\rm acy}({}^{\perp_1}\mathcal{B};\widetilde{\mathcal{B}})$. Let $X \in {}^{\perp_1}\widetilde{\mathcal{B}}$. We first show that $X_m \in {}^{\perp_1}\mathcal{B}$ for every $m \in \mathbb{Z}$. Let $B \in \mathcal{B}$. Then, we know that $D^{m+1}(B) \in \widetilde{\mathcal{B}}$ by Remark \ref{sumB'} (3), and so 
\[
\Ext^1_{\mathcal{C}}(X_m,B) \cong \Ext^1_{\Ch}(X,D^{m+1}(B)) = 0
\]
since $X \in {}^{\perp_1}\widetilde{\mathcal{B}}$. 

Now we show that $\mathcal{H}{om}(X,B)$ is an exact complex of abelian groups for every $B \in \widetilde{\mathcal{B}}$. We have natural isomorphisms
\[
{\rm H}_m(\mathcal{H}{om}(X,B)) \cong \Ext^1_{\mathsf{dw}}(X,B[-m-1]) = \Ext^1_{\Ch}(X,B[-m-1]) = 0
\]
where $\Ext^1_{\mathsf{dw}}(X,B[-m-1]) = \Ext^1_{\Ch}(X,B[-m-1])$ follows as in (1), and $B[-m-1] \in \widetilde{\mathcal{B}}$ by Remark~\ref{sumB'} (2).
\end{enumerate}
\end{proof}

We are know ready to show how to induce $n$-cotorsion pairs in $\Ch(\mathcal{C})$ involving the classes $\widetilde{\mathcal{A}}$, $\widetilde{\mathcal{B}}$, $\Ch_{\rm acy}(\mathcal{A};\widetilde{\mathcal{B}})$ and $\Ch_{\rm acy}(\widetilde{\mathcal{A}};\mathcal{B})$ from an $n$-cotorsion pair $(\mathcal{A,B})$ in $\mathcal{C}$.

\begin{theorem}\label{theo:induced1}
Let $\mathcal{A}$ and $\mathcal{B}$ be two classes of objects in an abelian category $\C$ with enough injectives, such that $\Ext_{\C}^{1}(\mathcal{A,B}) = 0$ and $\mathcal{B}$ is closed under extensions and contains the injectives of $\mathcal{C}$. If $(\Ch_{\rm acy}(\mathcal{A};\widetilde{\mathcal{B}}),\widetilde{\mathcal{B}})$ or $(\widetilde{\mathcal{A}},\Ch_{\rm acy}(\widetilde{\mathcal{A}};\mathcal{B}))$ is a left $n$-cotorsion pair in $\Ch(\mathcal{C})$, then $(\mathcal{A,B})$ is a left $n$-cotorsion pair in $\mathcal{C}$. 
\end{theorem}

\begin{proof} \
Suppose first that $(\Ch_{\rm acy}(\mathcal{A};\widetilde{\mathcal{B}}),\widetilde{\mathcal{B}})$ is a left $n$-cotorsion pair in $\Ch(\mathcal{C})$. By Proposition~\ref{prop:cotorsion_vs_ncotorsion}, it suffices to show that $(\mathcal{A},\mathcal{B}^\wedge_{n-1})$ is a complete left cotorsion pair in $\mathcal{C}$. 

We can apply Proposition~\ref{prop:cotorsion_vs_ncotorsion} in the setting of $\Ch(\mathcal{C})$, noticing that $\Ch(\mathcal{C})$ has enough injectives and that $\widetilde{\mathcal{B}}$ contains the injective complexes. Thus, $(\Ch_{\rm acy}(\mathcal{A};\widetilde{\mathcal{B}}),\widetilde{\mathcal{B}}^\wedge_{n-1})$ is a complete left cotorsion pair in $\Ch(\mathcal{C})$. In particular, the class $\Ch_{\rm acy}(\mathcal{A};\widetilde{\mathcal{B}})$ is closed under extensions, and so is $\mathcal{A}$ by Remark~\ref{sumB'} (5). 
\begin{itemize}
\item[(i)] We first show that for every $C \in \mathcal{C}$, we can construct a short exact sequence
\[
0 \to N \to A \to C \to 0
\]
with $A \in \mathcal{A}$ and $N \in \mathcal{B}^\wedge_{n-1}$. For the complex $S^0(C)$, we can find a short exact sequence
\[
0 \to Y \to X \to S^0(C) \to 0
\]
where $X \in \Ch_{\rm acy}(\mathcal{A};\widetilde{\mathcal{B}})$ and $Y \in \widetilde{\mathcal{B}}^\wedge_{n-1}$, since $(\Ch_{\rm acy}(\mathcal{A};\widetilde{\mathcal{B}}),\widetilde{\mathcal{B}}^\wedge_{n-1})$ is a complete left cotorsion pair in $\Ch(\mathcal{C})$. Then, it suffices to take $A = X_0$ and $N = Y_0$. Note that $N \in \mathcal{B}^\wedge_{n-1}$ since $\mathcal{B}$ is closed under extensions. 

\item[(ii)] Now we prove the equality $\mathcal{A} = {}^{\perp_1}\mathcal{B}^\wedge_{n-1}$. So let $A \in \mathcal{A}$ and $N \in \mathcal{B}^\wedge_{n-1}$. Consider the complexes $S^0(A)$ and $D^1(N)$. We have that $S^0(A) \in \Ch_{\rm acy}(\mathcal{A};\widetilde{\mathcal{B}})$ and $D^1(N) \in \widetilde{\mathcal{B}}^\wedge_{n-1}$ by Remark~\ref{sumB'}. This, along with \cite[Lemma 4.2]{GillespieDegreewise}, yields that
\[
\Ext^1_{\C}(A,N) = \Ext^1_{\C}(A,Z_0(D^1(N))) \cong \Ext^1_{\Ch}(S^0(A),D^1(N)) = 0.
\]
Then, the containment $\mathcal{A} \subseteq {}^{\perp_1}\mathcal{B}^\wedge_{n-1}$ follows. On the other hand, for every $M \in {}^{\perp_1}\mathcal{B}^\wedge_{n-1}$ there is an exact sequence
\[
0 \to N \to A \to M \to 0
\]
with $A \in \mathcal{A}$ and $N \in \mathcal{B}^\wedge_{n-1}$. We thus have that $M$ is a direct summand of $A$, since $\Ext^1_{\mathcal{C}}(M,N) = 0$. It follows that $M \in \mathcal{A}$. 
\end{itemize}
Hence, (i) and (ii) show that $(\mathcal{A},\mathcal{B}^\wedge_{n-1})$ is a complete left cotorsion pair in $\mathcal{C}$. The same conclusion can be reached if we assume that $(\widetilde{\mathcal{A}},\Ch_{\rm acy}(\widetilde{\mathcal{A}};\mathcal{B}))$ is a left $n$-cotorsion pair instead. 
\end{proof}

The converse of the previous result holds with some extra assumptions on the pair $(\mathcal{A,B})$ and the ground category $\mathcal{C}$. We will also need the following lemma, which follows after a careful revision of Yang and Ding's \cite[Lemma 2.1]{YangDingQuestion}.

\begin{lemma}
Let $(\mathcal{A,B})$ be a pair of classes of objects in a bicomplete abelian category $\mathcal{C}$ with enough $\mathcal{A}$-objects such that $\mathcal{B}$ is closed under extensions and satisfying $\Ext^1_{\mathcal{C}}(\mathcal{A,B}) = 0$. Then, for every complex $X \in \Ch(\mathcal{C})$, there exists a short exact sequence
\[
0 \to X \to E \to A \to 0
\]
with $E$ exact and $A \in \Ch_{\rm acy}(\mathcal{A};\widetilde{\mathcal{B}})$.
\end{lemma}

\begin{theorem}\label{theo:induced_n-cotorsion_Ch}
Let $\mathcal{A}$ and $\mathcal{B}$ be two classes of objects in a bicomplete abelian category $\C$ such that $\Ext_{\C}^{1}(\mathcal{A,B}) = 0$ and $\mathcal{B}$ closed under extensions. 
\begin{enumerate}
\item If $(\mathcal{A,B})$ is a hereditary left $1$-cotorsion pair in $\mathcal{C}$, then $(\Ch_{\rm acy}(\mathcal{A};\widetilde{\mathcal{B}}),\widetilde{\mathcal{B}})$ is a left $1$-cotorsion pair in $\Ch(\mathcal{C})$. For $n \geq 2$ in the case where $\mathcal{C}$ has enough injectives and $\mathcal{B}$ contains the injectives of $\mathcal{C}$, if $(\mathcal{A,B})$ is a hereditary left $n$-cotorsion pair in $\mathcal{C}$ and $\Ext^1_{\mathcal{C}}(\mathcal{B},\mathcal{B}^\wedge_{n-1}) = 0$, then $(\Ch_{\rm acy}(\mathcal{A};\widetilde{\mathcal{B}}),\widetilde{\mathcal{B}})$ is a left $n$-cotorsion pair in $\Ch(\mathcal{C})$.

\item Suppose $\mathcal{C}$ has enough injectives. If $(\mathcal{A,B})$ is a hereditary left $1$-cotorsion pair in $\mathcal{C}$, then $(\widetilde{\mathcal{A}}, \Ch_{\rm acy}(\widetilde{\mathcal{A}};\mathcal{B}))$ is a left $1$-cotorsion pair in $\Ch(\mathcal{C})$. For $n \geq 2$, if $\mathcal{B}$ contains the injectives of $\mathcal{C}$ and $(\mathcal{A,B})$ is a hereditary left $n$-cotorsion pair in $\mathcal{C}$ with $\Ext^1_{\mathcal{C}}(\mathcal{B},\mathcal{B}^\wedge_{n-1}) = 0$, then $(\widetilde{\mathcal{A}},\Ch(\mathcal{B}))$ is a left $n$-cotorsion pair in $\Ch(\mathcal{C})$.
\end{enumerate}
\end{theorem}

\begin{proof} 
We focus in the case where $n \geq 2$. The remaining case $n = 1$ follows in the same way. 

We first show part (1). Again, as in the proof of Theorem~\ref{theo:induced1}, the idea is to use Proposition~\ref{prop:cotorsion_vs_ncotorsion} and show that $(\Ch_{\rm acy}(\mathcal{A};\widetilde{\mathcal{B}}),\widetilde{\mathcal{B}}^\wedge_{n-1})$ is a complete left cotorsion pair in $\Ch(\mathcal{C})$.
\begin{itemize}
\item[(i)] First, for every complex $X \in \Ch(\mathcal{C})$ we construct a short exact sequence 
\[
0 \to B \to A \to X \to 0
\]
with $A \in \Ch_{\rm acy}(\mathcal{A};\widetilde{\mathcal{B}})$ and $B \in \widetilde{\mathcal{B}}^\wedge_{n-1}$. It suffices to prove the case where $X$ is an exact complex. Indeed, the general case follows by using the previous lemma and a standard pullback argument, along with the fact that $\Ch_{\rm acy}(\mathcal{A};\widetilde{\mathcal{B}})$ is resolving since $(\mathcal{A,B})$ is hereditary. Thus, for each $m \in \mathbb{Z}$ we have a short exact sequence 
\[
0 \to Z_m(X) \to X_m \to Z_{m-1}(X) \to 0.
\]
Since $(\mathcal{A,B})$ is a left $n$-cotorsion pair, each $Z_m(X)$ has a $\mathcal{A}$-precover with kernel in $\mathcal{B}^\wedge_{n-1}$. By Theorem~\ref{theo:compatible}, from these $\mathcal{A}$-precovers we can construct an exact sequence
\[
0 \to B^{n-1}_m \to B^{n-2}_m \to \cdots \to B^1_m \to B^0_m \to A_m \to X_m \to 0
\]
compatible with them, such that $A_m \in \mathcal{A}$ and $B_k \in \mathcal{B}$ for every $0 \leq k \leq n-1$. Thus, for each $m \in \mathbb{Z}$ we have a commutative diagram as in \eqref{fig_compatible}. Connecting these digrams yields an exact sequence
\[
0 \to B^{n-1} \to B^{n-1} \to \cdots \to B^1 \to B^0 \to A \to X \to 0
\]
in $\Ch(\mathcal{C})$ with $A \in \widetilde{\mathcal{A}} \subseteq \Ch_{\rm acy}(\mathcal{A};\widetilde{\mathcal{B}})$ (see Lemma~\ref{lem:null_homotopic}) and $B^k \in \widetilde{\mathcal{B}}$ for every $0 \leq k \leq n-1$. 

\item[(ii)] We now show the equality $\Ch_{\rm acy}(\mathcal{A};\widetilde{\mathcal{B}}) = {}^{\perp_1}(\widetilde{\mathcal{B}}^\wedge_{n-1})$. Note that $\Ext^1_{\mathcal{C}}(\mathcal{A},\mathcal{B}^\wedge_{n-1}) = 0$, then the containment $\Ch_{\rm acy}(\mathcal{A};\widetilde{\mathcal{B}}) \subseteq {}^{\perp_1}(\widetilde{\mathcal{B}}^\wedge_{n-1})$ follows by Lemma \ref{B subseteq dgA}. The converse inclusion follows after using part (i) and noticing that if $\mathcal{A}$ is closed under direct summands, then so is $\Ch_{\rm acy}(\mathcal{A};\widetilde{\mathcal{B}})$. 
\end{itemize}
Therefore, $(\Ch_{\rm acy}(\mathcal{A};\widetilde{\mathcal{B}}),\widetilde{\mathcal{B}}^\wedge_{n-1})$ is a complete left cotorsion pair in $\Ch(\mathcal{C})$. For the proof concerning the pair $(\widetilde{\mathcal{A}},\Ch(\mathcal{B}))$, we only show that every complex is the epimorphic image of an objet in $\widetilde{\mathcal{A}}$ with kernel in $\Ch(\mathcal{B})^\wedge_{n-1}$. So let $X \in \Ch(\mathcal{C})$ be a complex. Since $\mathcal{C}$ has enough injectives, there exists a short exact sequence
\[
0 \to I \to E \to X \to 0
\]
where $E$ is exact and $I$ is a differential graded injective complex (see \cite[Lemma 2.1]{YangDingQuestion}). On the other hand, since $E$ is exact, there is a short exact sequence
\[
0 \to K^0 \to A \to E \to 0
\] 
where $A \in \widetilde{\mathcal{A}}$ and $K^0 \in \widetilde{\mathcal{B}}^\wedge_{n-1}$. Taking the pullback of $I \to E \leftarrow A$ gives rise to two short exact sequences
\begin{align*}
0 & \to K^0 \to K \to I \to 0, \\
0 & \to K \to A \to X \to 0.
\end{align*}
Note that $I \in \Ch_{\rm acy}(\widetilde{\mathcal{A}};\mathcal{B})$ and $K^0 \in \Ch_{\rm acy}(\widetilde{\mathcal{A}};\mathcal{B})^\wedge_{n-1}$. By Lemma~\ref{Bk-1ext}, we have that $\mathcal{B}^\wedge_{n-1}$ is closed under extensions, and so $K \in \Ch(\mathcal{B})^\wedge_{n-1}$. For the case $n = 1$, it is not hard to see that also $K \in \Ch_{\rm acy}(\widetilde{\mathcal{A}};\mathcal{B})$. 
\end{proof}

\begin{remark} \
\begin{enumerate}
\item From the previous theorem, we have hereditary left $1$-cotorsion pairs 
\[
(\Ch(\mathcal{P}(R);\widetilde{\mathsf{Mod}(R)}),\widetilde{\mathsf{Mod}(R)}) \mbox{ \ and \ } (\widetilde{\mathcal{P}(R)},\Ch(\widetilde{\mathcal{P}(R)};\mathsf{Mod}(R)))
\]
where $\widetilde{\mathsf{Mod}(R)}$ is the class of exact complexes, $\widetilde{\mathcal{P}(R)}$ coincides with the class of projective complexes (which we denote by $\mathscr{P}(R)$), and $\Ch(\mathcal{P}(R);\widetilde{\mathsf{Mod}(R)})$ is the class of differential graded projective chain complexes. Note also that $\mathscr{P}(R)$ is part of another left $1$-cotorsion pair $(\mathscr{P}(R),\Ch(R))$. 

\item Let $R$ be a Gorenstein ring which is not a QF ring. The results mentioned in Section \ref{sec:applications} for Gorenstein modules also hold in the context of $\Ch(R)$. So, if $\mathscr{GP}(R)$ denotes the class of Gorenstein projective chain complexes, we have that $(\mathscr{GP}(R),\mathscr{P}(R))$ is a left $n$-cotorsion pair in $\Ch(R)$. On the other hand, we know from \cite[Theorem 2.2]{GPcomplexes} that $\mathscr{GP}(R)$ is the class of complexes of Gorenstein projective modules, that is, $\mathscr{GP}(R) = \Ch(\mathcal{GP}(R))$. It follows that we have a left $n$-cotorsion pair in $\Ch(R)$ of the form $(\Ch(\mathcal{GP}(R)),\widetilde{\mathcal{P}(R)})$. Thus, the condition that $\mathcal{B}$ contains the injective objects required in Theorem \ref{theo:induced_n-cotorsion_Ch} is sufficient but not necessary. 

Similar observations are also true for the classes of Gorenstein injective and injective complexes. Moreover, by \cite[Theorem 3.11]{GFcomplexes} and \cite[Corollary 3.12]{SarochStovicek}, we have that a chain complex $X \in \Ch(R)$ is Gorenstein flat if, and only if, each $X_m$ is a Gorenstein flat module. Thus, the results in Section \ref{sec:applications} for the classes $\mathcal{GF}(R)$, $\mathcal{GI}(R)$ and $\mathcal{F}(R)$ carry over to the classes $\Ch(\mathcal{F}(R))$, $\Ch(\mathcal{I}(R))$ and $\widetilde{\mathcal{F}(R)}$ of Gorenstein flat, Gorenstein injective and flat complexes, respectively. 
\end{enumerate}
\end{remark}


\section*{\textbf{Funding}}

The authors thank Project PAPIIT-Universidad Nacional Aut\'onoma de M\'exico IN103317. The third named author is supported by a Comisi\'on Acad\'emica de Posgrado de la Universidad de la Rep\'ublica (CAP - UdelaR) posdoctoral fellowship.


\bibliographystyle{alpha}
\bibliography{bibliohmp}

\begin{thebibliography}{{Gar}99}

\bibitem[AA02]{Akinci}
K.~D. Akinci and R.~Alizade.
\newblock Special precovers in cotorsion theories.
\newblock {\em Proc. Edinb. Math. Soc. (2)}, 45(2):411--420, 2002.

\bibitem[AB89]{ABtheory}
M.~Auslander and {R.-O.} Buchweitz.
\newblock The homological theory of maximal {C}ohen-{M}acaulay approximations.
\newblock {\em M\'em. Soc. Math. France (N.S.)}, (38):5--37, 1989.
\newblock Colloque en l'honneur de Pierre Samuel (Orsay, 1987).

\bibitem[AM93]{AsensioMartinez}
J.~{Asensio Mayor} and J.~{Mart{\'\i}nez Hern\'andez}.
\newblock On flat and projective envelopes.
\newblock {\em J. Algebra}, 160(2):434--440, 1993.

\bibitem[Ben09]{Bennis}
D.~Bennis.
\newblock Rings over which the class of {G}orenstein flat modules is closed
  under extensions.
\newblock {\em Comm. Algebra}, 37(3):855--868, 2009.

\bibitem[Bla11]{Bland}
P.~E. Bland.
\newblock {\em Rings and their modules.}
\newblock Berlin: Walter de Gruyter, 2011.

\bibitem[BM10]{BMglobal}
D.~Bennis and N.~Mahdou.
\newblock Global {G}orenstein dimensions.
\newblock {\em Proc. Am. Math. Soc.}, 138(2):461--465, 2010.

\bibitem[BMPS]{BMPS}
V.~Becerril, O.~Mendoza, {M. A.} P\'erez, and V.~Santiago.
\newblock Frobenius pairs in abelian categories: {C}orrespondences with
  cotorsion pairs, exact model categories, and {A}uslander-{B}uchweitz
  contexts.
\newblock {\em J. Homotopy Relat. Struct.}, (In press).

\bibitem[BMS]{BecerrilMendozaSantiago}
V.~Becerrill, O.~Mendoza, and V.~Santiago.
\newblock Relative {G}orenstein {O}bjects in {A}belian {C}ategories.
\newblock {\em arXiv}, 1810.08524v1.

\bibitem[BR07a]{BR}
A.~Beligiannis and I.~Reiten.
\newblock {\em {Homological and Homotopical Aspects of Torsion Theories}},
  volume 883 of {\em Memoirs of the American Mathematical Society}.
\newblock American Mathematical Society, 2007.

\bibitem[BR07b]{BR07}
A.~Beligiannis and I.~Reiten.
\newblock Homological and {H}omotopical {A}spects of {T}orsion {T}heories.
\newblock {\em Mem. Amer. Math. Soc.}, 188(883):viii+207, 2007.

\bibitem[Col75]{Colby}
{R. R.} Colby.
\newblock Rings which have flat injective modules.
\newblock {\em J. Algebra}, 35:239--252, 1975.

\bibitem[CT08]{CriveiTorrecillas}
S.~Crivei and B.~Torrecillas.
\newblock On some monic covers and epic envelopes.
\newblock {\em Arab. J. Sci. Eng. Sect. C Theme Issues}, 33(2):123--135, 2008.

\bibitem[DC93]{DingChen93}
N.~Ding and J.~Chen.
\newblock The flat dimensions of injective modules.
\newblock {\em Manuscripta Math.}, 78(2):165--177, 1993.

\bibitem[DC96]{DingChen96}
N.~Ding and J.~Chen.
\newblock Coherent rings with finite self-{$FP$}-injective dimension.
\newblock {\em Comm. Algebra}, 24(9):2963--2980, 1996.

\bibitem[Din96]{Ding96}
N.~Ding.
\newblock On envelopes with the unique mapping property.
\newblock {\em Comm. Algebra}, 24(4):1459--1470, 1996.

\bibitem[DLM09]{DLM}
N.~Ding, Y.~Li, and L.~Mao.
\newblock Strongly {G}orenstein flat modules.
\newblock {\em J. Aust. Math. Soc.}, 86(3):323--338, 2009.

\bibitem[EJ00]{EJ}
E.~E. Enochs and O.~M.~G. Jenda.
\newblock {\em Relative Homological Algebra}, volume~30 of {\em De Gruyter
  Expositions in Mathematics}.
\newblock Walter de Gruyter \& Co., Berlin, 2000.

\bibitem[Fie71]{Fieldhouse}
{D. J.} Fieldhouse.
\newblock Character modules.
\newblock {\em Comment. Math. Helv.}, 46:274--276, 1971.

\bibitem[{Gar}99]{JRGR}
{J. R.} {Garc\'{\i}a Rozas}.
\newblock {\em Covers and envelopes in the category of complexes of modules.}
\newblock Boca Raton, FL: Chapman and Hall/CRC, 1999.

\bibitem[Gil04]{GillespieFlat}
J.~Gillespie.
\newblock The flat model structure on {${\rm Ch}(R)$}.
\newblock {\em Trans. Amer. Math. Soc.}, 356(8):3369--3390, 2004.

\bibitem[Gil08]{GillespieDegreewise}
J.~Gillespie.
\newblock Cotorsion pairs and degreewise homological model structures.
\newblock {\em Homology Homotopy Appl.}, 10(1):283--304, 2008.

\bibitem[Gil10]{GillespieDing}
J.~Gillespie.
\newblock Model structures on modules over {D}ing-{C}hen rings.
\newblock {\em Homology Homotopy Appl.}, 12(1):61--73, 2010.

\bibitem[HJ09]{HJduality}
H.~Holm and P.~J{\o}rgensen.
\newblock Cotorsion pairs induced by duality pairs.
\newblock {\em J. Commut. Algebra}, 1(4):621--633, 2009.

\bibitem[Hol04]{Holm}
H.~Holm.
\newblock Gorenstein homological dimensions.
\newblock {\em J. Pure Appl. Algebra}, 189(1-3):167--193, 2004.

\bibitem[Hov02]{Hovey}
M.~Hovey.
\newblock Cotorsion pairs, model category structures, and representation
  theory.
\newblock {\em Math. Z.}, 241(3):553--592, 2002.

\bibitem[Iac16]{IacobAuslandercondition}
A.~Iacob.
\newblock Gorenstein injective covers and envelopes over rings that satisfy the
  {A}uslander condition.
\newblock {\em Acta Math. Univ. Comenian. (N.S.)}, 85(1):165--172, 2016.

\bibitem[Iya11]{IyamaCluster}
O.~Iyama.
\newblock Cluster tilting for higher {A}uslander algebras.
\newblock {\em Adv. Math.}, 226(1):1--61, 2011.

\bibitem[Mao18]{MaoPiCoherent}
L.~Mao.
\newblock Rings satisfying every finitely generated module has a {G}orenstein
  projective (pre)envelope.
\newblock {\em Communications in Algebra}, 46(5):2010--2022, 2018.

\bibitem[MD07a]{MaoDing_finite}
L.~Mao and N.~Ding.
\newblock Envelopes and covers by modules of finite {FP}-injective and flat
  dimensions.
\newblock {\em Comm. Algebra}, 35(3):833--849, 2007.

\bibitem[MD07b]{MaoDing_wgd}
L.~Mao and N.~Ding.
\newblock Weak global dimension of coherent rings.
\newblock {\em Comm. Algebra}, 35(12):4319--4327, 2007.

\bibitem[MP11]{MengPan}
F.~Meng and Q.~Pan.
\newblock {$\mathscr{X}$}-{G}orenstein projective and
  {$\mathscr{Y}$}-{G}orenstein injective modules.
\newblock {\em Hacet. J. Math. Stat.}, 40(4):537--554, 2011.

\bibitem[MT11]{MahdouTamekkante}
N.~Mahdou and M.~Tamekkante.
\newblock {Strongly Gorenstein flat modules and dimensions.}
\newblock {\em {Chin. Ann. Math., Ser. B}}, 32(4):533--548, 2011.

\bibitem[Pin08]{Pinzon}
Katherine Pinzon.
\newblock Absolutely pure covers.
\newblock {\em Comm. Algebra}, 36(6):2186--2194, 2008.

\bibitem[Sal79]{Salce}
L.~Salce.
\newblock Cotorsion theories for abelian groups.
\newblock In {\em Symposia {M}athematica, {V}ol. {XXIII} ({C}onf. {A}belian
  {G}roups and their {R}elationship to the {T}heory of {M}odules, {INDAM},
  {R}ome, 1977)}, pages 11--32. Academic Press, London-New York, 1979.

\bibitem[Sie10]{Sieg}
D.~Sieg.
\newblock {\em A {H}omological {A}pproach to the {S}plitting {T}heory of
  {PLS}-spaces}.
\newblock PhD thesis, Universit\"at Trier, Universit\"atsring 15, 54296 Trier,
  2010.

\bibitem[{\v{S}}{\v{S}}]{SarochStovicek}
J.~{\v{S}}aroch and J.~{\v{S}}\v{t}ov\'{\i}\v{c}ek.
\newblock Singular compactness for {$\Sigma$}-cotorsion and projectively
  resolved {G}orenstein flat modules.
\newblock {\em arXiv}, 1804.09080v2.

\bibitem[Wan17]{Wang}
J.~Wang.
\newblock Ding projective dimension of {G}orenstein flat modules.
\newblock {\em Bull. Korean Math. Soc.}, 54(6):1935--1950, 2017.

\bibitem[Yan12]{Yang}
G.~Yang.
\newblock Homological properties of modules over {D}ing-{C}heng rings.
\newblock {\em J. Korean Math. Soc.}, 49(1):31--47, 2012.

\bibitem[YD15]{YangDingQuestion}
X.~Yang and N.~Ding.
\newblock On a question of gillespie.
\newblock {\em Forum Math.}, 27(6):3205--3231, 2015.

\bibitem[YL11]{GPcomplexes}
Xiaoyan {Yang} and Zhongkui {Liu}.
\newblock {Gorenstein projective, injective, and flat complexes.}
\newblock {\em {Commun. Algebra}}, 39(5):1705--1721, 2011.

\bibitem[YL12]{GFcomplexes}
Gang {Yang} and Zhongkui {Liu}.
\newblock {Stability of {G}orenstein flat categories.}
\newblock {\em {Glasg. Math. J.}}, 54(1):177--191, 2012.

\bibitem[YLL13]{DPcomplexes}
Gang {Yang}, Zhongkui {Liu}, and Li~{Liang}.
\newblock {Model structures on categories of complexes over {D}ing-{C}hen
  rings.}
\newblock {\em {Commun. Algebra}}, 41(1):50--69, 2013.

\end{thebibliography}
\end{document}